\documentclass[11pt,a4paper]{article}
\usepackage[margin=25mm]{geometry}
\usepackage[numbers,sort&compress]{natbib}
\usepackage{microtype}
\usepackage[T1]{fontenc}
\usepackage{lmodern}

\usepackage{url}

\usepackage{amsfonts}
\usepackage{amssymb}
\usepackage{amsmath}
\usepackage{amsthm}
\usepackage{mathrsfs}

\usepackage{array}
\usepackage{multirow}
\usepackage{graphicx}
\usepackage{xcolor}

\usepackage[utf8]{inputenc}

\usepackage[title]{appendix}
\usepackage{textcomp}
\usepackage{manyfoot}
\usepackage{booktabs}

\usepackage{subcaption}

\usepackage{pgfplots}
\pgfplotsset{compat=1.18}
\usepgfplotslibrary{groupplots}

\IfFileExists{algorithm.sty}{\usepackage{algorithm}}{%
	\usepackage{float}
	\floatstyle{ruled}
	\newfloat{algorithm}{htbp}{loa}
	\floatname{algorithm}{Algorithm}
}
\newif\ifPDQNnativealgorithmic
\IfFileExists{algorithmic.sty}{\PDQNnativealgorithmictrue}{\PDQNnativealgorithmicfalse}
\ifPDQNnativealgorithmic
\usepackage{algorithmic}
\else

\fi

\usepackage{listings}
\lstdefinestyle{statplotappendix}{
	language=Python,
	basicstyle=\ttfamily\scriptsize,
	keywordstyle=\bfseries,
	commentstyle=\itshape\color{black!65},
	stringstyle=\color{black},
	numbers=left,
	numberstyle=\tiny\color{black!65},
	numbersep=6pt,
	stepnumber=1,
	frame=single,
	rulecolor=\color{black!40},
	framesep=3pt,
	xleftmargin=2.1em,
	xrightmargin=0.3em,
	breaklines=true,
	breakatwhitespace=false,
	columns=fullflexible,
	keepspaces=true,
	showstringspaces=false,
	tabsize=4,
	captionpos=b,
	aboveskip=8pt,
	belowskip=8pt
}

\usepackage{hyperref}
\hypersetup{hidelinks,hypertexnames=false}
\usepackage{ragged2e}

\def\bfi#1{\textbf{#1}} 

\newtheoremstyle{thmstyleone}{6pt}{6pt}{\itshape}{}{\bfseries}{.}{.5em}{}
\newtheoremstyle{thmstyletwo}{6pt}{6pt}{}{}{\bfseries}{.}{.5em}{}
\newtheoremstyle{thmstylethree}{6pt}{6pt}{}{}{\bfseries}{.}{.5em}{}

\theoremstyle{thmstyleone}

\newtheorem{theorem}{Theorem}[section]
\newtheorem{proposition}[theorem]{Proposition}
\newtheorem{lemma}[theorem]{Lemma}
\newtheorem{corollary}[theorem]{Corollary}
\newtheorem{assumption}[theorem]{Assumption}

\theoremstyle{thmstyletwo}

\newtheorem{example}[theorem]{Example}
\newtheorem{remark}[theorem]{Remark}

\theoremstyle{thmstylethree}

\newtheorem{definition}[theorem]{Definition}

\def\Box{{\hbox{\raisebox{0.0em}{\rlap{$\sqcap$}}\kern0em%
			\raisebox{-0.0em}{$\sqcup$}}}}

\def\bfi#1{\textbf{#1}}

\def\eeq{\end{equation}}
\def\lbeq#1{\begin{equation}\label{#1}}

\def\bary{\begin{array}}
	\def\eary{\end{array}}

\def\fct#1{\mathop{\rm #1}}

\def\Argmin{\fct{argmin}}

\def\ealg{\end{alg}}

\newcommand{\PDstep}[2]{\ensuremath{\textbf{S#1}_{\ref{#2}}}}

\usepackage{tikz}
\usetikzlibrary{arrows.meta}

\makeatletter
\newcommand{\eqalias}[2]{\AtEndDocument{\@ifundefined{r@#2}{}{%
		\immediate\write\@auxout{\string\newlabel{#1}{\csname r@#2\endcsname}}}}}
\makeatother
\eqalias{e.modelfunc}{pr:pxy-subproblem}
\eqalias{sol:px}{pr:px-subproblem}
\eqalias{pr:py-subproblem}{eq:prox-center-definition}
\eqalias{eq:zeta-unbounded-definition}{eq:zeta-definition}
\eqalias{eq:C-zeta-definition}{eq:zeta-definition}
\eqalias{eq:reference-armijo}{eq:reference-projection}
\eqalias{eq:reference-error}{eq:reference-mapping}

\begin{document}

\begin{center}
	
	{\Large \bf A safeguarded quasi-Newton penalty decomposition method for }\\
	{\Large \bf minimization over sparse symmetric sets}\\

	\vspace{0.5cm}
	{\large \bf Ahmad Mousavi}
	\centerline{\sl Department of Mathematics and Statistics,}
	\centerline{\sl American University, Washington, DC, USA}
	\centerline{\sl email: mousavi@american.edu}

	{\large \bf Morteza Kimiaei}
	\centerline{\sl Fakult\"at f\"ur Mathematik, Universit\"at Wien}
	\centerline{\sl Oskar--Morgenstern--Platz 1, A--1090 Wien, Austria}
	\centerline{\sl email: morteza.kimiaei@univie.ac.at}

	{\large \bf Saman Babaie--Kafaki}
	\centerline{\sl Faculty of Engineering, Free University of Bozen--Bolzano}
	\centerline{\sl Bolzano 39100, Italy}
	\centerline{\sl email: saman.babaiekafaki@unibz.it}

	{\large \bf Vyacheslav Kungurtsev}
	\centerline{\sl email: slavakung5@gmail.com}
	
\end{center}
\begin{abstract}
	We study smooth, possibly nonconvex minimization over the intersection of a closed convex permutation-symmetric set and a cardinality constraint, where exact block minimization in a penalty-decomposition model need not satisfy outer consensus or true-objective descent. We safeguard finite-budget diagonal quasi-Newton trials with an exact feasible projected-gradient reference and two objective tests; when a trial fails, we take the reference step. If the gradient is Lipschitz on bounded balls and the initial feasible level set is compact, the method is well-defined regardless of trial failures; its reference and accepted displacements are square summable, and every accumulation point satisfies a common-stepsize projection condition, basic feasibility on every super-support, intrinsic Fr\'echet stationarity, and geometric asymptotic Mordukhovich stationarity. Joint regularity of the convex and sparse normal cones yields geometric Mordukhovich stationarity, including for polyhedral sets and convex sets containing an interior origin. A normal witness at the returned reference point gives a universal approximate basic-feasibility inequality and, under the same pointwise regularity, an upper bound on the geometric Mordukhovich residual; a positive target requires $O(\epsilon^{-2})$ reference computations. An exploratory, author-reported study compares nine quasi-Newton configurations with three adapted methods on 192 synthetic instances using a weighted evaluation budget. The postprocessed quasi-Newton outputs reach the reference-witness target more often, although objective and geometric-stationarity rankings differ. The computational artifacts needed to verify these results were not supplied with this manuscript, and a reference-only control is absent.

\end{abstract}
\noindent\textbf{Keywords:} Sparse optimization; penalty decomposition; diagonal quadratic models; geometric stationarity; support-restricted projection.

\section{Introduction}\label{sec:intro}
We consider the cardinality-constrained optimization problem
\begin{equation}\label{pr:p-original}\tag{\texttt{CCOP}}
	\min_{x\in\Omega}f(x),\ \ \Omega:=C\cap C_s\ne\varnothing, \ \ C_s:=\{x\in\mathbb R^n:\|x\|_0\le s\},\ \ s\in\{1,\ldots,n-1\},
\end{equation}
where $f\in C^1(\mathbb R^n)$,
$g:=\nabla f:\mathbb R^n\to\mathbb R^n$, and
$C\subseteq\mathbb R^n$ is closed, convex and invariant under coordinate permutations. Here $\|x\|_0$ counts nonzeros, and $s$ sets a resource or feature budget. Well-known applications include sparse regression and portfolio selection~\cite{Tillmann2024,mousavi2025cardinality}. Choices of $C$ include $\mathbb R^n$, the nonnegative orthant, the standard simplex, boxes with common bounds and permutation-invariant norm balls.

\bfi{Related work.}
Various methods have been developed for~\eqref{pr:p-original}.
Hard-thresholding and coordinate methods preserve sparsity and provide
first-order guarantees~\cite{Blumensath2009,Beck2013,Beck2016,zhao2020optimal},
with probabilistic and composite variants~\cite{bergamaschi2025probabilistic,hu2025convergence}.
Greedy methods serve sparse recovery~\cite{wen2025randomized} and spectral
optimization~\cite{cvetkovic2022greedy}.
Basis pursuit, the lasso and shrinkage offer convex or penalized
alternatives~\cite{Chen1998,tibshirani1996regression,Esmaeili2018}, while 
mixed-integer and complementarity formulations address discrete search
and regularization~\cite{Bertsimas2016,Burdakov2016}.

Limited-memory curvature and line searches supply standard smooth
optimization tools~\cite{LBFGS,LMBOPT,CLS,NocedalWright}. Their extensions
to composite problems replace the purely smooth step by a model that also
accounts for the nonsmooth component. Proximal quasi-Newton methods combine a quadratic model of the smooth
term with a proximal treatment of the nonsmooth term. Leconte and Orban~\cite{Leconte2024} developed trust-region methods with
possibly indefinite diagonal quasi-Newton models.
Diouane et al.~\cite{Dominique2024} used adaptive quadratic regularization
and also developed a diagonal variant admitting explicit steps for
selected separable nonsmooth terms.

Penalty decomposition separates smooth and sparse blocks~\cite{lu2013sparse},
with inexact, geometric and discrete-neighborhood
extensions~\cite{Lapucci2021,KanzowLapucci,Lapucci2023}.
Applications include adaptive sparsity~\cite{DongZhu},
rank minimization~\cite{LuZhangLi},
multiobjective problems~\cite{Lapucci},
wavelet restoration~\cite{ZhangDongLu},
trend filtering~\cite{PatrascuNecoara},
nuclear-norm models~\cite{WangJinShang},
and relaxed subproblems~\cite{Spiridonovetal}.

\bfi{Obstacle and construction.} Diagonal quasi-Newton models make smooth-block updates explicit. Yet exact block minimization of a frozen model may leave persistent disagreement. Even agreement and a small final residual can conflict with sufficient decrease in $f$ (Examples~\ref{app:nonstationary-consensus} and~\ref{app:incompatible-filter}, below). More inner work alone cannot guarantee outer progress. We safeguard finite-budget trials with an exact feasible projected-gradient reference on $\Omega$. Two objective comparisons preserve reference decrease and control the accepted displacement. On failure, we take the reference step and reset the multiplier. Outer progress is independent of trial success.

\bfi{Contributions.}
Building on classical projection-stationarity results~\cite{Beck2013,Beck2016},
we combine a finite-budget quasi-Newton trial generator with a feasible reference step and derive a stationarity certificate for the returned point. The convergence and finite-stationarity guarantees hold even when the trial fails, because the method then uses the feasible reference step.

We use geometric asymptotic Mordukhovich (AM) and Mordukhovich (M)
stationarity~\cite{KanzowLapucci,Mordukhovich2018}.
These notions differ from the cardinality-specific CC-AM/CC-M
framework of~\cite{kanzow2021} and from sequential conditions based
on lifted cardinality reformulations such as those in~\cite{Ribeiro2022}. Under the stated assumptions, the method remains well-defined
despite arbitrary trial failures and satisfies coupled descent
(Theorem~\ref{thm:certified-descent}).
Every accumulation point of an infinite run satisfies a
common-stepsize projection condition
(Theorem~\ref{thm:certified-stationarity}).
The universal basic feasibility (BF) conclusion also yields all-super-support
S-stationarity under a supportwise normal-cone sum rule
(Theorem~\ref{thm:all-support-S-BF} and
Corollary~\ref{cor:certified-all-support-S}).
At a stationarity-target exit, the returned reference satisfies
approximate basic feasibility; joint AM-regularity there
additionally bounds its geometric M-residual
(Proposition~\ref{prop:reference-certificate}).
In the absence of an external interruption or essential-oracle failure,
a positive target $\epsilon_{\rm stat}$ is reached within
$O(\epsilon_{\rm stat}^{-2})$ reference computations.
The known reduction uses at most $s+1$ convex slice solves
per sparse reference projection~\cite{Beck2016}
(Corollary~\ref{cor:certified-complexity}).
We also provide a conditional analysis of the quasi-Newton penalty decomposition (\texttt{QNPD}) core
without reference recovery and an extension to square-summable
errors in frozen smooth-block solves
(Appendix~\ref{sec:inexact-extension}).

We assume Lipschitz continuity of $g$ on every bounded ball,
a compact initial feasible level set, and exact terminating
essential projection, function, and gradient oracles.
Optional routines terminate within finite budgets;
zero trial budgets are allowed. The work is organized as follows: 
Sections~\ref{preliminaries}--\ref{convergence} present
the geometric framework, algorithms, and convergence analysis, respectively, while Section~\ref{numerica} reports an exploratory comparison and discusses its limitations.

\section{Preliminaries and stationarity}\label{preliminaries}

This section collects the notation and variational tools needed below. We also give an example separating AM- and M-stationarity.

\subsection{Notation and normals}\label{sec:project}
Write $[n]:=\{1,\ldots,n\}$, $I_1(x):=\{i:x_i\ne0\}$, $I_0(x):=[n]\setminus I_1(x)$,
$S_{\mathcal L}:=\{x:x_{\mathcal L^c}=0\}$, $x_{\mathcal L}:=U_{\mathcal L}^Tx$ and $C_{\mathcal L}:=\{v\in\mathbb R^{|\mathcal L|}:U_{\mathcal L}v\in C\}$, where $U_{\mathcal L}$ contains the identity columns indexed increasingly by $\mathcal L\subseteq[n]$. 
For an index set $\mathcal L\subseteq[n]$,
we write
$S_{\mathcal L}^{\perp}
:=\{v\in\mathbb R^n:v_i=0\ \text{for all }i\in\mathcal L\}$
for the orthogonal complement of $S_{\mathcal L}$. A super-support satisfies $I_1(x)\subseteq\mathcal L$ and $|\mathcal L|=s$. 
For a set $C\subseteq\mathbb R^n$, we denote by
$\operatorname{int}C$ its interior in $\mathbb R^n$.
Norms are Euclidean unless specified; $B_r(z)$ and $\overline B_r(z)$ are open and closed balls ($r>0$), and $B_j(R):=\{x:\|x\|_j\le R\}$ for $j\in\{1,2\}$ and $R>0$. Independent sign symmetry means $(\epsilon_i x_i)_{i=1}^n\in C$ for all $x\in C$ and $\epsilon\in\{-1,1\}^n$~\cite[Definition~3.5]{Beck2016}. For nonempty closed $D$, $P_D(z):=\Argmin_{v\in D}\|v-z\|^2$ is nonempty and is a singleton for convex $D$. Sparse projections may be multivalued.
A feasible $\bar x$ is a local (global) minimizer if it minimizes $f$ on $\Omega\cap B_\delta(\bar x)$ for some $\delta>0$ (on all of $\Omega$).

For closed $D$ and $x\in D$, use the Fr\'echet and limiting normals~\cite{Mordukhovich2006}
\[
N_D^F(x):=\left\{v:\limsup_{\substack{y\to x,\,y\in D\\y\ne x}}
\frac{v^T(y-x)}{\|y-x\|}\le0\right\},\qquad
N_D(x):=\limsup_{y\to x,\,y\in D}N_D^F(y).
\]
Set-valued $\limsup$ denotes the sequential outer limit. At isolated points, $N_D^F(x)=\mathbb R^n$; both normals are empty outside $D$. For closed convex $C$ and $x\in C$, $N_C^F(x)=N_C(x)=\{u:u^T(y-x)\le0\ (y\in C)\}$. The graph of this normal mapping is closed. For $x\in C_s$,
\begin{equation}\label{eq:sparse-normal-compact}
	N_{C_s}(x):=\bigcup_{\substack{I_1(x)\subseteq\mathcal L\\|\mathcal L|=s}}S_{\mathcal L}^{\perp},
	\qquad N_{C_s}^F(x):=\begin{cases}S_{I_1(x)}^\perp,&\|x\|_0=s,\\\{0\},&\|x\|_0<s.\end{cases}
\end{equation}

When $\|x\|_0=s$, the support is locally fixed, so the
Fr\'echet and limiting normal cones coincide. If
$\|x\|_0<s$, sufficiently small perturbations in either
direction along each coordinate are feasible, implying
$N_{C_s}^{F}(x)=\{0\}$. Nevertheless, the limiting normal
cone may contain nonzero vectors. 
Indeed, for any nonzero $v\in N_{C_s}(x)$, there exist
$x^k\to x$ and $v^k\to v$ with
$v^k\in N_{C_s}^F(x^k)$. Since $v^k\ne0$ for all
sufficiently large $k$, the points $x^k$ have full
support. Passing to a subsequence, their supports coincide
with a fixed super-support $\mathcal L$ of $x$, yielding
$v\in S_{\mathcal L}^{\perp}$. The zero vector belongs
to every such orthogonal complement. Conversely, for each super-support $\mathcal L$
of $x$, filling its missing coordinates with vanishing
nonzero entries yields a sequence converging to $x$ whose
Fr\'echet normal cone is $S_{\mathcal L}^{\perp}$.

\subsection{Stationarity and regularity}\label{sec:theorem5.6}
We use the identity-map specialization of Kanzow--Lapucci's geometric conditions~\cite[Definitions~2.1--2.3 and Theorem~2.4]{KanzowLapucci}. A feasible $\bar x$ is M-stationary if
\begin{equation}\label{eq:geometric-M}
	0\in g(\bar x)+N_C(\bar x)+N_{C_s}(\bar x),
\end{equation}
and AM-stationary if there exist $a^k\in C$, $b^k\in C_s$, $a^k,b^k\to\bar x$, and
\begin{equation}\label{eq:geometric-AM}
	u^k\in N_C(a^k),\quad v^k\in N_{C_s}(b^k),\quad g(a^k)+u^k+v^k\to0.
\end{equation}
The two points may differ and the normals may diverge. Continuity of $g$ allows evaluation at either point. These geometric notions differ from the cardinality-specific CC-AM/CC-M conditions of~\cite{kanzow2021} and from lifted sequential conditions such as AW-stationarity~\cite{Ribeiro2022}. At feasible $\bar x$, joint AM-regularity of $\mathcal N(a,b):=N_C(a)+N_{C_s}(b)$ means
\begin{equation}\label{eq:geometric-AM-regularity}
	\limsup_{\substack{a\to\bar x,\,a\in C\\b\to\bar x,\,b\in C_s}}\mathcal N(a,b)
	\subseteq\mathcal N(\bar x,\bar x).
\end{equation}
It concerns limits of sums, not separate boundedness. Symmetry alone is insufficient.

\begin{example}[AM-stationarity without M-stationarity]
	\label{ex:AM-not-M}
	Let
	\[
	C=\{x\in\mathbb R^2:(x_1-1)^2+(x_2-1)^2\le 1\},
	\qquad
	C_1=\{x\in\mathbb R^2:\|x\|_0\le s=1\},
	\]
	and consider $f(x)=-x_1$. Take the feasible point
	$\bar x=(1,0)\in C\cap C_1$. We first show that $\bar x$ is not M-stationary.
	Since $g(\bar x)=(-1,0)$, the convex normal cone to $C$ at $\bar x$ is
	\[
	N_C(\bar x)=\{(0,-\lambda):\lambda\ge 0\},
	\]
	while, because $\|\bar x\|_0=1$, the sparse normal cone is $N_{C_1}(\bar x)=\{(0,t):t\in\mathbb R\}$. 
	Hence
	\[
	N_C(\bar x)+N_{C_1}(\bar x)\subseteq \{(0,t):t\in\mathbb R\},
	\]
	which contains only vertical vectors. Therefore
	\[
	-g(\bar x)=(1,0)\notin N_C(\bar x)+N_{C_1}(\bar x),
	\]
	and so $\bar x$ is not M-stationary. Next, we show that $\bar x$ is AM-stationary.
	Let $\delta_k\in(0,1)$ with $\delta_k\downarrow 0$ and define
	\[
	a^k=\Bigl(1+\delta_k,\ 1-\sqrt{1-\delta_k^2}\Bigr)\in C,
	\qquad
	b^k=\bar x=(1,0)\in C_1.
	\]
	Then $a^k\to \bar x$ and $b^k\to \bar x$.
	Setting
	\[
	u^k=\left(1,-\delta_k^{-1}\sqrt{1-\delta_k^2}\right)\in N_C(a^k),
	\qquad
	v^k=\left(0,\delta_k^{-1}\sqrt{1-\delta_k^2}\right)\in N_{C_1}(b^k)
	\]
	we have, for every $k$, since $g(a^k)=(-1,0)$,
	\[
	g(a^k)+u^k+v^k
	=
	(-1,0)+\Big(1,-\delta_k^{-1}{\sqrt{1-\delta_k^2}}\Big)
	+\Big(0,\delta_k^{-1}{\sqrt{1-\delta_k^2}}\Big)
	=0.
	\]
	Thus, $\bar x$ is AM-stationary. Finally, the example also shows the failure of joint AM-regularity at $\bar x$:
	although $u^k+v^k=(1,0)$ stays bounded, both sequences
	$\{u^k\}$ and $\{v^k\}$ are unbounded as $k\to\infty$.
	Hence one cannot extract limiting normals
	$u\in N_C(\bar x)$ and $v\in N_{C_1}(\bar x)$ satisfying $g(\bar x)+u+v=0$. This explains why AM-stationarity does not imply M-stationarity here.
\end{example}


\begin{figure}[htbp]
	\centering
	\begin{tikzpicture}[
		x=2.45cm,
		y=2.45cm,
		>=Latex,
		font=\small
		]
		
		
		\fill[teal!8] (1,1) circle (1);
		\draw[teal!75!black,very thick] (1,1) circle (1);
		\node[teal!75!black] at (1.55,1.72) {$C$};
		
		\draw[gray,densely dashed,thick]
		(-0.15,0) -- (2.30,0);
		\draw[gray,densely dashed,thick]
		(0,-0.20) -- (0,2.30);
		
		\draw[->,black!65] (-0.20,0) -- (2.38,0)
		node[right] {$x_1$};
		\draw[->,black!65] (0,-0.25) -- (0,2.38)
		node[above] {$x_2$};
		
		\node[gray,above] at (2.15,0) {$C_1$};
		\node[gray,right] at (0,2.17) {$C_1$};
		
		\fill[red] (1,0) circle (2.5pt);
		\node[below left=4pt] at (1,0)
		{$\bar x=(1,0)$};
		
		\fill[red] (0,1) circle (2.5pt);
		\node[left=5pt] at (0,1) {$(0,1)$};
		
		\foreach \dd in {0.45,0.30,0.18,0.10}{
			\pgfmathsetmacro{\xx}{1+\dd}
			\pgfmathsetmacro{\yy}{1-sqrt(1-\dd*\dd)}
			\fill[blue] (\xx,\yy) circle (1.8pt);
		}
		
		\node[blue,above right] at (1,0.1)
		{$a^k\to\bar x$};
		
		\coordinate (ak) at
		(1.45,{1-sqrt(1-0.45*0.45)});
		
		\draw[->,blue,very thick]
		(ak) -- ++(0.23,-0.46);
		
		\node[blue,right] at (1.72,-0.34) {$u^k$};
		
		\draw[->,purple,very thick]
		(1,0) -- ++(0,0.60);
		
		\node[purple,left] at (1,0.57) {$v^k$};
		
		\draw[->,orange!90!black,very thick]
		(1,0) -- ++(0.65,0);
		
		\node[orange!80!black,below] at (1.3,-0.04)
		{$-g(\bar x)$};
		
		
		\tikzset{
			explanation/.style={
				draw=black!25,
				rounded corners=2pt,
				fill=black!2,
				text width=5.65cm,
				align=left,
				inner sep=6pt,
				font=\footnotesize,
				anchor=north west
			}
		}
		
		\node[explanation] at (2.75,2.18) {
			\textbf{AM-stationarity}\\[3pt]
			$a^k\in C,\quad b^k=\bar x\in C_1$, $a^k,b^k\to\bar x$, \\
			$u^k\in N_C(a^k),\
			v^k\in N_{C_1}(b^k)$.
		};
		
		\node[explanation] at (2.75,1.31) {
			\textbf{Asymptotic normal cancellation}\\[3pt]
			$g(a^k)+u^k+v^k=0$,\\
			$u^k+v^k=(1,0)$, $\|u^k\|,\|v^k\|\to\infty$.
		};
		
		\node[explanation] at (2.75,0.43) {
			\textbf{M-stationarity fails}\\[3pt]
			$N_C(\bar x)+N_{C_1}(\bar x)
			=\{0\}\times\mathbb R$,\\
			$-g(\bar x)=(1,0)$
			does not belong to this sum. Hence joint AM-regularity fails.
		};
		
	\end{tikzpicture}
	
	\caption{Geometric illustration of
		AM-stationarity without M-stationarity
		(Example~\ref{ex:AM-not-M}).
		The points $a^k$ approach $\bar x$ along
		the boundary of $C$, while $b^k=\bar x$
		remains fixed in $C_1$.
		The corresponding normals cancel the
		gradient, although their individual norms
		diverge. Since their sum has a limit outside
		$N_C(\bar x)+N_{C_1}(\bar x)$,
		joint AM-regularity fails at $\bar x$.}
	\label{fig:AM-not-M}
\end{figure}

\begin{lemma}[Two-point normals for a supportwise local minimum]
	\label{lem:two-point-normal-support}
	Let $D\subseteq\mathbb R^n$ be nonempty, closed and convex, let
	$\mathcal L\subseteq[n]$ be a support with $|\mathcal L|=s$, and
	let $q\in C^1(\mathbb R^n)$. If $\bar y\in D\cap S_{\mathcal L}$
	locally minimizes $q$ over $D\cap S_{\mathcal L}$, then there are
	$a^k\in D$, $b^k\in S_{\mathcal L}$, and normals
	$u^k\in N_D(a^k)$, $v^k\in S_{\mathcal L}^{\perp}$ such that
	$a^k,b^k\to\bar y$ and
	$\nabla q(a^k)+u^k+v^k\to0$.
\end{lemma}
\begin{proof}
	Choose $\delta>0$ so local minimality holds on the closed ball, and minimize
	$q(a)+\|a-\bar y\|^2/2+k\|a-b\|^2/2$
	on $(D\cap\overline B_\delta(\bar y))\times(S_{\mathcal L}\cap\overline B_\delta(\bar y))$.
	Compactness gives minimizers $(a^k,b^k)$. Comparison with $(\bar y,\bar y)$ bounds $k\|a-b\|^2/2$ by $q(\bar y)-\min_{D\cap\overline B_\delta(\bar y)}q$, so every cluster pair is $(z,z)$. The same comparison and local minimality give
	$q(z)+\|z-\bar y\|^2/2\le q(\bar y)\le q(z)$; hence $a^k,b^k\to\bar y$.
	The balls are eventually inactive. Separate Fermat conditions give
	$u^k=-\nabla q(a^k)-(a^k-\bar y)-k(a^k-b^k)\in N_D(a^k)$ and
	$v^k=k(a^k-b^k)\in S_{\mathcal L}^\perp$. Their residual is $-(a^k-\bar y)\to0$.
\end{proof}

The following necessary condition is a specialization of
\cite[Theorem~2.4(a)]{KanzowLapucci} to the identity
mapping and the sparse set $C_s$. We retain a direct
supportwise proof based on Lemma~\ref{lem:two-point-normal-support}.

\begin{theorem}[Local minimizers are AM-stationary]\label{thm:locmin-CCAM}
	Every local minimizer of $f$ over $\Omega$ is AM-stationary in the sense of~\eqref{eq:geometric-AM}.
\end{theorem}
\begin{proof}
	Choose any full-size super-support $\mathcal L$ of the local minimizer.
	Since $C\cap S_{\mathcal L}\subseteq\Omega$, it locally minimizes $f$
	on $C\cap S_{\mathcal L}$. We apply
	Lemma~\ref{lem:two-point-normal-support} with $D=C$ and $q=f$, so that $a^k\in C$, $b^k\in S_{\mathcal L}$,
	$u^k\in N_C(a^k)$ and $v^k\in S_{\mathcal L}^{\perp}$, with
	$a^k,b^k\to\bar x$ and
	$g(a^k)+u^k+v^k\to0$.
	Since $I_1(b^k)\subseteq\mathcal L$ and $|\mathcal L|=s$,
	\eqref{eq:sparse-normal-compact} gives
	$v^k\in N_{C_s}(b^k)$.
\end{proof}
\begin{lemma}[A sufficient regularity condition]
	\label{lem:AM-regularity-polyhedral}\label{lem:CCP-symmetric}
	If $C$ is polyhedral, then~\eqref{eq:geometric-AM-regularity}
	holds at every $\bar x\in\Omega$.
\end{lemma}
\begin{proof}
	Write $C:=\{x\in\mathbb R^n:Ax\le d,\ Ex=h\}$
	for suitable $A\in\mathbb R^{m\times n}$,
	$E\in\mathbb R^{p\times n}$, $d\in\mathbb R^m$, and
	$h\in\mathbb R^p$. Constraints inactive at $\bar x$ stay inactive
	nearby, so $N_C(a)\subseteq N_C(\bar x)$ for all $a\in C$
	sufficiently close to $\bar x$. Nonzeros of $\bar x$ persist in
	nearby $b\in C_s$, giving
	$N_{C_s}(b)\subseteq N_{C_s}(\bar x)$.
	The target sum
	\[
	N_C(\bar x)+N_{C_s}(\bar x)
	=
	\bigcup_{\mathcal L\in\mathfrak L_s(\bar x)}
	\bigl(N_C(\bar x)+S_{\mathcal L}^{\perp}\bigr)
	\]
	is closed, since it is a finite union of polyhedral cones.
	Thus limits of nearby sums remain in
	$N_C(\bar x)+N_{C_s}(\bar x)$.
\end{proof}

A related polyhedral regularity conclusion appears in
\cite[Corollary~4.10(b)]{kanzow2021}, where
CC-AM-regularity is established for affine constraint
systems. That condition is defined through a
cardinality-specific multiplier mapping, whereas Lemma \ref{lem:AM-regularity-polyhedral} concerns joint geometric AM-regularity
of $N_C$ and $N_{C_s}$.

\begin{lemma}[Regularity when zero is an interior point]
	\label{lem:AM-regularity-interior-zero}
	If $C$ is closed and convex and $0\in\operatorname{int}C$, then
	\eqref{eq:geometric-AM-regularity} holds at every $\bar x\in\Omega$.
	No symmetry of $C$ is needed for this conclusion.
\end{lemma}
\begin{proof}
	Let $u^k+v^k\to w$ with normals at $a^k,b^k\to\bar x$.
	Because there are finitely many supports, pass to a subsequence with a fixed full super-support $\mathcal L$ such that
	$v^k\in S_{\mathcal L}^\perp$; persistence of the nonzero coordinates of $\bar x$ gives
	$\bar x\in S_{\mathcal L}$. Since $\bar x\in S_{\mathcal L}$ and
	$v^k\in S_{\mathcal L}^{\perp}$, 
	$(v^k)^T\bar x=0$ is obtained. Choosing $r>0$ with $\overline B_r(0)\subset C$, for $u^k\ne0$,
	testing the normal inequality with
	$r u^k/\|u^k\|\in C$ gives
	$(u^k)^Ta^k\ge r\|u^k\|$; the case $u^k=0$ is trivial. Hence,
	using $(v^k)^T\bar x=0$,
	\[
	(u^k+v^k)^T\bar x
	=(u^k)^Ta^k+(u^k)^T(\bar x-a^k)
	\ge(r-\|a^k-\bar x\|)\|u^k\|.
	\]
	The sums are bounded, and $r-\|a^k-\bar x\|\ge r/2$ for all sufficiently large $k$. Thus $u^k$ is bounded; boundedness of $v^k$ follows from boundedness of $u^k+v^k$. Extracting convergent subsequences and using closedness of the convex normal graph and $S_{\mathcal L}^\perp$ gives $w\in N_C(\bar x)+S_{\mathcal L}^\perp\subseteq N_C(\bar x)+N_{C_s}(\bar x)$.
\end{proof}

The next result follows from the general
AM-regularity principle in
\cite[Theorem~2.4(b)]{KanzowLapucci},
specialized to $G=I$ and $D=C_s$.

\begin{theorem}[AM-stationarity implies M-stationarity]
	\label{thm:AM-to-M}
	If a feasible $\bar x$ is AM-stationary and satisfies
	\eqref{eq:geometric-AM-regularity}, then it is M-stationary.
	Consequently, every local minimizer satisfying this regularity
	condition is M-stationary.
\end{theorem}
\begin{proof}
	By AM-stationarity, there are $a^k\in C$, $b^k\in C_s$ with $a^k,b^k\to\bar x$, and
	$u^k\in N_C(a^k)$, $v^k\in N_{C_s}(b^k)$ such that
	$g(a^k)+u^k+v^k\to0$. Since $g$ is continuous,
	$u^k+v^k\to-g(\bar x)$. Therefore, by~\eqref{eq:geometric-AM-regularity},
	$-g(\bar x)\in N_C(\bar x)+N_{C_s}(\bar x)$,  which is exactly M-stationarity. The local-minimum assertion follows
	from Theorem~\ref{thm:locmin-CCAM}.
\end{proof}

\begin{remark}[Basic feasibility]\label{rem:BF-distinction}
	For feasible $x$,  BF~\cite{Beck2013,Beck2016} requires $g(x)^T(v-x)\ge0$ for every $v\in C\cap S_{\mathcal L}$ and every super-support $\mathcal L$ of $x$. In contrast,~\eqref{eq:geometric-M} selects one favorable super-support. At full cardinality, the latter implies BF by taking inner products of the cancelling normals with $v-x$; at incomplete support, the implication fails (Example~\ref{app:BF-example}). Padding a support does not change these quantifiers.
\end{remark}

\begin{definition}[Projection stationarity]\label{def:alpha-stationarity}
	For $\alpha>0$, $\bar x\in\Omega$ is $\alpha$-stationary if
	\begin{equation}\label{eq:alpha-stationarity}
		\bar x\in P_\Omega(\bar x-\alpha g(\bar x)).
	\end{equation}
	This is classical $L$-stationarity with $L=1/\alpha$~\cite[Definition~5.2]{Beck2016}; see~\cite[Definition~2.3]{Beck2013}. It asserts neither uniqueness nor global optimality.
\end{definition}

\begin{proposition}[Local optimality and projection stationarity]
\label{prop:local-min-projection}
Let $\Omega\subseteq\mathbb R^n$ be nonempty and closed, let $f$ be continuously differentiable with locally Lipschitz gradient, and let $\bar x\in\Omega$ be a local minimizer of $f$ over $\Omega$. Then there is $\bar\alpha>0$ such that $\bar x\in P_\Omega(\bar x-\alpha g(\bar x))$ for every $\alpha\in(0,\bar\alpha]$. The converse is false, even for a smooth unconstrained problem.
\end{proposition}
\begin{proof}
Choose $\delta>0$ such that $f(y)\ge f(\bar x)$ for all $y\in\Omega\cap\overline B_\delta(\bar x)$, and let $L$ be a Lipschitz constant for $g$ on this ball. If $y\in\Omega$ and $\|y-\bar x\|\le\delta$, the descent lemma and local optimality give
\[
 g(\bar x)^T(y-\bar x)\ge -\frac L2\|y-\bar x\|^2.
\]
If instead $\|y-\bar x\|>\delta$, Cauchy--Schwarz gives
\[
 g(\bar x)^T(y-\bar x)\ge-\frac{\|g(\bar x)\|}{\delta}\|y-\bar x\|^2.
\]
Set $A:=\max\{L,2\|g(\bar x)\|/\delta,1\}$ and $\bar\alpha:=1/A$. Thus, for all $y\in\Omega$ and $0<\alpha\le\bar\alpha$,
\[
 g(\bar x)^T(y-\bar x)+\frac{\|y-\bar x\|^2}{2\alpha}\ge0.
\]
Expanding squared distances shows that this is precisely $\bar x\in P_\Omega(\bar x-\alpha g(\bar x))$. Conversely, with $\Omega=\mathbb R$ and $f(x)=-x^2$, the point $\bar x=0$ satisfies the projection condition for every $\alpha>0$ but is a strict local maximum.
\end{proof}

The following relationships combine classical sparse
projection-stationarity properties
\cite{Beck2013,Beck2016} with the geometric
AM/M framework of \cite{KanzowLapucci}.
Related implications between cardinality-specific
sequential conditions are studied in
\cite{kanzow2021}; these conditions use a different
normal or multiplier representation.

\begin{proposition}[Local minima and the stationarity relationships]
	\label{prop:stationarity-relations}
	For closed convex $C$, every local minimizer of $f$ on $\Omega$ is
	BF-stationary. Every $\alpha$-stationary point is BF-stationary.
	BF-stationarity is equivalent to the intrinsic condition
	$-g(\bar x)\in N_\Omega^F(\bar x)$ and implies geometric
	AM-stationarity. Under joint AM-regularity, it consequently implies
	geometric M-stationarity. At full cardinality, M-stationarity implies
	BF; at incomplete support, the converse implication from M to BF can
	fail, as described in Remark~\ref{rem:BF-distinction}.
\end{proposition}
\begin{proof}
	For an admissible $v\in C\cap S_{\mathcal L}$ and a super-support
	$\mathcal L$ of $\bar x$, the segment
	$\bar x+t(v-\bar x)$ is feasible for every $t\in[0,1]$.
	Differentiating local minimality at $t=0$ gives BF.
	For an $\alpha$-stationary point, applying the projection inequality
	to $\bar x+t(v-\bar x)$, $t\in(0,1]$, gives
	\[
	t\,g(\bar x)^T(v-\bar x)
	+\frac{t^2}{2\alpha}\|v-\bar x\|^2\ge0.
	\]
	Dividing by $t$ and letting $t\downarrow0$ gives BF.

	If $-g(\bar x)\in N_\Omega^F(\bar x)$, testing the Fr\'echet-normal
	condition along the same segment and letting $t\downarrow0$ also
	gives BF. Conversely, nearby feasible points retain all nonzeros of
	$\bar x$, so their supports can be extended to super-supports of
	$\bar x$. Hence, BF gives
	$g(\bar x)^T(y-\bar x)\ge0$ for every feasible $y$ sufficiently close
	to $\bar x$. Thus, $-g(\bar x)\in N_\Omega^F(\bar x)$.

	Moreover, the same inequality shows that the linear function $y\mapsto g(\bar x)^T(y-\bar x)$
	locally minimizes over $\Omega$ at $\bar x$.
	Theorem~\ref{thm:locmin-CCAM}, applied to this linear function, gives
	sequences with
	$g(\bar x)+u^k+v^k\to0$. Since $a^k\to\bar x$ and $g$ is continuous,
	$g(a^k)+u^k+v^k\to0$, which is AM-stationarity for $f$.
	Theorem~\ref{thm:AM-to-M} gives M-stationarity under regularity.
	Remark~\ref{rem:BF-distinction} treats the final implications.
\end{proof}

The intrinsic normal is to the actual intersection, whereas M splits into two sets. No general intersection sum rule is assumed. In the sparse neighborhood search framework~\cite[Definition~2.4]{Lapucci2023},
$\mathcal N$-stationarity combines stationarity on the current convex slice
with objective comparisons over a prescribed discrete neighborhood;
equal-value neighboring points must also be stationary on their
respective slices. Our BF and projection conditions do not explicitly
test these neighboring objective values, so no general ordering is
asserted without specifying the neighborhood.

\begin{definition}[All-super-support S-stationarity]
	\label{def:all-support-S}
	Motivated by universal-super-support multiplier conditions
	associated with Lu--Zhang stationarity~\cite{lu2013sparse,Lapucci2023}
	and by lifted CC-S-stationarity~\cite[Definition~4.12]{kanzow2021},
	we call a feasible
	point $\bar x\in\Omega$ \emph{all-super-support S-stationary} if,
	for every super-support $\mathcal L$ of $\bar x$, there exist $u_{\mathcal L}\in N_C(\bar x)$ and
	$v_{\mathcal L}\in S_{\mathcal L}^{\perp}$
	such that
	\begin{equation}\label{eq:all-support-S}
		g(\bar x)+u_{\mathcal L}+v_{\mathcal L}=0.
	\end{equation}
	The qualifier ``all-super-support'' is essential: unlike the lifted
	CC-S condition at a fixed auxiliary variable,~\eqref{eq:all-support-S}
	requires a first-order certificate for every admissible full
	super-support of $\bar x$. This is a problem-adapted split-normal
	condition, not the lifted CC-S-stationarity definition itself.
\end{definition}

\subsection{A geometric residual}\label{sec:geometric-residual}
For $x\in\Omega$, set $\mathfrak L_s(x):=\{\mathcal L\subseteq[n]:I_1(x)\subseteq\mathcal L,\ |\mathcal L|=s\}$ and define
\begin{equation}\label{eq:rgM-numerical}
	\mathrm{rg}_M(x):=\frac{\min_{\mathcal L\in\mathfrak L_s(x)}\inf_{u\in N_C(x)}\|(g(x)+u)_{\mathcal L}\|}
	{\max\{1,\|g(x)\|\}}.
\end{equation}
The sparse normal cancels components outside $\mathcal L$. M-stationarity implies zero residual. The converse holds if $U_{\mathcal L}^TN_C(x)$ is closed for every super-support $\mathcal L$, as for polyhedra and Euclidean balls: since there are finitely many super-supports, a zero minimum then gives an attained cancellation on one of them. Without closedness, the infimum can be zero without an attaining normal, so a zero residual alone need not certify M-stationarity. For $C=\mathbb R^n$, let $I=I_1(x)$ and choose the $s-|I|$ smallest inactive $|g_i(x)|$, with indices $J$. Then
\begin{equation}\label{eq:unconstrained-M-residual}
	\mathrm{rg}_M(x)=\frac{\Big(\sum_{i\in I}g_i(x)^2+\sum_{i\in J}g_i(x)^2\Big)^{1/2}}{\max\{1,\|g(x)\|\}}.
\end{equation}
For $C=\{x:Dx\le d,Ex=e\}$, with $D\in\mathbb R^{m\times n}$, $E\in\mathbb R^{p\times n}$, $d\in\mathbb R^m$ and $e\in\mathbb R^p$, each inner problem is
$\min_{\lambda\ge0,\nu}\|[g(x)+D_{\mathcal A(x)}^T\lambda+E^T\nu]_{\mathcal L}\|$, with $\mathcal A(x)=\{i:D_ix=d_i\}$ and $D_i$ as the $i$th row of $D$. On $B_2(R)$ the convex normal is $\{0\}$ in the interior and $\{tx:t\ge0\}$ on the boundary. Full cardinality leaves one super-support; incomplete support needs a proved reduction or complete search for coupled normals. A restricted support list or a verified normal gives an upper bound. Unverified approximate solves are estimates, not certificates; M alone does not certify universal BF. In floating-point computations the support and active sets depend on thresholds; without a documented exhaustive search or justified reduction and validated solve accuracy, a reported numerical value is an estimate, not a certified residual.

\section{Safeguarded penalty decomposition}\label{main}

The proposed safeguarded \texttt{QNPD} method consists of a main algorithm
(Algorithm~\ref{a.CPDQN}) and two subalgorithms
(Algorithms~\ref{a.EPD-inner} and~\ref{a.reference}),
whose main steps are illustrated in
Figure~\ref{fig:pdqn-three-flowcharts}.
The main algorithm maintains a feasible iterate $w^k$, a model
center $\xi^k$, a multiplier $\mu^k$, a penalty parameter $\rho^k$,
and a safeguarded diagonal curvature matrix $H^k$.
Its execution is controlled by a prescribed stationarity tolerance
$\epsilon_{\rm stat}$ and finite computational budgets.

At each outer iteration, Algorithm~\ref{a.reference} first computes
a feasible reference point
$p^k\in P_\Omega(w^k-\alpha_k g(w^k))$.
Here, $\overline\alpha>0$ is the initial reference stepsize,
$\beta\in(0,1)$ is the backtracking factor, and $\sigma>0$
controls sufficient decrease in the original objective function.
The procedure also computes a stationarity witness $E_k$,
which measures the first-order stationarity residual at the
reference point. If $\|E_k\|\le\epsilon_{\rm stat}$, the method
terminates and returns $p^k$ with its stationarity certificate.

Otherwise, the main algorithm may generate a quasi-Newton trial.
It first prepares the safeguarded curvature model and may adjust
its scale through extrapolation or backtracking.
This curvature-scale search is distinct from the reference
backtracking procedure. If the trial budget permits,
Algorithm~\ref{a.EPD-inner} is called with the frozen model
parameters $(z,\mu,\rho,H)$, a working radius $R>0$, a feasible
starting point, and tolerances controlling the inner residual
and agreement between the smooth and sparse blocks.
The inner solver alternates between an exact smooth-block update
and a sparse support projection, using a model-descent test
to accept support changes. The safeguarded trial disables only the inner objective filter;
the model-descent, residual and agreement tests remain active.

The main algorithm accepts an admissible trial only if it satisfies
both the reference-domination and reference-envelope tests.
Otherwise, including when the trial is skipped or fails,
the method takes the reference step $p^k-w^k$, setting
$w^{k+1}=p^k$. The penalty and curvature information are then
updated before the next iteration. If an external budget is
exhausted or an essential oracle fails, the method returns
a feasible point with the corresponding status, without
asserting an unverified stationarity certificate.

\begin{remark}[Exact projection cost]\label{rem:exact-global-projection}
	\begin{samepage}
		A decreasing sorting permutation $\pi$ of a projection input $v$ defines
		\begin{equation}\label{eq:reduced-reference-supports}
			T_k=\{\pi(1),\ldots,\pi(k)\}\cup
			\{\pi(n-s+k+1),\ldots,\pi(n)\},\quad0\le k\le s,
		\end{equation}
	\end{samepage}
	where an empty end block contributes no indices. We compute
	\begin{equation}\label{eq:reduced-reference-projection}
		y^{[k]}=P_{C\cap S_{T_k}}(v),\qquad
		k_*=\min\Argmin_{0\le k\le s}\|y^{[k]}-v\|^2.
	\end{equation}
	Then $y^{[k_*]}\in P_\Omega(v)$~\cite[Theorem~4.1 and Algorithm~1]{Beck2016}.
	Coercivity gives a global closest point. Rearrangement permits one
	ordered like $v$ without increasing distance; if it has $k$
	positive and $m$ negative entries, then $k+m\le s$ and its
	support lies in $T_k$. The slice comparison therefore proves the
	reduction. Each slice is nonempty by permutation. The next example shows why sparse projection must use the precise
	symmetry of $C$.
	
	\begin{example}[Symmetry determines the projection rule]
		\label{ex:symmetry-dependent-projection}
		Under the standing assumptions, if $C\subseteq\mathbb R_+^n$,
		the $s$ largest signed entries of the input determine one
		sufficient slice. Invariance under independent sign changes
		and coordinate permutations instead permits the $s$ largest
		magnitudes~\cite[Theorems~4.2--4.4]{Beck2016}.
		General permutation symmetry admits a reduction to at most
		$s+1$ candidate slices. For $C=\{\mathbf1^Tx=1\}$, $n=4$, $s=2$ and
		$v=(4,1,-2,-3)$, the squared projection distances on the
		slices indexed by $T_0=\{3,4\}$, $T_1=\{1,4\}$ and
		$T_2=\{1,2\}$ are $35$, $5$ and $21$, respectively.
		Thus, signed top-$s$ selection misses the true projection
		$(4,0,0,-3)$.
		
		Central symmetry alone does not justify magnitude
		selection either. For $C=\{\mathbf1^Tx=0\}$, $n=3$,
		$s=2$ and $v=(10,9,-1)$, the two largest magnitudes
		select $\{1,2\}$, whose slice gives squared distance
		$363/2$, whereas $\{1,3\}$ gives $243/2$.
	\end{example}
	We therefore use single-slice shortcuts only under their stated
	symmetry assumptions. Otherwise, we compare the candidate slice
	projections using their full distances to the input.
	
	Sorting costs $O(n\log n)$, and each convex slice solve has
	representation-dependent arithmetic cost. Exhaustive
	$\binom ns$ enumeration is exact but unnecessary. The reference
	uses these global projections, whereas the inner ranking uses
	selected slices. Neither uniqueness nor Lipschitz continuity
	of $P_\Omega$ is assumed.
\end{remark}

\begin{figure}[H]
	\centering
	\refstepcounter{algorithm}\label{a.CPDQN}
	\refstepcounter{algorithm}\label{a.EPD-inner}
	\refstepcounter{algorithm}\label{a.reference}
	
	\begin{tikzpicture}[
		>=Latex,
		pdtitle/.style={
			draw=black!70,
			rounded corners=2pt,
			fill=black!10,
			align=center,
			text width=4.10cm,
			minimum height=8mm,
			inner sep=2pt,
			font=\bfseries\footnotesize
		},
		pdstep/.style={
			draw=black!55,
			rounded corners=2pt,
			fill=black!1,
			align=center,
			text width=4.10cm,
			minimum height=8mm,
			inner sep=2.5pt,
			font=\scriptsize
		},
		pdcheck/.style={
			draw=black!65,
			rounded corners=2pt,
			fill=black!6,
			align=center,
			text width=2.60cm,
			minimum height=8mm,
			inner sep=2pt,
			font=\scriptsize
		},
		pdstop/.style={
			draw=black!65,
			rounded corners=2pt,
			fill=black!3,
			align=center,
			text width=1.10cm,
			minimum height=8mm,
			inner sep=1.5pt,
			font=\scriptsize
		},
		pdchoice/.style={
			draw=black!60,
			rounded corners=2pt,
			fill=black!2,
			align=center,
			text width=1.65cm,
			minimum height=8mm,
			inner sep=2pt,
			font=\scriptsize
		},
		pdnote/.style={
			draw=black!40,
			dashed,
			rounded corners=2pt,
			align=center,
			text width=3.95cm,
			inner sep=3pt,
			font=\scriptsize
		},
		pdarrow/.style={
			->,
			line width=.5pt,
			draw=black!75
		},
		pdlabel/.style={
			font=\tiny,
			fill=white,
			inner sep=.8pt
		}
		]
		
		\begin{scope}[xshift=0cm]
			
			\node[pdtitle] (mT) at (0,0)
			{Algorithm~\ref{a.CPDQN}\\
				Safeguarded \texttt{QNPD}};
			
			\node[pdstep] (m0) at (0,-1.05)
			{\PDstep{0}{a.CPDQN}: Initialize feasible $w^0$, state,
				tolerances and budgets};
			
			\node[pdstep] (m1) at (0,-2.15)
			{\PDstep{1}{a.CPDQN}: Call Algorithm~\ref{a.reference};
				obtain $(\alpha_k,p^k,E_k)$};
			
			\node[pdcheck] (m2) at (-.65,-3.35)
			{\PDstep{2}{a.CPDQN}: Stop if
				$\|E_k\|\le\epsilon_{\rm stat}$?};
			
			\node[pdstop] (mStop) at (1.72,-3.35)
			{return $p^k$};
			
			\node[pdstep] (m3) at (0,-5.05)
			{\PDstep{3}{a.CPDQN}: Clip $B$; optional scale search:\\
				$\alpha=1$ passes $\Rightarrow \alpha=2,4,\ldots$
				(extrapolate);\\
				$\alpha=1$ fails $\Rightarrow
				\alpha=\frac12,\frac14,\ldots$ (backtrack)};
			
			\node[pdstep] (m4) at (0,-6.95)
			{\PDstep{4}{a.CPDQN}: If trial budget permits, call
				Algorithm~\ref{a.EPD-inner}; otherwise skip};
			
			\node[pdcheck] (m5) at (0,-8.85)
			{\PDstep{5}{a.CPDQN}: Admissible trial and both objective
				tests\\(domination + envelope) pass?};
			
			\node[pdchoice] (mA) at (-1.10,-10.55)
			{accept $q^k$};
			
			\node[pdchoice] (mR) at (1.10,-10.55)
			{use $p^k$};
			
			\coordinate (mJoin) at (0,-11.45);
			
			\node[pdstep] (m6) at (0,-12.25)
			{\PDstep{6}{a.CPDQN}: Refresh $\rho,H$ and repeat
				when budget permits};
			
			\node[pdstep] (m7) at (0,-13.75)
			{\PDstep{7}{a.CPDQN}: External budget/oracle exit:
				return feasible point and status};

			\draw[pdarrow] (mT)--(m0);
			\draw[pdarrow] (m0)--(m1);
			\draw[pdarrow] (m1)--(m2);
			
			\draw[pdarrow]
			(m2.east)
			-- node[pdlabel,above] {yes}
			(mStop.west);
			
			\draw[pdarrow]
			(m2.south)
			-- node[pdlabel,right] {no}
			(m3.north);
			
			\draw[pdarrow] (m3)--(m4);
			\draw[pdarrow] (m4)--(m5);
			
			\draw[pdarrow]
			(m5.south west)
			-- node[pdlabel,left] {yes}
			(mA.north);
			
			\draw[pdarrow]
			(m5.south east)
			-- node[pdlabel,right] {no / skip}
			(mR.north);
			
			\draw[draw=black!75,line width=.5pt]
			(mA.south) |- (mJoin);
			
			\draw[draw=black!75,line width=.5pt]
			(mR.south) |- (mJoin);
			
			\draw[pdarrow]
			(mJoin) -- (m6.north);
			
			\draw[pdarrow]
			(m6)
			-- node[pdlabel,right] {budget/oracle exit}
			(m7);
			
			\draw[pdarrow,dashed]
			(m6.west)
			-- ++(-.55,0)
			|- (m1.west);
			
		\end{scope}

		\begin{scope}[xshift=5.10cm]
			
			\node[pdtitle] (iT) at (0,0.1)
			{Algorithm~\ref{a.EPD-inner}\\
				Frozen-model inner solver};
			
			\node[pdstep] (i0) at (0,-1.3)
			{\PDstep{0}{a.EPD-inner}: Receive frozen
				$(z,\mu,\rho,H)$, tolerances and budget};
			
			\node[pdstep] (i1) at (0,-2.5)
			{\PDstep{1}{a.EPD-inner}: Validate feasible $y_0$;
				initialize size-$s$ support};
			
			\node[pdstep] (i2) at (0,-3.72)
			{\PDstep{2}{a.EPD-inner}: Exact smooth-block update
				$x_\ell$};
			
			\node[pdstep] (i3) at (0,-4.97)
			{\PDstep{3}{a.EPD-inner}: Propose support and project
				onto its convex slice};
			
			\node[pdstep] (i4) at (0,-6.22)
			{\PDstep{4}{a.EPD-inner}: Model-descent gate;
				if rejected, retain preceding support};
			
			\node[pdstep] (i5) at (0,-7.52)
			{\PDstep{5}{a.EPD-inner}: Check residual, agreement,
				and optional filter; return if admissible};
			
			\node[pdstep] (i6) at (0,-9.07)
			{\PDstep{6}{a.EPD-inner}: If budget exhausted,
				return failure/budget status};

			\draw[pdarrow] (iT.south)--(i0.north);
			\draw[pdarrow] (i0)--(i1);
			\draw[pdarrow] (i1)--(i2);
			\draw[pdarrow] (i2)--(i3);
			\draw[pdarrow] (i3)--(i4);
			\draw[pdarrow] (i4)--(i5);
			
			\draw[pdarrow]
			(i5)
			-- node[pdlabel,right] {budget exhausted}
			(i6);
			
			\coordinate (iLoopA) at ($(i5.west)+(-.52,0)$);
			\coordinate (iLoopB) at (iLoopA |- i2.west);
			
			\draw[pdarrow,dashed]
			(i5.west) -- (iLoopA)
			-- node[pdlabel,left,rotate=90] {continue} (iLoopB)
			-- (i2.west);
			
		\end{scope}

		\begin{scope}[xshift=10.20cm]
			
			\node[pdtitle] (rT) at (0,0)
			{Algorithm~\ref{a.reference}\\
				Feasible reference};
			
			\node[pdstep] (r0) at (0,-1.05)
			{\PDstep{0}{a.reference}: Receive feasible $w^k$;
				set $\alpha=\overline\alpha$};
			
			\node[pdstep] (r1) at (0,-2.50)
			{\PDstep{1}{a.reference}: Project
				$w^k-\alpha g(w^k)$ onto $\Omega$;\\
				backtrack $\alpha\gets\beta\alpha$ until
				sufficient decrease holds};
			
			\node[pdstep] (r2) at (0,-4.05)
			{\PDstep{2}{a.reference}: Compute
				$G_k,\nu_k,E_k$ at $(\alpha_k,p^k)$};
			
			\node[pdstep] (r3) at (0,-5.30)
			{\PDstep{3}{a.reference}: Return
				$(\alpha_k,p^k,E_k,\nu_k)$ and status};
			
			\node[pdnote] (rN) at (0,-7.05)
			{Only the feasible-reference stepsize is backtracked here.\\
				Optional extrapolation belongs to
				Step~\PDstep{3}{a.CPDQN}.};

			\draw[pdarrow] (rT)--(r0);
			\draw[pdarrow] (r0)--(r1);
			\draw[pdarrow] (r1)--(r2);
			\draw[pdarrow] (r2)--(r3);
			
		\end{scope}
		
	\end{tikzpicture}
	
	\caption{High-level presentation of the three original algorithms,
		retaining their algorithm numbers and step identifiers.
		The main method first obtains a feasible reference
		(Algorithm~\ref{a.reference}); if the stationarity criterion
		is unmet, it optionally performs curvature-scale extrapolation
		or backtracking and calls the frozen-model solver
		(Algorithm~\ref{a.EPD-inner}). Both objective tests
		(reference domination and the reference envelope) must pass
		to accept a trial; otherwise the reference point is used.
		The reference procedure itself uses backtracking only.
		Precise formulas and failure conditions remain in the text.}
	
	\label{fig:pdqn-three-flowcharts}
\end{figure}

\subsection{Frozen model and exact block updates}\label{sec:diagonal-safeguards}

Steps~\PDstep{2}{a.EPD-inner}--\PDstep{4}{a.EPD-inner} use the frozen model described here. We split the variable into $x\in\mathbb R^n$ and $y\in\Omega$, with $x=y$. A center $z$, multiplier $\mu$, penalty $\rho>0$, and positive diagonal $H$ define
\begin{equation}\label{pr:pxy-subproblem}\min_{x\in\mathbb R^n,\,y\in\Omega}
	\Phi(x,y):=g(z)^T(x-z)+\tfrac12(x-z)^TH(x-z)-\mu^T(x-y)+\tfrac\rho2\|x-y\|^2.
\end{equation}
Here, $\Phi=\Phi_{(\rho,z,\mu)}$, with $H$ implicit, and $f(z)+\Phi$ models the augmented Lagrangian. All four parameters remain fixed during a call. For $0<\lambda_{\min}\le\lambda_{\max}<\infty$, we use
\begin{equation}\label{eq:diagonal-clipping}
	h_i=\min\{\lambda_{\max},\max\{\lambda_{\min},\widetilde h_i\}\},\qquad H=\operatorname{diag}(h_i).
\end{equation}
The initial raw values $\widetilde h_i^0=1$ are clipped; later undefined or nonfinite proposals retain the preceding clipped values. The choice $H=I$ is complete when $\lambda_{\min}\le1\le\lambda_{\max}$.

Step~\PDstep{6}{a.CPDQN} refreshes the safeguarded diagonal; the separate unrecovered core uses its own update in Section~\ref{sec:zeta}. Using the pairs $s^k=w^{k+1}-w^k$ and $t^k=g(w^{k+1})-g(w^k)$, we retain the latest at most $m\ge1$ available pairs, indexed by $\mathcal I_k$~\cite{LBFGS,LMBOPT}, and set $a_i^k=\max_{\ell\in\mathcal I_k}|s_i^\ell|$ and $b_i^k=\max_{\ell\in\mathcal I_k}|t_i^\ell|$. We use three limited-memory proposals, denoted \texttt{LM1}--\texttt{LM3}, and a well-conditioning-inspired proposal, denoted \texttt{WCI}. The latter uses a fixed regularization parameter $\vartheta>0$; larger values move its raw diagonal entries toward $1$:
\begin{equation}\label{eq:four-diagonal-proposals}
	\widetilde h_i^{k+1}=
	\begin{cases}
		b_i^k/a_i^k, & \texttt{LM1},\\
		a_i^k/b_i^k, & \texttt{LM2},\\
		\max\{b_i^k/a_i^k,\;a_i^k/b_i^k\}, & \texttt{LM3},\\
		\bigl(\vartheta/((t_i^k)^2+\vartheta)\bigr)^{1/4},
		& \texttt{WCI}.
	\end{cases}
\end{equation}
All proposals use~\eqref{eq:diagonal-clipping}, with its initialization and fallback conventions. The ratios are empirical diagonal proposals: \texttt{LM1} is Hessian-like, whereas \texttt{LM2} uses its reciprocal and \texttt{LM3} combines both; no scale-invariance or secant-equation property is claimed. Optional scaling below is completed before freezing $H$. Here $w^{k+1}$ is the accepted or recovered point. Basic feasible search (\texttt{BFS})~\cite[Algorithm~5]{Beck2016}, restricted \texttt{FISTA}~\cite{FISTA} and partial sparse-simplex (\texttt{PSS})~\cite[Section~3.3]{Beck2013} may supply budgeted warm starts or supports; their usual convergence claims require their own hypotheses. The final pair still follows the exact updates and tests below. An arbitrary scaled or approximate solve is not an exact update.

\bfi{Block solves.}\label{sec:solvex}\label{sec:prox-center}\label{sec:solvey}
For a frozen call, the unconstrained smooth-block minimizer satisfies
\begin{align}
	x_\ell&=\Argmin_x\Phi(x,y_{\ell-1})
	=(H+\rho I)^{-1}(Hz+\rho y_{\ell-1}-g(z)+\mu),\label{pr:px-subproblem}\\
	0&=g(z)+H(x_\ell-z)-\mu+\rho(x_\ell-y_{\ell-1}).\label{e.optPx}
\end{align}
The inverse acts componentwise. For a fixed radius $R\ge0$, let $D_R:=\Omega\cap\overline B_R(0)$, chosen so that the initial projected block belongs to $D_R$. A size-$s$ selected set $\mathcal L_\ell$ gives
\begin{equation}\label{eq:prox-center-definition}z_\ell=x_\ell-\mu/\rho,\qquad y_\ell=P_{D_R\cap S_{\mathcal L_\ell}}(z_\ell).
\end{equation}
Every slice is nonempty: a permutation moves the nonzeros of an existing point of $D_R$ into the selected set. The slice equals $C\cap\overline B_R(0)\cap S_{\mathcal L}$ and is closed and convex, so its projection is unique. A feasible radius $R=0$ gives a singleton; symmetry alone need not give a closed-form projector.

\bfi{Support proposals.}\label{sec:solvey-support}
Steps~\PDstep{3}{a.EPD-inner}--\PDstep{4}{a.EPD-inner} propose and validate a support. The set $\widetilde{\mathfrak S}(v)$ contains the decreasing sorting permutations of $v\in\mathbb R^n$. After checking the warm start, the call pads its support to size $s$ with the smallest unused indices. The ranking is
\begin{equation}\label{eq:rank-score}
	v_{\mathrm{rank},\ell}=x_\ell-H^{-1}g(x_\ell),\qquad
	\pi\in\widetilde{\mathfrak S}(p(v_{\mathrm{rank},\ell})),
\end{equation}
where index resolves ties; $p(v)=|v|$ for independent sign symmetry and $p(v)=v$ otherwise. If $y_{\ell-1}$ has fewer than $s$ nonzeros, we retain their indices and add the highest-ranked inactive indices; if it has $s$ nonzeros, we may exchange coordinates through the top $s$ indices. For the proposed support $\widehat{\mathcal L}_\ell$, 
$\widehat y_\ell=P_{D_R\cap S_{\widehat{\mathcal L}_\ell}}(z_\ell)$ is computed by \eqref{eq:prox-center-definition}. The proposed exact projection $\widehat y_\ell$ is retained only if
\begin{equation}\label{eq:inner-support-descent}
	\Phi(x_\ell,\widehat y_\ell)\le\Phi(x_\ell,y_{\ell-1}).
\end{equation}
Otherwise the method projects on the preceding selected set, where $y_{\ell-1}$ remains feasible. An unavailable ranking also keeps that set. The method stores all $s$ selected indices even when the projection has fewer nonzeros. Ranking does not replace global sparse projection.

\begin{remark}[Choice of the support score]\label{app:score-degeneracy}
	The alternative score $h_\ell=g(x_\ell)-\mu+\rho(x_\ell-y_{\ell-1})$ equals $g(x_\ell)-g(z)-H(x_\ell-z)$ by~\eqref{e.optPx}. For every matching quadratic $f(x)=x^THx/2+b^Tx$, it vanishes, whereas~\eqref{eq:rank-score} gives $-H^{-1}b$.
\end{remark}
We use~\eqref{eq:rank-score} to propose support changes and retain the model-descent test for inner exchanges. Model descent alone does not ensure descent in the original objective, as the following example shows.
\begin{example}[Model progress and objective progress]\label{ex:model-objective-gap}
	We take $n=2$, $s=1$, $C=\mathbb R^2$, $f(x)=2x_1^2$, $z=e_1$, $H=I$ and $\mu=0$. Along $x=y=te_1$, the model is $m(t)=4(t-1)+(t-1)^2/2$, for any $\rho>0$. Its minimizer $t=-3$ gives $m(-3)=-8<m(1)=0$, whereas $f(-3e_1)=18>f(e_1)=2$. Thus, frozen-model decreases cannot be telescoped as true-objective decreases.
\end{example}
Model descent is used to generate trials, while acceptance is tested separately using the original objective. These tests preserve the decrease supplied by an exact feasible projected-gradient reference and control the accepted displacement. If a trial fails, we take the reference step.

\subsection{The unrecovered core}\label{sec:zeta}
The independent core trajectory repeatedly calls Algorithm~\ref{a.EPD-inner}
without reference recovery. Step~\PDstep{5}{a.EPD-inner} accepts pairs
satisfying the core tests; the updates below then refresh the multiplier,
penalty, diagonal and next initialization. A failure or budget exit in
Step~\PDstep{6}{a.EPD-inner} does not advance the accepted index.
Sections~\ref{sec:legacy-core-scope}--\ref{sec:outerconvergence}
analyze this accepted-call trajectory. Initially, $y^{(0)}=y_0^{(0)}$ is feasible, $x^{(0)},\mu^{(0)}$ are finite, and the matrix is clipped. At accepted index $j\ge1$, we freeze $z=x^{(j-1)}$, $\mu=\mu^{(j-1)}$, $\rho=\rho^{(j-1)}$, $H=H^{(j-1)}$ and write $\Phi_j:=\Phi_{(\rho^{(j-1)},x^{(j-1)},\mu^{(j-1)})}$. The fixed domain is
\begin{equation}\label{eq:zeta-definition}C_\zeta=\Omega\cap\overline B_\zeta(0),\quad
	\zeta=\begin{cases}\max_{y\in C}\|y\|,&C\text{ bounded},\\\max\{1,\|y^{(0)}\|\},&C\text{ unbounded}.\end{cases}
\end{equation}
A known finite upper bound may replace the maximum for bounded $C$. The core has an outer limit $J_{\max}\in\mathbb N\cup\{\infty\}$, inner budgets $K_j\in\mathbb N$ and finite oracle budgets. With $0<\tau<1$, $0<\varepsilon_j\downarrow0$ and $0\leq\eta_j\downarrow0$, Step~\PDstep{5}{a.EPD-inner} admits a final pair only if
\begin{align}
	&\|\nabla_x\Phi_j(x^{(j)},y^{(j)})\|\le\varepsilon_{j-1},\label{e.gnorm}\\
	&\|x^{(j)}-y^{(j)}\|\le\tau\|x^{(j-1)}-y^{(j-1)}\|+\eta_j.\label{eq:primalDual}
\end{align}
The switch $\chi\in\{0,1\}$ enables only the objective filter: when $\chi=1$, we also require, for fixed $\sigma>0$,
\begin{equation}\label{eq:outer-objective-filter}
	f(y^{(j)})\le f(y^{(j-1)})-\sigma\|y^{(j)}-y^{(j-1)}\|^2.
\end{equation}
The agreement test and filter keep the preceding accepted pair as their anchor after any restart. Mode $\chi=0$ omits only the filter. The residual is measured after projection:
\[
\nabla_x\Phi(x_\ell,y_\ell)=g(z)+H(x_\ell-z)-\mu+\rho(x_\ell-y_\ell)
=\rho(y_{\ell-1}-y_\ell).
\]

The last equality follows from~\eqref{pr:px-subproblem}.
The zero gradient at $(x_\ell,y_{\ell-1})$ cannot replace
this test. If an inner call fails or exhausts its budget,
it returns the last feasible projected block with available
diagnostics; the accepted index and multiplier remain unchanged.
Along an infinite accepted sequence, agreement gives
\begin{equation}\label{eq:primal-dual-agreement-limit}
	\|x^{(j)}-y^{(j)}\|\to0.
\end{equation}
We test residual accuracy and agreement separately. A small final residual does not imply consensus; the following example gives a counterexample. 
\begin{example}[Residual accuracy does not ensure consensus]\label{app:nonstationary-consensus}
	We take $n=2$, $s=1$, $C=[-1,1]^2$, $f(x)=-x_1-x_2$, $z=x_0=y_0=0$, $H=I$, $\rho=100$ and $\mu=0$. With $q=100/101$ and $a_\ell=1-q^\ell$, exact updates give
	\begin{equation}\label{eq:nonstationary-consensus-example}
		x_\ell=(a_\ell,1/101),\quad y_\ell=(a_\ell,0),\quad r_\ell=(-q^\ell,0).
	\end{equation}
	The recurrence is $a_\ell=(1+100a_{\ell-1})/101$. Ranking of $(1+a_\ell,1+1/101)$ selects index $1$, initially by tie; projection and model descent pass. Every positive residual target is eventually met, but disagreement stays $1/101>\eta_1=10^{-3}$. The $\sigma=1$ filter passes because $-a_\ell\le-a_\ell^2$. Every finite incumbent, including the origin, has $\mathrm{rg}_M=1/\sqrt2$. No inner budget meets the consensus or a smaller M-target.
	The obstruction also occurs at a global minimizer: $C=[0,1]^2$ and $f(x)=x_1+x_2$, the same frozen data give $x_\ell=(-1,-1)/101$ and $y_\ell=0$. The residual is zero, but disagreement is $\sqrt2/101>10^{-3}$. The radius $\sqrt2$ is redundant in both cases.
\end{example}
A trial is declared successful only after it also passes a separate consensus test. Agreement alone, however, does not prevent a conflict between the final residual target and the objective filter.
\begin{example}[Residual and filter targets can conflict]\label{app:incompatible-filter}
	We take $C=[-1,1]^2$, $s=1$, $f(x)=\|x\|^2$, $z=x_0=y_0=ae_1$, $a=1/4$, $H=I$, $\mu=0$, $\rho=2$ and $\sigma=1$, with penalty bounds $1$ and $3$. For $q=2/3$, exact updates give $x_\ell=y_\ell=t_\ell e_1$, $t_\ell=-a+2aq^\ell$ and $r_\ell=2aq^\ell e_1$. The recurrence is $t_\ell=(-a+2t_{\ell-1})/3$; magnitude ranking of $-t_\ell e_1$ retains support $1$. Projection preserves consensus and model descent. However, $\varepsilon_0=1/100$ requires $q^\ell\le1/50$, whereas
	\[t_\ell^2\le a^2-(t_\ell-a)^2\quad\Longleftrightarrow\quad q^\ell\ge1/2.\]
	For $\sigma>0$, the filter requires $q^\ell\ge\sigma/(1+\sigma)$ and conflicts with every $\varepsilon<2a\sigma/(1+\sigma)$. Since $t_1=1/12$, $t_2=-1/36$ and $t_\ell\downarrow-1/4$, every finite incumbent has $\mathrm{rg}_M=2|t_\ell|\ge1/18$; the objective-best point is $-e_1/36$. The first preprojection model gradient is zero, whereas the final residual is $e_1/3$; the latter is the required test.
\end{example}
This conflict concerns the optional core filter ($\chi=1$).
The safeguarded trial uses $\chi=0$ and applies the objective
comparisons with the feasible reference at outer acceptance.
A trial that exhausts its finite budget without passing the required
tests is rejected, and the method takes the reference step.

After Step~\PDstep{5}{a.EPD-inner} accepts, the core uses $r>1$ and $0<\rho_{\min}<\rho^{(0)}<\rho_{\max}$ to update
\[
\mu^{(j)}=\mu^{(j-1)}-\rho^{(j-1)}(x^{(j)}-y^{(j)}),\quad
\rho^{(j)}=\max\{\rho_{\min},\min\{r\rho^{(j-1)},\rho_{\max}\}\},
\]
and refresh the clipped matrix. The multiplier uses the frozen penalty before its increase. To choose the next initialization in Step~\PDstep{1}{a.EPD-inner}, the core uses $\widehat c>0$ and an initial model safeguard satisfying
\begin{equation}\label{e.UpsilonCon}
	\Upsilon^{(0)}\ge\max\{f(x^{(0)}),\min_x\Phi_{(\rho^{(0)},x^{(0)},\mu^{(0)})}(x,y_0^{(0)}),\widehat c\}>0.
\end{equation}
The new model defines $m_j(y):=\min_x\Phi_{(\rho^{(j)},x^{(j)},\mu^{(j)})}(x,y)$ using $H^{(j)}$ and the explicit solve. If $m_j(y^{(j)})>\Upsilon^{(j-1)}$, we set $y_0^{(j)}=y_0^{(0)}$; otherwise $y_0^{(j)}=y^{(j)}$. We then set $\Upsilon^{(j)}=\max\{\Upsilon^{(j-1)},f(x^{(j)}),m_j(y_0^{(j)})\}$. These evaluations count against the budget; the new state is recorded only after they succeed.

\begin{remark}\label{rem:support-restart}\label{rem:restart-safeguards}
	A restart changes only the next initialization. Supports may change indefinitely, and bounded $\Upsilon^{(j)}$ does not imply finitely many restarts. Core stationarity requires accepted calls, outer regularity and, for unbounded $C$, radius inactivity; M-stationarity also requires joint AM-regularity. A core M-target uses~\eqref{eq:rgM-numerical}, or a verified upper bound, at the returned feasible point. Neither a small model residual nor a budget exit certifies this target; accuracy floors do not replace vanishing schedules.
\end{remark}

Successful calls do not rule out cycling; see
Example~\ref{app:unfiltered-cycle} in Section~\ref{sec:outerconvergence}.
The unfiltered core analysis therefore retains outer regularity as an explicit assumption. For the safeguarded method, the coupled descent estimate forces accepted displacements to vanish. Moreover, M-stationarity alone does not certify BF or projection stationarity, as the next example shows.

\begin{example}[A core M-target need not certify BF]\label{app:BF-example}
	The origin in Figure~\ref{fig:stationarity-counterexamples}(b) is M-stationary but fails BF and projection stationarity for every $\alpha>0$. It therefore cannot be an accumulation point of the safeguarded method under the assumptions of Theorem~\ref{thm:certified-stationarity}.
\end{example}

The artificial core radius can change stationarity at a boundary point;
see Example~\ref{app:radius-example} in Section~\ref{sec:outerconvergence}.
The core analysis accordingly retains the radius-inactivity hypothesis. The reference construction below obtains its certificate directly on $\Omega$, independently of whether the trial radius is active.

\subsection{Feasible reference and recovery}\label{sec:certified-safeguard}
Steps~\PDstep{0}{a.reference}--\PDstep{3}{a.reference} compute the reference
and its witness; Steps~\PDstep{1}{a.CPDQN}--\PDstep{2}{a.CPDQN} call this
procedure and check the stopping target. The safeguarded method maintains feasible $w^k$, a center $\xi^k$ and a multiplier $\mu^k$, with $\xi^0=w^0$, finite $\mu^0$ and the same safeguards. For fixed $\overline\alpha>0$, $\beta\in(0,1)$, $\sigma>0$, $\epsilon_{\rm stat}\ge0$, backtracking starts at $\alpha=\overline\alpha$ and computes
\begin{equation}\label{eq:reference-projection}p_\alpha^k\in P_\Omega(w^k-\alpha g(w^k)),\qquad
	f(p_\alpha^k)\le f(w^k)-\sigma\|p_\alpha^k-w^k\|^2.
\end{equation}
After each failure, $\alpha$ becomes $\beta\alpha$, and the projection is recomputed. For the first accepted pair $(\alpha_k,p^k)$, we evaluate
\begin{equation}\label{eq:reference-mapping}G_k:=(w^k-p^k)/\alpha_k,\quad
	\nu_k:=G_k-g(w^k),\quad E_k:=G_k+g(p^k)-g(w^k).
\end{equation}
If $\|E_k\|\le\epsilon_{\rm stat}$, the method returns $p^k$ and $(\alpha_k,E_k,\nu_k)$ before any trial. A zero target requests exact stationarity; a positive target gives Proposition~\ref{prop:reference-certificate}'s approximate certificate at $p^k$. Rejected reference evaluations and $g(p^k)$ are counted.
The witness $(E_k,\nu_k)$ certifies the reference point $p^k$.
Any stationarity certificate for a different feasible incumbent
must be verified at that point.

Steps~\PDstep{3}{a.CPDQN}--\PDstep{5}{a.CPDQN} prepare, test and accept or reject a finite trial. When the target has not passed, the trial freezes $(\xi^k,\mu^k,\rho^k,H^k)$, starts its projected block at $w^k$, and uses
\begin{equation}\label{eq:trial-radius}
	R_k:=1+\max\{\|w^k\|,\|\xi^k\|\},\qquad
	\Omega_k^{\rm trial}:=\Omega\cap\overline B_{R_k}(0).
\end{equation}
The reference has no artificial ball. Modified warm starts must lie in $\Omega_k^{\rm trial}$. Finite nonnegative-integer budgets $K_k^{\rm trial},B_k^{\rm trial}$ cover initialization, inner work, diagnostics, objective comparisons and curvature proposals through a shared ledger; either zero budget skips the trial. Invalid optional proposals retain valid data. Failed mandatory trial operations reject the trial; essential reference-oracle failures terminate with the current feasible incumbent and failure status.

The agreement anchor stays $(\xi^k,w^k)$. The final pair $(\widehat x^k,q^k)$ has an exact slice projection of $\widehat x^k-\mu^k/\rho^k$ and a stored size-$s$ set. Its residual $r_k:=\nabla_x\Phi_{(\rho^k,\xi^k,\mu^k)}(\widehat x^k,q^k)$ must satisfy $\|r_k\|\le\varepsilon_k$, together with
\begin{equation}\label{eq:certified-trial-agreement}
	\|\widehat x^k-q^k\|\le\tau\|\xi^k-w^k\|+\eta_{k+1}.
\end{equation}
The inner trial uses $\chi=0$; outer acceptance requires both the following comparisons:
\begin{align}
	f(q^k)&\le f(p^k),\label{eq:reference-domination}\\
	f(q^k)+\tfrac\sigma2\|q^k-w^k\|^2&\le f(p^k)+\tfrac\sigma2\|p^k-w^k\|^2.\label{eq:reference-envelope}
\end{align}
The same trial ledger covers $f(q^k)$; an unavailable value or unfinished test cannot count as acceptance. Objective domination bounds the trial value. The next example explains why we also need the envelope to control displacement.
\begin{example}[Why objective domination alone is insufficient]\label{ex:envelope-needed}
	We take $C=[0,2]^2$, $s=1$, $f(x)=\tfrac12[(x_1-1)^2+(x_2-1)^2]$, $w=\xi=(5/4,0)$, $\mu=(-1,0)$, $H=I$, $\rho=2$, $\sigma=1$ and $\overline\alpha=1/4$. The reference $p=(19/16,0)$ passes its test and has $E=(3/16,0)$; choose $\epsilon_{\rm stat}<3/16$. The warm start $y_0=e_2$ gives $\widehat x=q=e_2$ in one exact sweep: ranking proposes index $1$, its projection $(1/2,0)$ fails model descent, and fallback to index $2$ returns $e_2$. The final residual and disagreement are zero, but $f(q)=1/2\le f(p)=265/512$ passes domination while
	\[f(q)+\tfrac12\|q-w\|^2={57}/{32}>{133}/{256}=f(p)+\tfrac12\|p-w\|^2.\]
	The envelope rejects the trial; accepting it would violate coupled descent.
\end{example}
These two tests in Step~\PDstep{5}{a.CPDQN} yield the coupled descent estimate of Theorem~\ref{thm:certified-descent}. An accepted trial gives
\begin{equation}\label{eq:certified-PD-update}
	(w^{k+1},\xi^{k+1},\mu^{k+1})=(q^k,\widehat x^k,\mu^k-\rho^k(\widehat x^k-q^k));
\end{equation}
otherwise the method recovers by
\begin{equation}\label{eq:certified-recovery-update}
	(w^{k+1},\xi^{k+1},\mu^{k+1})=(p^k,p^k,g(p^k)).
\end{equation}
A failed trial therefore still completes an outer update. The method then refreshes the capped penalty and clipped matrix; unavailable curvature proposals leave the matrix unchanged. Recovery has its own multiplier reset. External exits return the current feasible point with its own diagnostics. The analysis assumes exact terminating essential oracles; floating-point tests require error control to give rigorous certificates.

For the data of Example~\ref{app:incompatible-filter}, with
$w^0=ae_1$, $\overline\alpha=1$, $\beta=1/2$ and $\sigma=1$,
the reference accepts $\alpha_0=1/2$ and gives $p^0=0$ and $E_0=0$.
The method therefore meets its stationarity target before
attempting any optional trial.

\subsection{Backtracking and extrapolation with closed-form trials}\label{sec:closed-form-search}
The optional scale search belongs to Step~\PDstep{3}{a.CPDQN}
and is distinct from the reference backtracking in
Step~\PDstep{1}{a.reference}. Before freezing a trial model, we may scale a safeguarded proposal $B=\operatorname{diag}(b_i)$. For fixed $z,\mu,\rho$ and feasible anchor $y_0$, we use
$\alpha\in[\alpha_-,\alpha_+]$, where
$\alpha_-=(\max_i b_i)/\lambda_{\max}\le1$ and
$\alpha_+=(\min_i b_i)/\lambda_{\min}\ge1$. For each candidate scale, the exact smooth-block point is
\begin{equation}\label{eq:closed-form-scale-path}
	u(\alpha)=(B+\alpha\rho I)^{-1}
	\bigl[Bz+\alpha(\rho y_0-g(z)+\mu)\bigr],\qquad H_\alpha=B/\alpha.
\end{equation}
Here $\alpha$ scales the model curvature and is independent of the reference stepsize $\alpha_k$. Every candidate is tested against the same quadratic $\Phi_B$, obtained from~\eqref{pr:pxy-subproblem} with $H=B$. Set
$r_0=g(z)-\mu+\rho(z-y_0)$ and choose $c_{\rm ls}\in(0,1)$. We test
\begin{equation}\label{eq:closed-form-scale-armijo}
	\Phi_B(u(\alpha),y_0)
	\le\Phi_B(z,y_0)-c_{\rm ls}\alpha r_0^TB^{-1}r_0.
\end{equation}
The trial can equivalently be written as
$u(\alpha)=z-\alpha(B+\alpha\rho I)^{-1}r_0$.
Thus \emph{extrapolation} means $\alpha>1$ in this curvature-scaled
closed-form family, rather than extrapolation along a fixed direction.
The search starts with $\alpha=1$. If $r_0=0$, return $\alpha=1$.
Otherwise, after a passing unit-scale test, try successively
$\alpha=2,4,8,\ldots$ (\emph{extrapolation}) and return the last
passing scale when the next test fails or a spectral bound or finite
search budget is reached. If the unit-scale test fails, try
$\alpha=1/2,1/4,1/8,\ldots$ (\emph{backtracking}) and return the
first passing scale.

Only candidates in $[\alpha_-,\alpha_+]$ are tested.
Consequently, the selected scale may satisfy $\alpha>1$, $\alpha=1$,
or $\alpha<1$. If no tested scale passes, the trial fails, and the main
method takes the reference-recovery step; a zero search budget skips
scaling and retains the clipped diagonal proposal. All tests use the
auxiliary quadratic $\Phi_B$, not the original objective $f$.

We then freeze the selected matrix $H=H_\alpha$ throughout the inner call. The bounds preserve the spectral safeguards, and
$g(z)+H(u(\alpha)-z)-\mu+\rho(u(\alpha)-y_0)=0$.
Thus, $u(\alpha)$ is the exact smooth-block update~\eqref{pr:px-subproblem} from $y_0$. Subsequent inner updates use the same frozen matrix, with the same final residual, consensus, and outer acceptance tests. The next example exhibits all three scale regimes and finite extrapolation.
\begin{example}[Short, unit and extrapolated closed-form trials]\label{ex:closed-form-scale-regimes}
	We take $C=[-1,1]^2$, $s=1$, $f(x)=\|x\|^2/2-x_1$, $z=y_0=\mu=0$, $B=I$, $(\lambda_{\min},\lambda_{\max})=(1/4,4)$ and $c_{\rm ls}=1/16$. Then $u(\alpha)=\alpha e_1/(1+\alpha\rho)$. With a search limit of at least three tests, the returned pairs $(\rho,\alpha)$ are $(1,2)$, $(4,1)$ and $(10,1/2)$. For $\rho=1$, scale $2$ passes and scale $4$ fails, so extrapolation terminates.
\end{example}

\section{Convergence analysis}\label{convergence}\label{sec:limitations-appendix}\label{secA1}

Algorithm~\ref{a.CPDQN} calls Algorithm~\ref{a.reference} in
Step~\PDstep{1}{a.CPDQN}, checks its witness in
Step~\PDstep{2}{a.CPDQN}, and accepts a trial or recovers in
Step~\PDstep{5}{a.CPDQN}. Section~\ref{sec:strong-convergence}
analyzes this complete method. Sections~\ref{sec:legacy-core-scope}
and~\ref{sec:outerconvergence} separately analyze consecutive accepted
calls of Algorithm~\ref{a.EPD-inner} in the core of Section~\ref{sec:zeta}.
Their identities across core updates are not applied across reference
recovery resets. Results for one frozen inner call also apply at
Step~\PDstep{4}{a.CPDQN}, with its call radius held fixed.

\begin{figure}[htbp]
	\centering
	\setlength{\tabcolsep}{7pt}
	\renewcommand{\arraystretch}{1.1}
	\begin{tabular}{@{}c@{\hspace{0.035\linewidth}}p{0.57\linewidth}@{}}
		\begin{tikzpicture}[x=1.05cm,y=1.05cm,>=Stealth,baseline=(current bounding box.center)]
			\draw[->,gray!60] (-1.75,0)--(1.85,0) node[below] {$x_1$};
			\draw[->,gray!60] (0,-1.30)--(0,1.65) node[left] {$x_2$};
			\draw[very thick,blue!65!black] (-1.55,0)--(1.6,0);
			\draw[very thick,blue!65!black] (0,-1.15)--(0,1.45);
			\draw[gray!55] (.50,1.0) circle[radius=.55];
			\draw[gray!55] (.50,1.0) circle[radius=1.0];
			\fill[black] (.5,0) circle[radius=2pt];
			\node[below right,font=\footnotesize] at (.5,0) {$\bar x$};
			\filldraw[fill=white,draw=black] (.5,1.0) circle[radius=2pt];
			\node[right,font=\footnotesize] at (.54,1.12) {$z_1$};
			\fill[orange!80!black] (0,1) circle[radius=2.2pt];
			\node[left,font=\footnotesize] at (0,1) {$P_\Omega(z_1)$};
			\draw[->,thick,dashed,red!75!black] (.5,.08)--(.5,.88);
			\node[font=\footnotesize,anchor=north west] at (-1.67,1.64) {\textbf{(a)}};
		\end{tikzpicture}
		&
		\begin{minipage}[c]{\linewidth}\raggedright\small
			\textbf{Projection stationarity is stronger (for fixed $\alpha$).}
			\par\smallskip
			$C=[-2,2]^2$, $s=1$, $f=\tfrac12(x_1-\tfrac12)^2+\tfrac12(x_2-1)^2$, and 
			$\bar x=(\tfrac12,0)$, $g(\bar x)=-e_2$. At $\alpha=1$,
			$z_1=\bar x-g(\bar x)=(\tfrac12,1)$ and
			$P_\Omega(z_1)=\{(0,1)\}\ne\{\bar x\}$.
			Yet $\bar x$ is BF, all-super-support S, intrinsic Fr\'echet, AM, and M stationary.
			For $\alpha=\tfrac14$, it \emph{is} projection-stationary.
		\end{minipage}
		\\[7pt]
		\begin{tikzpicture}[x=1.05cm,y=1.05cm,>=Stealth,baseline=(current bounding box.center)]
			\draw[->,gray!60] (-1.75,0)--(1.85,0) node[below] {$x_1$};
			\draw[->,gray!60] (0,-1.32)--(0,1.65) node[left] {$x_2$};
			\draw[very thick,blue!65!black] (-1,0)--(1,0);
			\draw[very thick,blue!65!black] (0,-1)--(0,1);
			\fill (0,0) circle[radius=2pt];
			\node[below left,font=\footnotesize] at (0,0) {$\bar x=0$};
			\draw[->,thick,red!80!black] (.10,-.05)--(.10,-.8);
			\node[right,font=\footnotesize] at (.13,-.71) {$g=-e_2$};
			\draw[->,thick,orange!90!black] (-.13,.04)--(-.13,.84);
			\node[left,font=\footnotesize] at (-.15,.71) {$v_s=e_2$};
			\node[font=\footnotesize,anchor=north west] at (-1.67,1.64) {\textbf{(b)}};
		\end{tikzpicture}
		&
		\begin{minipage}[c]{\linewidth}\raggedright\small
			\textbf{M/AM stationarity does not imply BF at incomplete support.}
			\par\smallskip
			$C=[-1,1]^2$, $s=1$,
			$f=\tfrac12x_1^2+\tfrac12(x_2-1)^2$, and $\bar x=0$.
			The split normal $v_s=e_2\in N_{C_1}(0)$ cancels $g(0)=-e_2$:
			M (hence AM) holds. But $g(0)^Te_2=-1<0$,
			so BF, all-super-support S, and intrinsic Fr\'echet stationarity fail; projection
			stationarity fails for every $\alpha>0$.
		\end{minipage}
		\\[7pt]
		\begin{tikzpicture}[x=1.05cm,y=1.05cm,>=Stealth,baseline=(current bounding box.center)]
			\draw[->,gray!60] (-.62,0)--(2.34,0) node[below] {$x_1$};
			\draw[->,gray!60] (0,-.57)--(0,2.42) node[left] {$x_2$};
			\draw[thick,blue!65!black] (1,1) circle[radius=1];
			\fill (1,0) circle[radius=2pt];
			\node[below left,font=\footnotesize] at (1,0) {$\bar x$};
			\fill (0,1) circle[radius=2pt];
			\node[left,font=\footnotesize] at (0,1) {$(0,1)$};
			\node[font=\footnotesize,anchor=north west] at (-.53,2.42) {\textbf{(c)}};
		\end{tikzpicture}
		&
		\begin{minipage}[c]{\linewidth}\raggedright\small
			\textbf{Without joint regularity, even a global minimizer need not be M-stationary.}
			\par\smallskip
			$C=\overline B_1((1,1))$, $s=1$, $f=-x_1$, $\bar x=(1,0)$, and 
			$\Omega=\{(1,0),(0,1)\}$. Thus $\bar x$ is
			projection-stationary for every $\alpha>0$, hence BF,
			Fr\'echet and AM. However,
			$N_C(\bar x)+N_{C_1}(\bar x)=\operatorname{span}\{e_2\}$
			does not contain $-g(\bar x)=e_1$: M and all-super-support S fail.
		\end{minipage}
	\end{tabular}
	\caption{Three strict separations among the stationarity notions of
		\texttt{CCOP}. Blue lines indicate the sparse feasible axes in panels~(a)--(b),
		and the blue circle is the convex set in panel~(c);
		red arrows indicate gradients and orange arrows indicate split normals. All coordinates and inequalities
		are given alongside the panels rather than superimposed on the drawings.}
	\label{fig:stationarity-counterexamples}
\end{figure}

\noindent\textbf{Stationarity hierarchy and role of the safeguard.}
Figure~\ref{fig:stationarity-counterexamples} illustrates the distinctions
in Proposition~\ref{prop:stationarity-relations}: projection stationarity
implies BF, intrinsic Fr\'echet stationarity and geometric AM-stationarity;
M-stationarity follows under joint AM-regularity, but alone need not imply
BF at incomplete support. Theorem~\ref{thm:certified-stationarity}
obtains a common-stepsize projection condition from the reference in
Step~\PDstep{1}{a.reference} and the vanishing outer displacements in
Step~\PDstep{5}{a.CPDQN}. Theorem~\ref{thm:all-support-S-BF}
is a pointwise geometric result: under its normal-cone sum rule,
Corollary~\ref{cor:certified-all-support-S} interprets the same limits
as all-super-support S-stationary. It does not upgrade a single
selected-slice core limit to universal BF.

\noindent\textbf{Two convergence routes.} The convergence results have distinct hypotheses. For the complete
safeguarded method, Theorems~\ref{thm:certified-descent} and
\ref{thm:certified-stationarity} rely on the feasible reference and the
outer acceptance tests; the optional trial may be rejected or skipped.
For the unrecovered core, the accepted-consensus route and the optional
boundedness route are alternative ways to obtain the bounds needed for
subsequent core analysis. In particular, the accepted-consensus route
of Lemma~\ref{lem:agreement-uniform-bound} does not require the optional
global gradient-growth and recurrence-balance assumptions. Frozen-inner
limits, accepted-core outer limits, and the objective-filter condition
are distinct assertions; see their individual statements for the exact
hypotheses. None by itself guarantees an infinite sequence of accepted
core calls in the complete safeguarded method.

\bfi{A recurring scalar estimate.} The agreement test in
Step~\PDstep{5}{a.EPD-inner} and the recovery update in
Step~\PDstep{5}{a.CPDQN} use the following estimate. If $d_k\ge0$,
$\eta_k\ge0$, $d_{k+1}\le\tau d_k+\eta_{k+1}$, $0<\tau<1$ and
$\eta_k\to0$, then
\begin{equation}\label{eq:scalar-consensus}
	d_k\le\tau^kd_0+\sum_{i=1}^k\tau^{k-i}\eta_i\to0,
	\qquad d_k\le d_0+\frac{\sup_i\eta_i}{1-\tau}.
\end{equation}
Iteration gives the first inequality. For any fixed $J$, split the
convolution at $J$: the weighted finite prefix tends to zero as
$k\to\infty$, while the tail is at most
$\sup_{i\ge J}\eta_i/(1-\tau)$, which tends to zero as $J\to\infty$.
The second inequality follows from the geometric-series bound.

\subsection{The complete safeguarded method}\label{sec:strong-convergence}
\begin{assumption}[Local smoothness and a compact initial level set]\label{ass:certified-main}
	The function $f$ is continuously differentiable on $\mathbb R^n$, and $g$ is Lipschitz on every bounded ball. The set $C$ is closed, convex and permutation-invariant, $\Omega=C\cap C_s$ is nonempty, and $w^0\in\Omega$. The initial feasible level set
	\begin{equation}\label{eq:initial-level-set}
		\mathcal K:=\{w\in\Omega:f(w)\le f(w^0)\}
	\end{equation}
	is compact. Exact projections in~\eqref{eq:reference-projection} and the declared function/gradient evaluations are available. Every core call is a terminating finite-budget trial procedure and retains the spectral and penalty safeguards; an unavailable or unsuccessful core oracle is reported as a failed trial. The trial parameters satisfy $0<\tau<1$, $\varepsilon_k>0\downarrow0$, and $\eta_k\ge0\downarrow0$. Optional initialization, diagnostic and curvature-proposal routines also terminate within declared finite budgets. An unavailable or invalid curvature proposal keeps the preceding safeguarded matrix, so those optional routines cannot block a reference recovery. Essential reference projections and the declared function/gradient evaluations are exact terminating oracles in this analysis.
\end{assumption}
A $C^2$ objective satisfies the smoothness requirement. Compact $C$, or coercivity on $\Omega$, suffices for compactness of $\mathcal K$. Closedness of an unbounded feasible set alone does not.

\begin{example}[Removing compactness can destroy finite-target attainment]\label{app:noncompact-level}
	For $C=\mathbb R^2$, $s=1$, $f(x)=-x_1$, $w^0=e_1$, $\overline\alpha=\sigma=1$ and zero trial budgets, each reference is $p^k=w^k+e_1$. Armijo holds by equality, $E_k=-e_1$ and $w^k=(k+1)e_1$. The initial level set is the unbounded ray $\{te_1:t\ge1\}$ with objective unbounded below; no target below $1$ is reached. This violates the main level-set assumption.
\end{example}
We therefore retain compactness of $\mathcal K$ in the main analysis. It keeps the iterates bounded and provides the objective lower bound and uniform smoothness estimates used in the finite-target proof. Put $F_{\min}=\min_{\mathcal K}f$, $\Delta=f(w^0)-F_{\min}$, $B=\max_{\mathcal K}\|w\|$, $M=\max_{\mathcal K}\|g(w)\|$, and let $L$ be a Lipschitz constant on $\overline B_{B+2\overline\alpha M+1}(0)$. These finite analysis constants need not be known by the method.

\begin{lemma}[Finite reference backtracking]\label{lem:reference-backtracking}
	At every $w\in\mathcal K$, the backtracking in Step~\PDstep{1}{a.reference} terminates. All accepted step sizes satisfy
	\begin{equation}\label{eq:reference-alpha-lower}
		0<\underline\alpha:=\beta\min\left\{\overline\alpha,\frac1{L+2\sigma}\right\}
		\le\alpha_k\le\overline\alpha.
	\end{equation}
	The number of sparse-projection calls per reference computation is uniformly finite.
\end{lemma}
\begin{proof}
	The nonempty closed set $\Omega$ has a closest point to every input. Since
	$w\in\mathcal K\subseteq\Omega$, comparing
	$p\in P_\Omega(w-\alpha g(w))$ with the feasible point $w$ gives
	\begin{equation}\label{eq:reference-projection-inequality}
		g(w)^T(p-w)+\frac{\|p-w\|^2}{2\alpha}\le0,
		\qquad\|p-w\|\le2\alpha\|g(w)\|.
	\end{equation}
	Since $\alpha\le\overline\alpha$, $\|w\|\le B$ and
	$\|g(w)\|\le M$, the points $w,p$ and their segment lie in the ball
	on which $g$ has Lipschitz constant $L$. The descent lemma therefore gives
	$f(p)\le f(w)-(1/(2\alpha)-L/2)\|p-w\|^2$.
	Thus, every $\alpha\le T:=1/(L+2\sigma)$ satisfies the reference
	decrease test, for any tied projection selection. Hence, the
	backtracking terminates. If the initial trial is rejected, the
	trial immediately preceding the accepted one must satisfy
	$\alpha>T$; otherwise, it would already pass. Therefore, the accepted
	step is larger than $\beta T$. If the initial trial is accepted, it
	equals $\overline\alpha$. This proves the stated lower bound. At most
	$Q_{\rm bt}=1+\min\{m\in\mathbb N_0:\beta^m\overline\alpha\le T\}$
	trials are needed.
\end{proof}

Steps~\PDstep{2}{a.reference}--\PDstep{3}{a.reference} return the
quantities certified below; Step~\PDstep{2}{a.CPDQN} tests their witness.

\begin{proposition}[A finite certificate at the returned reference point]\label{prop:reference-certificate}
	Under Assumption~\ref{ass:certified-main}, let $w\in\mathcal K$, let $\alpha\in(0,\overline\alpha]$, and let
	$p\in P_\Omega(w-\alpha g(w))$. Using the notation of~\eqref{eq:reference-mapping}, with the iteration
index suppressed, write
\begin{equation}\label{eq:finite-normal-witness}
	G:=\frac{w-p}{\alpha},\qquad
	\nu:=G-g(w),\qquad E:=G+g(p)-g(w).
\end{equation}
	Then, $\nu\in N_\Omega^F(p)$ and $E=g(p)+\nu$. More precisely,
	\begin{equation}\label{eq:finite-projection-certificate}
		p\in P_\Omega\bigl(p-\alpha g(p)+\alpha E\bigr),\qquad
		\|E\|\le(1+L\alpha)\|G\|.
	\end{equation}
	For every size-$s$ super-support $\mathcal L$ of $p$ and every
	$v\in C\cap S_{\mathcal L}$,
	\begin{equation}\label{eq:finite-BF-certificate}
		g(p)^T(v-p)\ge E^T(v-p)\ge-\|E\|\,\|v-p\|.
	\end{equation}
	If joint AM-regularity holds at $p$, there exist $u\in N_C(p)$ and
	$v_s\in N_{C_s}(p)$ such that $g(p)+u+v_s=E$. Consequently the geometric
	M-residual defined in~\eqref{eq:rgM-numerical} satisfies
	\begin{equation}\label{eq:finite-M-certificate}
		\mathrm{rg}_M(p)\le\frac{\|E\|}{\max\{1,\|g(p)\|\}}\le\|E\|.
	\end{equation}
	In particular, this bound is valid for every such $p$ when $C$ is
	polyhedral or $0\in\operatorname{int}C$. An exact zero $E=0$ certifies
	$\alpha$-stationarity of $p$, and hence BF and geometric AM-stationarity;
	geometric M-stationarity follows under the same regularity qualification.
	The pointwise regularity in the finite M bound is essential even when
	$E=0$, as Remark~\ref{rem:convex-symmetry-insufficient} demonstrates.
\end{proposition}
\begin{proof}
	Expanding the global distance comparison and using $w-\alpha g(w)-p=\alpha\nu$ gives
	\begin{equation}\label{eq:finite-normal-inequality}
		\nu^T(v-p)\le\|v-p\|^2/(2\alpha)\qquad\forall v\in\Omega.
	\end{equation}
	Dividing by $\|v-p\|$ and letting $v\to p$ in $\Omega$ proves $\nu\in N_\Omega^F(p)$, with the isolated-point convention understood. The identities $E=g(p)+\nu$ and $p-\alpha g(p)+\alpha E=w-\alpha g(w)$ prove the perturbed projection claim. By~\eqref{eq:reference-projection-inequality}, $\|p-w\|\le2\alpha\|g(w)\|$, so the segment from $w$ to $p$ lies in the ball defining $L$, even without reference acceptance. Thus,
	$\|E\|\le\|G\|+L\|p-w\|=(1+L\alpha)\|G\|$.
	For any super-support and $v$ in its convex slice, substitute $p+t(v-p)$ into~\eqref{eq:finite-normal-inequality} and let $t\downarrow0$ after division by $t$. This gives $\nu^T(v-p)\le0$, and hence, since $E=g(p)+\nu$,~\eqref{eq:finite-BF-certificate}.
	For the split assertion, the linear objective $-\nu^Tx$ is intrinsically Fr\'echet-stationary at $p$. Proposition~\ref{prop:stationarity-relations} and pointwise AM-regularity give $\nu=u+v_s$ with the required separate normals. By~\eqref{eq:sparse-normal-compact}, choose a super-support $\mathcal L$ of $p$ such that $v_s\in S_{\mathcal L}^\perp$. Then,  $(g(p)+u)_{\mathcal L}=E_{\mathcal L}$, proving the residual bound without presuming attainment of an arbitrary zero infimum. The two regularity classes follow from the preceding lemmas. At $E=0$, the perturbed projection is exactly~\eqref{eq:alpha-stationarity} at $p$.
\end{proof}

\begin{theorem}[Execution, true descent and vanishing steps]\label{thm:certified-descent}
	Under Assumption~\ref{ass:certified-main}, Algorithm~\ref{a.CPDQN} is well-defined in exact arithmetic. With no external budget interruption, it either returns a feasible reference point with its requested stationarity witness or continues to generate feasible iterates, regardless of the number of unsuccessful core calls. Every completed update in Step~\PDstep{5}{a.CPDQN} satisfies
	\begin{align}
		f(w^{k+1})&\le f(p^k)\le f(w^k)-\sigma\|p^k-w^k\|^2,
		\label{eq:certified-reference-progress}\\
		f(w^k)-f(w^{k+1})&\ge\frac\sigma2\left(\|p^k-w^k\|^2+\|w^{k+1}-w^k\|^2\right).
		\label{eq:certified-coupled-descent}
	\end{align}
	In particular, $w^k\in\mathcal K$. On an infinite sequence, $f(w^k)$ converges, $\|G_k\|\to0$, $\|E_k\|\to0$, and
	\begin{equation}\label{eq:certified-summability}
		\sum_{k\ge0}\|p^k-w^k\|^2\le\frac\Delta\sigma,
		\qquad
		\sum_{k\ge0}\left(\|p^k-w^k\|^2+\|w^{k+1}-w^k\|^2\right)\le\frac{2\Delta}\sigma.
	\end{equation}
\end{theorem}
\begin{proof}
	Induct on $w^k\in\mathcal K$. Backtracking terminates and optional work terminates within its budget. Acceptance produces a feasible $q^k$; otherwise recovery produces $p^k$. Domination and the reference test give~\eqref{eq:certified-reference-progress}. The envelope, also an equality on recovery, gives
	\[
	f(w^{k+1})+\tfrac\sigma2\|w^{k+1}-w^k\|^2
	\le f(p^k)+\tfrac\sigma2\|p^k-w^k\|^2
	\le f(w^k)-\tfrac\sigma2\|p^k-w^k\|^2.
	\]
	This proves coupled descent and, together with feasibility, keeps the next point in $\mathcal K$, completing the induction. Since $f(w^k)\ge F_{\min}$, the decreasing objective sequence converges to a finite limit. Summing~\eqref{eq:certified-reference-progress} and~\eqref{eq:certified-coupled-descent} separately gives the two bounds in~\eqref{eq:certified-summability}. Hence $\|p^k-w^k\|\to0$. Since
	$G_k=(w^k-p^k)/\alpha_k$ and $\alpha_k\ge\underline\alpha>0$, we have $G_k\to0$, and Proposition~\ref{prop:reference-certificate} gives $E_k\to0$. A finite target exit is certified at its actual $p^k$ by that proposition.
\end{proof}

\begin{proposition}[Bounded auxiliary variables and multiplier consistency]\label{prop:certified-auxiliary}
	Under the same assumptions, all accepted auxiliary centers, multipliers, trial radii and exact core block iterates of Algorithm~\ref{a.CPDQN}, generated through Steps~\PDstep{4}{a.CPDQN}--\PDstep{6}{a.CPDQN}, are bounded. Discarded points generated by optional heuristics are not included. On an infinite sequence,
	\[
	\|\xi^k-w^k\|\to0,\qquad
	\|\xi^{k+1}-\xi^k\|\to0,\qquad
	\|\mu^k-g(w^k)\|\to0.
	\]
	No global gradient-growth condition or auxiliary recurrence balance is required.
\end{proposition}
\begin{proof}
	For $d_k=\|\xi^k-w^k\|$, an accepted PD step gives the recursion~\eqref{eq:scalar-consensus}, while recovery sets $d_{k+1}=0$. Thus $d_k\to0$ and $\|\xi^k\|\le B_\xi:=B+d_0+\eta_* /(1-\tau)$, where $\eta_*:=\sup_{j\ge1}\eta_j$. Put $M_\xi=\max_{\|x\|\le B_\xi}\|g(x)\|$ and $\varepsilon_*:=\sup_{j\ge0}\varepsilon_j$. On a PD step, its measured residual and multiplier update give
	\begin{equation}\label{eq:certified-residual-identity}
		\mu^{k+1}=g(\xi^k)+H^k(\xi^{k+1}-\xi^k)-r_k.
	\end{equation}
	Its norm is at most $M_\xi+2\lambda_{\max}B_\xi+\varepsilon_*$; recovery gives $\mu^{k+1}=g(w^{k+1})$. Including the finite initial multiplier, since $R_k\le1+B_\xi$, all trial sparse-block iterates
	$y_\ell\in D_{R_k}$ are bounded. Together with the bounded
	multipliers and safeguarded model parameters, the exact formula
	\eqref{pr:px-subproblem} implies that the smooth-block solutions
	$x_\ell$ are bounded. Consequently, continuity of $g$ and the
	spectral safeguards also bound $g(x_\ell)$ and the corresponding
	ranking vectors. Moreover $\|\xi^{k+1}-\xi^k\|\le d_{k+1}+\|w^{k+1}-w^k\|+d_k\to0$.
	Also $\|\xi^k-w^{k+1}\|\le d_k+\|w^{k+1}-w^k\|\to0$. On PD steps,
	subtracting $g(w^{k+1})$ in~\eqref{eq:certified-residual-identity},
	using uniform continuity of $g$ on the bounded set, the vanishing
	auxiliary step, and $\|r_k\|\le\varepsilon_k\to0$ gives multiplier
	consistency; recovery gives it exactly.
\end{proof}

\begin{theorem}[Stationarity for every accumulation point]\label{thm:certified-stationarity}
	Under Assumption~\ref{ass:certified-main}, every accumulation point $\bar w$ of an infinite sequence generated by Algorithm~\ref{a.CPDQN} satisfies the common-stepsize projection condition, with $\underline\alpha$ defined in~\eqref{eq:reference-alpha-lower},
\begin{equation}\label{eq:certified-limit-fixed-point}
	\bar w\in P_\Omega\bigl(\bar w-\underline\alpha g(\bar w)\bigr).
\end{equation}
	It is therefore BF-stationary on \emph{every} super-support, Fr\'echet-stationary for the actual intersection $\Omega$, and geometrically AM-stationary. If joint AM-regularity~\eqref{eq:geometric-AM-regularity} holds at $\bar w$, it is geometrically M-stationary. In particular, the M conclusion holds for polyhedral $C$ and for closed convex $C$ with $0\in\operatorname{int}C$. These conclusions concern the original unbounded problem as well when its initial feasible level set is compact; no inactive artificial-radius hypothesis is needed for Algorithm~\ref{a.CPDQN}.
\end{theorem}
\begin{proof}
	Along any cluster subsequence $w^{k_j}\to\bar w$, Theorem~\ref{thm:certified-descent} gives $\|p^{k_j}-w^{k_j}\|\to0$, hence $p^{k_j}\to\bar w$. After extraction, $\alpha_{k_j}\to\alpha_*\in[\underline\alpha,\overline\alpha]$. For each fixed $v\in\Omega$, the exact projection property
	$p^{k_j}\in P_\Omega(w^{k_j}-\alpha_{k_j}g(w^{k_j}))$ gives, after expanding the distance comparison and passing to the limit,
	\begin{equation}\label{eq:limit-quadratic-inequality}
		g(\bar w)^T(v-\bar w)+\frac{\|v-\bar w\|^2}{2\alpha_*}\ge0\quad(v\in\Omega).
	\end{equation}
	Since $\underline\alpha\le\alpha_*$, replacing $\alpha_*$ by $\underline\alpha$ increases the left side. The resulting inequality is equivalent to~\eqref{eq:certified-limit-fixed-point}. Proposition~\ref{prop:stationarity-relations} gives BF, intrinsic Fr\'echet and AM stationarity; Theorem~\ref{thm:AM-to-M} and the regularity lemmas give the qualified M conclusion. Only projections on $\Omega$ were used, so no trial-ball normal enters. No continuity or uniqueness of the sparse projector is required.
\end{proof}

\begin{corollary}[A finite stationarity target and oracle bound]\label{cor:certified-complexity}
Under Assumption~\ref{ass:certified-main}, put $S:=1+L\overline\alpha$ and recall $\underline\alpha$ from~\eqref{eq:reference-alpha-lower}. For every $N\ge1$ completed updates in Step~\PDstep{5}{a.CPDQN},
\begin{align}
\min_{0\le k<N}\|G_k\|^2
&\le\frac{\Delta}{\sigma\underline\alpha^2N},
\label{eq:certified-complexity}\\
\min_{0\le k<N}\|E_k\|^2
&\le\frac{S^2\Delta}{\sigma\underline\alpha^2N}.
\label{eq:certified-witness-complexity}
\end{align}
A positive target $\epsilon_{\rm stat}$ in Step~\PDstep{2}{a.CPDQN} is
reached after at most
\begin{equation}\label{eq:certified-target-count}
\left\lceil\frac{S^2\Delta}
{\sigma\underline\alpha^2\epsilon_{\rm stat}^2}\right\rceil+1
\end{equation}
reference computations, unless an explicit external budget or essential
oracle failure interrupts the run. The returned point is the reference
projection with its actual witness, so Proposition~\ref{prop:reference-certificate}
also gives the geometric M-residual bound for the stated regularity classes. If the trial budgets $K_k^{\rm trial}$ and $B_k^{\rm trial}$, defined in Section~\ref{sec:certified-safeguard}, have uniform finite caps and the diagonal proposal has uniformly bounded oracle work, all independently of $k$ and the requested tolerance, the numbers of core updates and function, gradient and projection-oracle calls are
$O(\epsilon_{\rm stat}^{-2})$ for fixed problem data, starting point and
reference parameters. These are exact-oracle bounds; the convex
slice cost is representation-dependent, with at most $s+1$ slices
per sparse reference projection as in Remark~\ref{rem:exact-global-projection}.
\end{corollary}
\begin{proof}
Summing~\eqref{eq:certified-reference-progress}, using
$\|p^k-w^k\|=\alpha_k\|G_k\|$, $\alpha_k\ge\underline\alpha$, and
$f(w^N)\ge F_{\min}$ gives
$\sigma\underline\alpha^2\sum_{k<N}\|G_k\|^2\le\Delta$.
Since Proposition~\ref{prop:reference-certificate} and
$\alpha_k\le\overline\alpha$ give $\|E_k\|\le S\|G_k\|$, averaging proves both minimum bounds. If the first $N$ tests all fail, they lead to $N$ completed updates and $\sum_{k<N}\|E_k\|^2>N\epsilon_{\rm stat}^2$. Taking $N$ equal to the bound in~\eqref{eq:certified-target-count} contradicts the summed bound. If $\Delta=0$, reference descent already forces the first displacement and witness to vanish.
Each reference has at most $Q_{\rm bt}$ projections and objective calls, one current gradient and at most one additional accepted-point gradient. Uniform optional caps bound the remaining work independently of the target. For $N_{\rm ref}$ computations, including the terminating one, Remark~\ref{rem:exact-global-projection} gives at most $(s+1)Q_{\rm bt}N_{\rm ref}$ reference slice solves, or $Q_{\rm bt}N_{\rm ref}$ in the two one-slice subclasses. Substituting the target count gives $O((s+1)\epsilon_{\rm stat}^{-2})$ reference slice work plus sorting and comparisons; arbitrary slice arithmetic cost remains explicit. Count caching and optional calls according to actual work.
\end{proof}

\begin{corollary}[Whole-sequence convergence under isolation]\label{cor:certified-whole-sequence}
Under Assumption~\ref{ass:certified-main}, let $f_\infty=\lim_k f(w^k)$ for an infinite run of Algorithm~\ref{a.CPDQN}. With $\underline\alpha$ defined in~\eqref{eq:reference-alpha-lower}, if
\[
\mathcal Z_\infty:={w\in\mathcal K:f(w)=f_\infty,
\ w\in P_\Omega(w-\underline\alpha g(w))}
\]
is finite, then $w^k$ and $\xi^k$ converge to one point of $\mathcal Z_\infty$, and $\mu^k$ converges to its gradient. This conclusion does not assert a finite-length property.
\end{corollary}
\begin{proof}
Every cluster point lies in the stated finite nonempty set by compactness, objective continuity and Theorem~\ref{thm:certified-stationarity}; the distance to that set tends to zero, since a contrary subsequence would yield a cluster point outside it. If $\mathcal Z_\infty$ has more than one point, let $d>0$ be its minimum point separation. Eventually all iterates lie in the union of its radius-$d/4$ balls and, by~\eqref{eq:certified-summability}, consecutive steps are shorter than $d/2$. Since two distinct such balls are separated by at least $d/2$, the iterates can no longer change balls. Hence only one cluster point remains, and the bounded sequence converges to it. The one-point case is immediate. Proposition~\ref{prop:certified-auxiliary} gives $\|\xi^k-w^k\|\to0$ and $\|\mu^k-g(w^k)\|\to0$; continuity of $g$ therefore gives the stated auxiliary and multiplier limits. Square-summable steps alone do not imply finite length.
\end{proof}
\subsection{Conditional core analysis}\label{sec:legacy-core-scope}
Here Algorithm~\ref{a.EPD-inner} is called repeatedly on the fixed
$C_\zeta$, with successful calls accepted in
Step~\PDstep{5}{a.EPD-inner} and the outer state updated as in
Section~\ref{sec:zeta}. The bounds concern this unrecovered accepted-call
trajectory, not the recovery branch of Step~\PDstep{5}{a.CPDQN}.
They assume successful calls and do not establish eventual acceptance.

\begin{assumption}[Safeguarded Hessian approximations]\label{assumption_hessian}
	There are $0<\lambda_{\min}\le\lambda_{\max}<\infty$ such that
	$\lambda_{\min}I\preceq H^{(j)}\preceq\lambda_{\max}I$ for every $j$.
\end{assumption}
The capped penalty update enforces
\begin{equation}\label{eq:rhobd}
	0<\rho_{\min}\le\rho^{(j)}\le\rho_{\max}<\infty.
\end{equation}
Lemma~\ref{lem:agreement-uniform-bound} uses the agreement test in
Step~\PDstep{5}{a.EPD-inner}. The next two assumptions instead provide
the alternative growth-based bounds of Lemma~\ref{lem:uniform-bound}
and Corollary~\ref{cor:autbd} for the smooth-block updates in
Step~\PDstep{2}{a.EPD-inner}; they are not needed for
Theorem~\ref{thm:outer-original-CCM}.
\begin{assumption}[Optional global gradient growth]\label{assGrad}
	The gradient $g$ is continuous.  With $\zeta>0$ defined by
	\eqref{eq:zeta-definition}, there
	are $\gamma>0$ and $c>0$ such that
	\[
	\|g(x)\|\le
	\begin{cases}\gamma,&\|x\|\le c\zeta,\\
		\gamma\|x\|,&\|x\|>c\zeta.
	\end{cases}
	\]
	In particular $\|g(x)\|\le\gamma(1+\|x\|)$ for all $x$.
\end{assumption}
This is linear growth, not pairwise Lipschitz continuity. For bounded $C$, $\zeta=0$ means $C=\{0\}$; accepted consensus gives $x^{(j)}\to0$ and $N_C(0)=\mathbb R^n$ gives M-stationarity. Only the following optional recurrence requires $\zeta>0$. With $\varepsilon_*=\sup_j\varepsilon_j$, put
\begin{equation}\label{eq:bar-theta-kappa-general}
	\begin{gathered}
		D=\lambda_{\min}+\rho_{\min},\quad a=(2\lambda_{\max}+\gamma)/D,
		\quad b=(\lambda_{\max}+\gamma)/D,\\
		\bar\theta=(a+\sqrt{a^2+4b})/2,\qquad
		\bar\kappa=(\rho_{\max}\zeta+2\gamma+\varepsilon_*)/(D\zeta).
	\end{gathered}
\end{equation}
\begin{assumption}[Optional recurrence balance]\label{assumption_penalty}
	The penalties satisfy \eqref{eq:rhobd}, and the constants in
	\eqref{eq:bar-theta-kappa-general} obey $a+b<1$ (equivalently
	$0<\bar\theta<1$).  The gradient-growth threshold $c$ satisfies
	\begin{equation}\label{eq:balanceCon}
		\bar\kappa<c(1-\bar\theta).
	\end{equation}
	The condition $a+b<1$ is algebraically equivalent to
	$\rho_{\min}>3\lambda_{\max}+2\gamma-\lambda_{\min}$.
\end{assumption}
For $C^2$ objectives, the growth hypothesis is equivalent, as an existence condition with different constants, to
\begin{equation}\label{eq:origin-anchored-growth}
	\|g(x)-g(0)\|\le L_0\|x\|\quad(x\in\mathbb R^n)
\end{equation}
for finite $L_0$. A bounded Hessian near zero gives the anchored bound there; away from zero, use linear growth divided by $\|x\|$. Conversely the anchored bound gives $\|g(x)\|\le\|g(0)\|+L_0\|x\|$, hence the piecewise growth condition after enlarging $\gamma$. This equivalence does not concern the penalty balance. Example~\ref{app:growth-example} distinguishes it from global Lipschitz continuity and explains the $C^1$ qualification.

We next distinguish linear gradient growth from anchored and global Lipschitz bounds.
\begin{example}[Growth and anchored versus global Lipschitz bounds]\label{app:growth-example}
	For $f(t)=(1-\cos t^2)/2$, $g(t)=t\sin t^2$ obeys $|g(t)-g(0)|\le|t|$, while $g'(t)=\sin t^2+2t^2\cos t^2$ satisfies $g'(\sqrt{2\pi k})=4\pi k$ for $k\ge1$. Thus $g$ satisfies an anchored Lipschitz bound at zero but is not globally Lipschitz. Under merely $C^1$ regularity, $f(t)=2|t|^{3/2}/3$ gives $|g(t)|=\sqrt{|t|}\le1+|t|$, but $|g(t)-g(0)|/|t|\to\infty$ near zero. Hence linear gradient growth need not imply an anchored Lipschitz bound.
\end{example}

\begin{lemma}[Uniform bounds from accepted consensus]\label{lem:agreement-uniform-bound}
	Assume the spectral safeguard,~\eqref{eq:rhobd}, continuity of $g$, and an accepted sequence satisfying~\eqref{e.gnorm} and~\eqref{eq:primalDual}. Let
	\[
	d_j=\|x^{(j)}-y^{(j)}\|,\qquad
	\eta_*:=\sup_{j\ge1}\eta_j,\qquad\varepsilon_*:=\sup_{j\ge0}\varepsilon_j.
	\]
	Then $d_j\to0$, the primal and multiplier sequences are uniformly bounded, and one may use
	\begin{align}
		M_x^{\rm acc}&:=\zeta+d_0+\frac{\eta_*}{1-\tau},\label{eq:accepted-Mx}\\
		M_g^{\rm acc}&:=\max_{\|x\|\le M_x^{\rm acc}}\|g(x)\|,\nonumber\\
		M_\mu^{\rm acc}&:=\max\{\|\mu^{(0)}\|,M_g^{\rm acc}+2\lambda_{\max}M_x^{\rm acc}+\varepsilon_*\}\label{eq:accepted-Mmu}
	\end{align}
	as norm bounds. Moreover, with the measured residual $r^{(j)}$,
	\begin{equation}\label{eq:residual-multiplier-identity}
		\mu^{(j)}=g(x^{(j-1)})+H^{(j-1)}(x^{(j)}-x^{(j-1)})-r^{(j)},
		\quad\|r^{(j)}\|\le\varepsilon_{j-1}.
	\end{equation}
	In particular, $\limsup_j\|x^{(j)}\|\le\zeta$. The conclusion holds also when $\zeta=0$ and does not require Assumptions~\ref{assGrad}--\ref{assumption_penalty}.
\end{lemma}
\begin{proof}
	Apply~\eqref{eq:scalar-consensus} to $d_j$. Then
	$d_j\to0$ and
	$d_j\le d_0+\eta_* /(1-\tau)$. Since
	$\|y^{(j)}\|\le\zeta$,
	\[
	\|x^{(j)}\|\le\|y^{(j)}\|+d_j
	\le\zeta+d_0+\frac{\eta_*}{1-\tau}
	=M_x^{\rm acc}.
	\]
	Moreover, $\|x^{(j)}\|\le\zeta+d_j$ and $d_j\to0$, so
	$\limsup_j\|x^{(j)}\|\le\zeta$, including the case $\zeta=0$.
	Continuity of $g$ on the ball $\{x:\|x\|\le M_x^{\rm acc}\}$ gives
	the finite bound $M_g^{\rm acc}$.

	With the measured residual $r^{(j)}$, we have
	\[
	r^{(j)}
	=g(x^{(j-1)})+H^{(j-1)}(x^{(j)}-x^{(j-1)})
	-\mu^{(j-1)}
	+\rho^{(j-1)}(x^{(j)}-y^{(j)}).
	\]
	Using
	$\mu^{(j)}=\mu^{(j-1)}
	-\rho^{(j-1)}(x^{(j)}-y^{(j)})$
	gives~\eqref{eq:residual-multiplier-identity}.
	Hence, using the spectral safeguard and
	$\|x^{(j)}-x^{(j-1)}\|\le2M_x^{\rm acc}$, gives $\|\mu^{(j)}\|
	\le M_g^{\rm acc}+2\lambda_{\max}M_x^{\rm acc}+\varepsilon_*$. Together with the finite initial multiplier, this gives
	$M_\mu^{\rm acc}$. No growth at infinity was used.
\end{proof}

\begin{lemma}[Uniform bounds and model minima]\label{lem:uniform-bound}
	Under Assumptions~\ref{assumption_hessian}--\ref{assumption_penalty},
	assume each outer iteration returns a pair satisfying \eqref{e.gnorm}.
	Let $X_j:=\|x^{(j)}\|$ and
	$W_j:=\max\{X_j,\bar\theta X_{j-1}\}$, $j\ge1$.
	Then, for $j\ge2$,
	\begin{equation}\label{eq:recurrence-general}
		W_j\le\bar\theta W_{j-1}+\bar\kappa\zeta.
	\end{equation}
	Consequently $\{x^{(j)}\}$ is uniformly bounded, and there exists
	$\Upsilon^{\mathrm{exact}}<\infty$ such that
	\begin{equation}\label{eq:uniform-exact-model-bound}
		\min_{x\in\mathbb R^n}
		\Phi_{(\rho^{(j)},x^{(j)},\mu^{(j)})}(x,y_0^{(j)})
		\le\Upsilon^{\mathrm{exact}}\qquad(j\ge0).
	\end{equation}
\end{lemma}
\begin{proof}
	The last exact solve uses a preceding projection of norm at most $\zeta$. Using the residual identity~\eqref{eq:residual-multiplier-identity}, whose derivation does not use~\eqref{eq:primalDual}, at $j-1$ in that solve gives, for $j\ge2$,
	\[
	DX_j\le\lambda_{\max}X_{j-1}+\rho_{\max}\zeta
	+\|g(x^{(j-1)})\|+\|g(x^{(j-2)})\|
	+\lambda_{\max}(X_{j-1}+X_{j-2})+\varepsilon_*.
	\]
	Using $\|g(x)\|\le\gamma(1+\|x\|)$ and the definitions in~\eqref{eq:bar-theta-kappa-general} gives
	$X_j\le aX_{j-1}+bX_{j-2}+\bar\kappa\zeta$.
	Since $W_{j-1}\ge X_{j-1}$ and
	$W_{j-1}\ge\bar\theta X_{j-2}$, $X_j\le(a+b/\bar\theta)W_{j-1}+\bar\kappa\zeta
	=\bar\theta W_{j-1}+\bar\kappa\zeta$, where $\bar\theta^2=a\bar\theta+b$. Also
	$\bar\theta X_{j-1}\le\bar\theta W_{j-1}$.
	Taking their maximum proves~\eqref{eq:recurrence-general}. Iterating gives
	\[
	W_j\le\bar\theta^{j-1}W_1
	+\frac{\bar\kappa\zeta(1-\bar\theta^{j-1})}{1-\bar\theta}.
	\]
	The initial values are finite by the nonsingular solve. Hence
	$M_x=\max\{X_0,W_1,\bar\kappa\zeta/(1-\bar\theta)\}$ bounds all $X_j$, and~\eqref{eq:residual-multiplier-identity} bounds $\mu^{(j)}$ by
	$M_\mu=\max\{\|\mu^{(0)}\|,\gamma(1+M_x)+2\lambda_{\max}M_x+\varepsilon_*\}$.
	Finally, the $x$-subproblem is strongly convex, so its minimum is finite. Evaluating the model at $x=x^{(j)}$ and using
	$\|x^{(j)}-y_0^{(j)}\|\le M_x+\zeta$ bounds this minimum above by
	$M_\mu(M_x+\zeta)+\rho_{\max}(M_x+\zeta)^2/2$, which is an admissible $\Upsilon^{\rm exact}$.
\end{proof}

\begin{corollary}[Autonomous boundedness]\label{cor:autbd}
	Under Lemma~\ref{lem:uniform-bound},
	\begin{equation}\label{eq:autbd}
		\limsup_{j\to\infty}\|x^{(j)}\|\le R_\infty,
	\end{equation}
	where
	\begin{equation}\label{eq:R-infty-bound}
		R_\infty:=\frac{\bar\kappa\zeta}{1-\bar\theta}<c\zeta.
	\end{equation}
	In particular, $\{x^{(j)}\}$ eventually enters and remains in
	$B_{c\zeta}(0)$.
\end{corollary}
\begin{proof}
	The preceding estimate and $X_j\le W_j$ give
	$X_j\le R_\infty+\bar\theta^{j-1}(W_1-R_\infty)_+$. The last term tends to zero, proving~\eqref{eq:autbd}. Moreover,
	$R_\infty<c\zeta$ by the balance condition, so the last term
	eventually stays below $(c\zeta-R_\infty)/2>0$. Hence, 
	$X_j<c\zeta$ for all sufficiently large $j$.
\end{proof}

\begin{proposition}[Uniform multipliers and model safeguards]
	\label{prop:uniform-penalty-bound}
	Suppose either Lemma~\ref{lem:uniform-bound} or
	Lemma~\ref{lem:agreement-uniform-bound} applies. If
	Lemma~\ref{lem:agreement-uniform-bound} applies, take
	\[
	M_x:=M_x^{\rm acc},\qquad
	M_g:=M_g^{\rm acc},\qquad
	M_\mu:=M_\mu^{\rm acc}.
	\]
	If instead Lemma~\ref{lem:uniform-bound} applies, take
	\[
	M_x:=\max\left\{X_0,W_1,
	\frac{\bar\kappa\zeta}{1-\bar\theta}\right\},
	\qquad
	M_g:=\max_{\|x\|\le M_x}\|g(x)\|,
	\]
	and $M_\mu:=\max\{\|\mu^{(0)}\|,
	M_g+2\lambda_{\max}M_x+\varepsilon_*\}$.  Under the optional growth assumption, one may use
	$M_g=\gamma(1+M_x)$ as an upper bound in the latter case.
	Then
	\begin{equation}\label{eq:mu-uniform-bound}
		\|\mu^{(j)}\|\le M_\mu\qquad(j\ge0).
	\end{equation}
	Moreover, with $D_0:=M_x+\zeta$, the constant
	\[
	\Upsilon:=M_\mu D_0+\tfrac12\rho_{\max}D_0^2
	+\frac{\varepsilon_*^2}{2(\lambda_{\min}+\rho_{\min})}
	\]
	majorizes every
	$\Phi_{(\rho^{(j)},x^{(j)},\mu^{(j)})}(x^{(j)},y^{(j)})$
	and every minimum in~\eqref{eq:uniform-exact-model-bound}.
	The safeguards $\{\Upsilon^{(j)}\}$ used for the next initialization
	in Step~\PDstep{1}{a.EPD-inner} are also bounded above.
\end{proposition}
\begin{proof}
	In either case, the selected constants satisfy
	\[
	\|x^{(j)}\|\le M_x,\qquad
	\|g(x^{(j)})\|\le M_g,\qquad
	\|\mu^{(j)}\|\le M_\mu.
	\]
	The last bound follows directly from
	Lemma~\ref{lem:agreement-uniform-bound} in the accepted-consensus
	case, and from~\eqref{eq:residual-multiplier-identity} together with
	$\|x^{(j)}-x^{(j-1)}\|\le2M_x$ in the growth-based case.
	Every generated $y$ satisfies $\|y\|\le\zeta$, and hence
	$\|x^{(j)}-y\|\le D_0$. Evaluating the frozen model at its center
	$x=x^{(j)}$ gives
	\[
	\Phi_{(\rho^{(j)},x^{(j)},\mu^{(j)})}(x^{(j)},y)
	=-\mu^{(j)T}(x^{(j)}-y)
	+\tfrac{\rho^{(j)}}2\|x^{(j)}-y\|^2
	\le M_\mu D_0+\tfrac12\rho_{\max}D_0^2.
	\]
	Since the minimum over $x$ is no larger than its value at
	$x=x^{(j)}$, this bounds both the asserted center values and the
	minima in~\eqref{eq:uniform-exact-model-bound}. The additional
	residual allowance in $\Upsilon$ is nonnegative. Continuity also gives
	$M_f:=\max_{\|x\|\le M_x}f(x)<\infty$. Since the initialization
	update uses
	$\Upsilon^{(j)}
	=\max\{\Upsilon^{(j-1)},f(x^{(j)}),m_j(y_0^{(j)})\}$,
	the preceding bounds give
	\[
	\Upsilon^{(j)}
	\le\max\{\Upsilon^{(0)},M_f,\Upsilon\}
	\]
	for $j\ge0$ by induction.
\end{proof}

The next bound concerns the quadratic part of the frozen model in Step~\PDstep{2}{a.EPD-inner}; it is not a true-objective descent estimate.

\begin{lemma}[Lower bound on the quadratic Taylor increment]
	\label{lem:bounded_below_quadratic}
	Suppose the hypotheses of Proposition~\ref{prop:uniform-penalty-bound} hold, and use its finite gradient bound $M_g$. The sequence
	\begin{equation}\label{e.seqconvex-general}
		\Big\{(x^{(j)}-x^{(j-1)})^Tg(x^{(j-1)})+
		\tfrac12(x^{(j)}-x^{(j-1)})^TH^{(j-1)}
		(x^{(j)}-x^{(j-1)})\Big\}_{j\ge1}
	\end{equation}
	is uniformly bounded below by
	$\widehat c_{\mathrm{quad}}:=-M_g^2/(2\lambda_{\min})$.
\end{lemma}
\begin{proof}
	With $d=x^{(j)}-x^{(j-1)}$ and $h=g(x^{(j-1)})$, the spectral safeguard gives
	\[
	d^Th+\tfrac12d^TH^{(j-1)}d
	\ge d^Th+\frac{\lambda_{\min}}2\|d\|^2
	=\frac{\lambda_{\min}}2
	\left\|d+{h}/{\lambda_{\min}}\right\|^2
	-{\|h\|^2}/{(2\lambda_{\min})}.
	\]
	Since $\|h\|\le M_g$, it follows that
	$d^Th+\tfrac12d^TH^{(j-1)}d
	\ge-{M_g^2}/{(2\lambda_{\min})}
	=\widehat c_{\mathrm{quad}}$, 
	which proves the uniform lower bound in~\eqref{e.seqconvex-general}.
\end{proof}

\begin{remark}\label{rem:zeta-bounded}\label{rem:zeta-bounded-revision}
	For unbounded $C$, the fixed ball bounds projections but need not be inactive at a limit. The original-set stationarity result below assumes $\|x^*\|<\zeta$ when $C$ is unbounded; for bounded $C$ the ball is redundant. Bounded penalties ensure strong convexity and bound the forcing $\rho y$, but consensus, rather than penalty bounds alone, gives agreement. The optional balance is sufficient, not necessary.
\end{remark}

\subsection{Frozen inner and accepted outer limits}\label{sec:outerconvergence}
The single-call results concern Steps~\PDstep{2}{a.EPD-inner}--\PDstep{5}{a.EPD-inner}.
Their frozen-call arguments also apply to a trial in
Step~\PDstep{4}{a.CPDQN}, with $R_k$ fixed during that call.
Step~\PDstep{6}{a.EPD-inner} still enforces its finite budget;
infinite-inner limits do not assert finite joint acceptance.
The outer limits below concern only the accepted core trajectory.

\begin{corollary}[Bounded inner iterates]\label{cor:BCD_boundedness}
	Suppose either Lemma~\ref{lem:uniform-bound} or
	Lemma~\ref{lem:agreement-uniform-bound} applies to the unrecovered
	outer trajectory, and use the corresponding bounds
	$M_x$, $M_g$ and $M_\mu$ from
	Proposition~\ref{prop:uniform-penalty-bound}. Then there is
	$R<\infty$ with
	$\|x^{(j-1)}_\ell\|\le R$ for all accepted outer indices $j$
	and all inner indices $\ell$ computed at that outer iteration.
\end{corollary}
\begin{proof}
	For $\ell\ge1$, the exact smooth-block update~\eqref{pr:px-subproblem}
	gives
	\[
	\|x_\ell^{(j-1)}\|
	\le
	\frac{
		\lambda_{\max}M_x+\rho_{\max}\zeta+M_g+M_\mu
	}{
		\lambda_{\min}+\rho_{\min}
	}
	=:R_1,
	\]
	since the frozen center satisfies $\|x^{(j-1)}\|\le M_x$,
	the preceding projected block satisfies $\|y_{\ell-1}^{(j-1)}\|\le\zeta$,
	and
	$\|(H^{(j-1)}+\rho^{(j-1)}I)^{-1}\|
	\le(\lambda_{\min}+\rho_{\min})^{-1}$.
	Since $x_0^{(j-1)}=x^{(j-1)}$, we may take
	$R=\max\{M_x,R_1\}$.
\end{proof}

Inexact alternating minimization for penalty
decomposition was analyzed in
\cite[Section~4]{KanzowLapucci}, using
gradient-related directions, Armijo line searches,
and exact projections onto the geometric constraint.
Here, the smooth block admits an exact quadratic
update, whereas the sparse block uses a selected
convex slice and a model-descent test. Consequently,
the analysis establishes supportwise block
minimality and a vanishing smooth-block residual,
without asserting global minimization over the
union of supports or finite acceptance of all
inner tests.

\begin{theorem}[Inner-loop model convergence]\label{thm:BCD_convergence}
	Fix one call of Algorithm~\ref{a.EPD-inner} with the exact block updates on $C_\zeta$, and abbreviate its frozen model by $\Phi$.
	Suppose each support replacement satisfies
	\eqref{eq:inner-support-descent}, with the previous-support projection
	used on rejection.  If the inner loop is continued indefinitely, then:
	\begin{enumerate}
		\item[(i)] $\Phi(x_\ell,y_\ell)$ is nonincreasing and converges to a finite value.
		\item[(ii)] $\|x_\ell-x_{\ell-1}\|\to0$ and
		$\|y_\ell-y_{\ell-1}\|\to0$. Every cluster point $(x^*,y^*)$
		admits a constant-support subsequence, with support $\mathcal L$, such that
		\begin{equation}\label{eq:inner-cluster-block-minimality}
	x^*\in\Argmin_x\Phi(x,y^*),\qquad
	y^*\in\Argmin_{y\in C_\zeta\cap S_{\mathcal L}}\Phi(x^*,y).
\end{equation}
		\item[(iii)] Every finitely accepted inner pair satisfies
		$\|\nabla_x\Phi(x_\ell,y_\ell)\|\le\varepsilon_{j-1}$
		and the prescribed agreement test.
	\end{enumerate}
\end{theorem}
\begin{proof}
	For the fixed model, the exact $x$-solve and the support acceptance
	rule give
	\begin{align}
		\Phi(x_\ell,y_{\ell-1})
		&\le\Phi(x,y_{\ell-1})
		\quad(x\in\mathbb R^n),\label{eq:inner-x-correct}\\
		\Phi(x_\ell,y_\ell)
		&\le\Phi(x_\ell,v)
		\quad(v\in C_\zeta\cap S_{\mathcal L_\ell}),\quad
		\Phi(x_\ell,y_\ell)
		\le\Phi(x_\ell,y_{\ell-1}).
		\label{eq:inner-y-correct}
	\end{align}
	On rejection, the preceding point $y_{\ell-1}$ is feasible on the
	previous selected slice, so the fallback projection also gives the
	second inequality in~\eqref{eq:inner-y-correct}. Hence
	\[
	\Phi(x_\ell,y_\ell)
	\le\Phi(x_\ell,y_{\ell-1})
	\le\Phi(x_{\ell-1},y_{\ell-1}),
	\]
	and the model values are nonincreasing. For $y\in C_\zeta$, let $m(y):=\min_{x\in\mathbb R^n}\Phi(x,y)$. Since $H+\rho I\succ0$, this minimum is finite, and the explicit
	$x$-minimizer shows that $m$ is continuous in $y$. Compactness of
	$C_\zeta$ therefore gives $\min_{y\in C_\zeta}m(y)>-\infty$. Thus, $\Phi(x_\ell,y_\ell)$ is bounded below and consequently
	converges to a finite value.

	Since $x_\ell$ minimizes the
	$(\lambda_{\min}+\rho)$-strongly convex function
	$x\mapsto\Phi(x,y_{\ell-1})$,
	\[
	\frac{\lambda_{\min}+\rho}{2}
	\|x_\ell-x_{\ell-1}\|^2
	\le
	\Phi(x_{\ell-1},y_{\ell-1})
	-\Phi(x_\ell,y_{\ell-1}).
	\]
	Using~\eqref{eq:inner-y-correct} therefore gives
	\[
	\frac{\lambda_{\min}+\rho}{2}
	\|x_\ell-x_{\ell-1}\|^2
	\le
	\Phi(x_{\ell-1},y_{\ell-1})
	-\Phi(x_\ell,y_\ell).
	\]
	Summing and using the convergence of the model values shows that
	$\sum_{\ell\ge1}\|x_\ell-x_{\ell-1}\|^2<\infty$, and hence $\|x_\ell-x_{\ell-1}\|\to0$. The exact solve~\eqref{pr:px-subproblem} satisfies
	\[
	(H+\rho I)x_\ell
	=Hz+\rho y_{\ell-1}-g(z)+\mu.
	\]
	Subtracting the equations at indices $\ell+1$ and $\ell$ gives
	\[
	(H+\rho I)(x_{\ell+1}-x_\ell)
	=\rho(y_\ell-y_{\ell-1}).
	\]
	Since $H$ and $\rho>0$ are fixed during the call and
	$\|x_{\ell+1}-x_\ell\|\to0$, it follows that $\|y_\ell-y_{\ell-1}\|\to0$. The projected points $y_\ell$ lie in compact $C_\zeta$, and the
	exact formula~\eqref{pr:px-subproblem} then also bounds $x_\ell$.
	Let
	\[
	(x_{\ell_m},y_{\ell_m})\to(x^*,y^*)
	\]
	be a cluster subsequence. Since
	$\|y_{\ell_m}-y_{\ell_m-1}\|\to0$, we also have
	$y_{\ell_m-1}\to y^*$. Because there are only finitely many
	size-$s$ selected supports, after further extraction we may assume
	$\mathcal L_{\ell_m}=\mathcal L$ for all $m$. Passing to the limit
	in~\eqref{eq:inner-x-correct} gives, for every fixed $x\in\mathbb R^n$, $\Phi(x^*,y^*)\le\Phi(x,y^*)$, and passing to the limit in~\eqref{eq:inner-y-correct} gives, for
	every fixed $v\in C_\zeta\cap S_{\mathcal L}$, $\Phi(x^*,y^*)\le\Phi(x^*,v)$. 
	Hence, $x^*$ and $y^*$ satisfy \eqref{eq:inner-cluster-block-minimality}. Finally, every finitely accepted pair satisfies its residual and
	agreement tests by the definition of Step~\PDstep{5}{a.EPD-inner};
	the infinite-loop argument itself does not imply finite joint
	acceptance.
\end{proof}

\begin{corollary}[A positive residual target is attainable]\label{cor:inner-positive-residual}
	For a frozen model satisfying Theorem~\ref{thm:BCD_convergence}, an indefinitely continued loop eventually satisfies the residual test of Step~\PDstep{5}{a.EPD-inner} for any fixed target $\|\nabla_x\Phi(x_\ell,y_\ell)\|\le\varepsilon$ with $\varepsilon>0$. This does not ensure its consensus or objective-filter tests.
\end{corollary}
\begin{proof}
	The exact first-block equation gives
	$\nabla_x\Phi(x_\ell,y_{\ell-1})=0$, and hence, $\nabla_x\Phi(x_\ell,y_\ell)
	=\rho(y_{\ell-1}-y_\ell)$. By Theorem~\ref{thm:BCD_convergence}(ii),
	$\|y_\ell-y_{\ell-1}\|\to0$, so
	$\|\nabla_x\Phi(x_\ell,y_\ell)\|\to0$.
	Thus every fixed positive residual target is satisfied for all
	sufficiently large $\ell$. Consensus and filtering test different
	quantities; Examples~\ref{app:nonstationary-consensus} and
	\ref{app:incompatible-filter} show their possible failure.
\end{proof}

\begin{remark}\label{rem:BCD-interpretation}
	Inner minimality is on one selected slice, not the union. Neither global support stabilization nor finite acceptance is inferred. Finite failed calls return their explicit status.
\end{remark}

\begin{theorem}[Subsequential convergence under bounded penalties]
	\label{thm:outer-model-CCAM}
	Under the hypotheses of Lemma~\ref{lem:agreement-uniform-bound}, let $\{(x^{(j)},y^{(j)})\}_{j\ge0}$ be an infinite
	sequence of accepted iterates generated by the unrecovered \texttt{QNPD} core. Then:
	\begin{enumerate}
		\item[(i)] The outer sequence has an accumulation point $(x^*,y^*)$;
		$y^*\in C_\zeta\subseteq C\cap C_s$ and
		$I_1(y^*)\subseteq\mathcal L$ for some $|\mathcal L|=s$.
		\item[(ii)] Along a subsequence, the selected size-$s$ sets equal $\mathcal L$ and
		\[
		(x^{(j)},y^{(j)},x^{(j-1)},\mu^{(j-1)},\rho^{(j-1)},H^{(j-1)})
		\to(x^*,y^*,\bar x,\mu^*,\rho^*,H^*),
		\]
		where $\rho^*\in[\rho_{\min},\rho_{\max}]$, and $y^*\in\Argmin_{v\in C_\zeta\cap S_{\mathcal L}}
		\Phi_{(\rho^*,\bar x,\mu^*)}(x^*,v)$.
		\item[(iii)] $\nabla_x
		\Phi_{(\rho^{(j-1)},x^{(j-1)},\mu^{(j-1)})}
		(x^{(j)},y^{(j)})\to0$, and on the subsequence in (ii),
		\[
		g(\bar x)+H^*(x^*-\bar x)-\mu^*
		+\rho^*(x^*-y^*)=0.
		\]
	\end{enumerate}
\end{theorem}
\begin{proof}
	Lemma~\ref{lem:agreement-uniform-bound} gives bounded primal pairs and multipliers, while the penalty and spectral safeguards place $\rho^{(j)}$ and $H^{(j)}$ in compact sets. Hence the tuple in (ii) has a convergent subsequence. Since there are only finitely many size-$s$ selected supports, after further extraction they equal a fixed set $\mathcal L$. Closedness of $C_\zeta\cap S_{\mathcal L}$ gives $y^*\in C_\zeta\cap S_{\mathcal L}$, and hence $I_1(y^*)\subseteq\mathcal L$, proving (i). For every fixed $v\in C_\zeta\cap S_{\mathcal L}$, the exact final projection gives
	$\Phi_{(\rho^{(j-1)},x^{(j-1)},\mu^{(j-1)})}(x^{(j)},y^{(j)})
	\le\Phi_{(\rho^{(j-1)},x^{(j-1)},\mu^{(j-1)})}(x^{(j)},v)$
	along this subsequence. Continuity in all model arguments therefore gives the $\Argmin$ assertion in (ii). Finally, the measured residual is
	$g(x^{(j-1)})+H^{(j-1)}(x^{(j)}-x^{(j-1)})-\mu^{(j-1)}
	+\rho^{(j-1)}(x^{(j)}-y^{(j)})$, whose norm is at most
	$\varepsilon_{j-1}\to0$ by~\eqref{e.gnorm}. Passing to the same subsequence proves (iii). The spectral and penalty bounds persist at the limits.
\end{proof}

\begin{lemma}[Asymptotic primal agreement]\label{lem:primal-agreement}
	If every accepted core pair passes Step~\PDstep{5}{a.EPD-inner} through \eqref{eq:primalDual}, with
	$0<\tau<1$ and $\eta_j\to0$, then $\|x^{(j)}-y^{(j)}\|\longrightarrow0$. Thus any cluster pair has $x^*=y^*$.
\end{lemma}
\begin{proof}
	Recall $d_j=\|x^{(j)}-y^{(j)}\|$, the agreement test is exactly the recursion~\eqref{eq:scalar-consensus}, so $d_j\to0$. Hence, along any subsequence
	$(x^{(j_k)},y^{(j_k)})\to(x^*,y^*)$, we have
	$\|x^*-y^*\|=\lim_k d_{j_k}=0$, and therefore $x^*=y^*$.
\end{proof}

\begin{proposition}[Outer regularity from the true-objective filter]\label{prop:objective-filter-regularity}
	If an infinite accepted core sequence enables the filter in Step~\PDstep{5}{a.EPD-inner} ($\chi=1$) and satisfies~\eqref{eq:outer-objective-filter} with fixed $\sigma>0$, then
	\[
	\sum_{j\ge1}\|y^{(j)}-y^{(j-1)}\|^2<\infty,
	\qquad\|x^{(j)}-x^{(j-1)}\|\to0.
	\]
\end{proposition}
\begin{proof}
	Compactness of $C_\zeta$ and continuity of $f$ give
	$m_f:=\min_{C_\zeta}f>-\infty$. Summing the true-objective filter gives
	$\sigma\sum_{j=1}^k\|y^{(j)}-y^{(j-1)}\|^2
	\le f(y^{(0)})-f(y^{(k)})\le f(y^{(0)})-m_f$.
	Hence the squared $y$-increments are summable and
	$\|y^{(j)}-y^{(j-1)}\|\to0$. Lemma~\ref{lem:primal-agreement} gives
	$d_j\to0$, and
	$\|x^{(j)}-x^{(j-1)}\|
	\le d_j+\|y^{(j)}-y^{(j-1)}\|+d_{j-1}\to0$.
	A lower bound on a changing model would not justify this telescoping.
\end{proof}

Successful inner calls can still cycle, as the following example shows.
\begin{example}[Accepted unfiltered steps can cycle]\label{app:unfiltered-cycle}
	We take $\chi=0$, $C=[-1,1]^2$, $s=1$, $f(x)=\|x\|^2$, $H=I$, $\mu^0=0$, $x^0=y^0=e_1/4$, $\rho^0=100$, $\rho_{\min}=99$, $\rho_{\max}=101$ and $r=1.15$. At a signed center $\omega e_1$ with $y_0=\omega e_1$, exact updates have $x_\ell=y_\ell=t_\ell e_1$, $t_\ell=-\omega+2\omega q^\ell$, $q=\rho/(1+\rho)$ and $r_\ell=2\omega q^\ell e_1$. Ranking retains support $1$; model descent and zero consensus hold, so $\mu=0$ persists. For $\varepsilon_{j-1}=2^{-j-4}$, let $M_j$ be the first index with $2|\omega|q^{M_j}\le\varepsilon_{j-1}$ and choose $K_j\ge M_j$. Each accepted scalar changes sign, with $A_j=A_{j-1}-\delta_j$, where $0<\delta_j=2A_{j-1}q^{M_j}\le\varepsilon_{j-1}$. Since $\sum_{j\ge1}\varepsilon_{j-1}=1/16$, we have $A_j\downarrow A_\infty\ge3/16$. The even and odd subsequences tend to $\pm A_\infty e_1$, and consecutive step lengths tend to $2A_\infty>0$; neither limit is M-stationary, and outer regularity fails.
	The optional bounds in Section~\ref{sec:legacy-core-scope} hold with $\gamma=6$, $c=2$, $\zeta=\sqrt2$, $D=100$, $\lambda_{\min}=\lambda_{\max}=1$ and $\varepsilon_*=1/32$: substitution in~\eqref{eq:bar-theta-kappa-general} gives $a=0.08$, $b=0.07$, $\bar\theta\approx0.307582$ and $\bar\kappa\approx1.095074<2(1-\bar\theta)\approx1.384836$. The model minimum is at most its zero center value, so $\Upsilon>0$ prevents restarts. A positive filter excludes this cycle by Proposition~\ref{prop:objective-filter-regularity}; the reference excludes it in the full method.
\end{example}

The AM- and M-stationarity conclusions below are
related to those of
\cite[Theorems~3.2 and~3.3]{KanzowLapucci}
for inexact penalty decomposition. The present
result concerns a different outer mechanism:
the penalty parameters remain bounded, primal
agreement is enforced explicitly, and outer
asymptotic regularity follows from the
true-objective filter or is assumed separately.
The conclusion also requires the stated
qualification on the artificial projection radius.

\begin{theorem}[AM-stationarity and M-stationarity for the original problem]
	\label{thm:outer-original-CCM}
	Suppose the hypotheses of Theorem~\ref{thm:outer-model-CCAM} hold,
	with full-size final super-supports stored through Steps~\PDstep{1}{a.EPD-inner}--\PDstep{4}{a.EPD-inner}. In mode $\chi=1$, require the filter~\eqref{eq:outer-objective-filter}. In mode $\chi=0$, assume outer asymptotic regularity, $\|x^{(j)}-x^{(j-1)}\|\to0$.  Then every accumulation point
	$(x^*,y^*,\mu^*)$ of accepted primal-multiplier triples has
	$x^*=y^*\in\Omega$ and $\mu^*=g(x^*)$.
	Suppose additionally that either $C$ is bounded or, when $C$ is
	unbounded, $\|x^*\|<\zeta$.  Then $x^*$ is geometrically
	\emph{AM-stationary}: there are
	$a^k\in C$, $b^k\in C_s$, $a^k,b^k\to x^*$, and
	$u^k\in N_C(a^k)$, $v^k\in N_{C_s}(b^k)$ such that
	\begin{equation}\label{eq:CCAM-stationarity}
		g(a^k)+u^k+v^k\longrightarrow0.
	\end{equation}
	If, in addition, joint AM-regularity
\eqref{eq:geometric-AM-regularity} holds at $x^*$, then $x^*$ is
	\emph{M-stationary}:
	\[
	0\in g(x^*)+N_C(x^*)+N_{C_s}(x^*).
	\]
	In particular, subject to the same radius qualification, M-stationarity follows for polyhedral $C$ and for closed convex $C$ with $0\in\operatorname{int}C$. No global gradient-growth or recurrence-balance assumption is needed.
\end{theorem} 
\begin{proof}
	In mode $\chi=1$, Proposition~\ref{prop:objective-filter-regularity} gives outer asymptotic regularity; in mode $\chi=0$, it is assumed. On any convergent triple subsequence $j_k$,
	Lemma~\ref{lem:primal-agreement} gives $x^*=y^*\in\Omega$, while
	outer asymptotic regularity gives $x^{(j_k-1)}\to x^*$.
	Equation~\eqref{eq:residual-multiplier-identity}, the spectral
	safeguard and $\varepsilon_{j_k-1}\to0$ then give
	$\mu^*=g(x^*)$. Moreover,
	$\|\mu^{(j)}-\mu^{(j-1)}\|\le\rho_{\max}d_j\to0$, so the preceding
	multipliers have the same limit.

	After further extraction, the stored size-$s$ supports equal a fixed
	set $\mathcal L$. The final projection minimizes
	\[
	q_j(y)=\frac{\rho^{(j-1)}}{2}
	\left\|y-\left(x^{(j)}
	-{\mu^{(j-1)}}/{\rho^{(j-1)}}\right)\right\|^2
	\]
	on $C_\zeta\cap S_{\mathcal L}$. The artificial ball is redundant
	when $C$ is bounded. If $C$ is unbounded and $\|x^*\|<\zeta$, set
	\[
	\delta:=\frac{\zeta-\|x^*\|}{3}>0.
	\]
	Since $y^{(j_k)}\to x^*$, for all sufficiently large $k$,
	$\|y^{(j_k)}-x^*\|<\delta$. If
	$z\in B_\delta(y^{(j_k)})$, then
	$\|z\|<\|x^*\|+2\delta<\zeta$. Hence
	$B_\delta(y^{(j_k)})\subset B_\zeta(0)$, so the artificial ball is
	locally inactive and $y^{(j_k)}$ locally minimizes $q_{j_k}$ over
	$C\cap S_{\mathcal L}$.

	By the definition of $q_j$ and the multiplier update,
	\begin{equation}\label{eq:projection-multiplier-gradient}
		\nabla q_j(y^{(j)})=\mu^{(j-1)}
		+\rho^{(j-1)}(y^{(j)}-x^{(j)})=\mu^{(j)}.
	\end{equation}
	For each sufficiently large $k$, apply
	Lemma~\ref{lem:two-point-normal-support} to $q_{j_k}$ at
	$y^{(j_k)}$. By a diagonal choice, select
	$a^k,b^k,u^k,v^k$ such that both point errors relative to
	$y^{(j_k)}$ and
	$\|\nabla q_{j_k}(a^k)+u^k+v^k\|$ are at most $1/k$.
	Then $a^k,b^k\to x^*$,
	$u^k\in N_C(a^k)$ and $b^k\in S_{\mathcal L}$, while
	$v^k\in S_{\mathcal L}^\perp\subseteq N_{C_s}(b^k)$ by
	\eqref{eq:sparse-normal-compact}.

	Since $\nabla q_{j_k}$ has Lipschitz constant
	$\rho^{(j_k-1)}\le\rho_{\max}$,
	$\|\nabla q_{j_k}(a^k)-\mu^{(j_k)}\|\le\rho_{\max}/k$ by
	\eqref{eq:projection-multiplier-gradient}. Hence
	$\nabla q_{j_k}(a^k)\to g(x^*)$, while continuity of $g$ gives
	$g(a^k)\to g(x^*)$. Consequently,
	$\|g(a^k)+u^k+v^k\|
	\le1/k+\|g(a^k)-\nabla q_{j_k}(a^k)\|\to0$,
	proving~\eqref{eq:CCAM-stationarity}. Theorem~\ref{thm:AM-to-M}
	and the regularity lemmas give the stated M-stationarity conclusions.
\end{proof}

The next example explains the strict radius-inactivity hypothesis
in Theorem~\ref{thm:outer-original-CCM}.
\begin{example}[An active core radius changes stationarity]\label{app:radius-example}
	We take $C=\mathbb R^2$, $s=1$, $\zeta=1$, $f(x)=\|x-2e_1\|^2/2$, $x_0=y_0=e_1$, $H=I$ and $\mu=-e_1=g(e_1)$. For every $\rho>0$, the exact smooth solve gives $e_1$, and the capped-slice projection of $(1+1/\rho)e_1$ returns $e_1$. The model-residual, consensus and objective-filter tests pass forever, but $g(e_1)=-e_1$ cannot be cancelled by the original normal sum $\operatorname{span}\{e_2\}$. Boundedness does not imply radius inactivity. This motivates projecting the reference onto the original $\Omega$.
\end{example}

\begin{remark}\label{rem:normal-decomposition-role}\label{rem:BFS-role}
	A direct M proof is possible when $N_{C\cap S_{\mathcal L}}(x)\subseteq N_C(x)+S_{\mathcal L}^\perp$, as for polyhedral $C$ or under a relative-interior qualification. Symmetry alone does not imply this rule. The two-point proof instead needs only limit regularity. Its support argument uses exact selected-slice projection, descent, and finitely many full-size supports, not the ranking formula; full size is essential to interpret the subspace normal as a cardinality normal. For $n=2$, $s=1$ and $b=0$, using $S_{I_1(b)}^\perp=\mathbb R^2$ would admit $(1,1)$, although $N_{C_1}(0)=\operatorname{span}\{e_1\}\cup\operatorname{span}\{e_2\}$.
\end{remark}

\begin{remark}[Symmetry without regularity]\label{rem:convex-symmetry-insufficient}
	In Example~\ref{ex:AM-not-M}, $\Omega=\{(1,0),(0,1)\}$ and
	$p=(1,0)$ is a global minimizer. For $w=p$, every $\alpha>0$
	gives $P_\Omega(w-\alpha g(w))=\{p\}$ and hence $E=0$,
	whereas $\mathrm{rg}_M(p)=1$. Thus the finite M-residual bound
	in Proposition~\ref{prop:reference-certificate} requires pointwise
	joint AM-regularity, even when the witness is zero.
\end{remark}

\begin{corollary}[Convergence of the unrecovered \texttt{QNPD} core]
	\label{cor:algo-convergence}
	Assume the hypotheses of Theorem~\ref{thm:outer-original-CCM} and
	an infinite accepted sequence. Every accumulation point of the
	primal iterates belongs to $\Omega$ and satisfies the geometric
	AM-stationarity condition~\eqref{eq:geometric-AM}, provided
	$\|x^*\|<\zeta$ when $C$ is unbounded.
	If joint AM-regularity~\eqref{eq:geometric-AM-regularity} holds at
	every such point (in particular, if $C$ is polyhedral or
	$0\in\operatorname{int}C$), every accumulation point is M-stationary:
	$0\in g(x^*)+N_C(x^*)+N_{C_s}(x^*)$.
	These are subsequential stationarity results; convergence of the
	\emph{whole} sequence or finite inner termination requires further
	hypotheses.
\end{corollary}
\begin{proof}
	Let $x^{(j_k)}\to x^*$ be any primal cluster subsequence.
	Lemma~\ref{lem:primal-agreement} gives $y^{(j_k)}\to x^*$, and the
	established multiplier bounds yield a further subsequence on which
	$\mu^{(j_k)}$ converges. Theorem~\ref{thm:outer-original-CCM} applied
	to this triple subsequence gives the AM-stationarity conclusion and,
	under joint AM-regularity, the M-stationarity conclusion.
\end{proof}

The following pointwise result supplies the interpretation used in
Corollary~\ref{cor:certified-all-support-S} for Algorithm~\ref{a.CPDQN};
it imposes no additional test on Algorithm~\ref{a.EPD-inner}. Unlike the lifted CC-S-stationarity condition in
\cite[Definition~4.12]{kanzow2021}, the following
result concerns a universal family of supportwise
split-normal conditions in the original variable
space. Its equivalence with BF-stationarity requires the stated normal-cone intersection rule.

\begin{theorem}[All-super-support S-stationarity and BF]
	\label{thm:all-support-S-BF}
	Let $C$ be closed and convex and let $\bar x\in\Omega$.
	Then the following statements hold.
	\begin{enumerate}
		\item If $\bar x$ is all-super-support S-stationary, then
		$\bar x$ is BF-stationary and M-stationary.
		
		\item If $\|\bar x\|_0=s$, then all-super-support
		S-stationarity is equivalent to M-stationarity; hence, in this
		case, M-stationarity implies BF-stationarity.
		
		\item Suppose that, for every super-support $\mathcal L$ of
		$\bar x$, the normal-cone intersection rule
		\begin{equation}\label{eq:support-normal-sum}
			N_{C\cap S_{\mathcal L}}(\bar x)
			=
			N_C(\bar x)+S_{\mathcal L}^{\perp}
		\end{equation}
		holds. Then BF-stationarity is equivalent to
		all-super-support S-stationarity. Consequently,
		under~\eqref{eq:support-normal-sum},
		\[
		\text{all-super-support S}
		\quad\Longleftrightarrow\quad
		\text{BF}
		\quad\Longleftrightarrow\quad
		-g(\bar x)\in N_\Omega^F(\bar x).
		\]
		A sufficient condition for~\eqref{eq:support-normal-sum} is $\operatorname{ri}(C)\cap S_{\mathcal L}\ne\varnothing$
		for every super-support $\mathcal L$ of $\bar x$.
	\end{enumerate}
\end{theorem}

\begin{proof}
	Let $\bar x$ be all-super-support S-stationary and fix an arbitrary
	super-support $\mathcal L$ of $\bar x$. By
	Definition~\ref{def:all-support-S}, there exist
	$u_{\mathcal L}\in N_C(\bar x)$ and
	$v_{\mathcal L}\in S_{\mathcal L}^{\perp}$ such that
	$g(\bar x)+u_{\mathcal L}+v_{\mathcal L}=0$.
	For arbitrary $z\in C\cap S_{\mathcal L}$, convex normality gives
	$u_{\mathcal L}^{T}(z-\bar x)\le0$. Since $z,\bar x\in S_{\mathcal L}$
	and $v_{\mathcal L}\in S_{\mathcal L}^{\perp}$, we also have
	$v_{\mathcal L}^{T}(z-\bar x)=0$. Taking the inner product of the
	stationarity equation with $z-\bar x$ yields
	\[
	g(\bar x)^T(z-\bar x)
	=
	-u_{\mathcal L}^{T}(z-\bar x)
	-v_{\mathcal L}^{T}(z-\bar x)
	\ge0.
	\]
	Since both $\mathcal L$ and $z\in C\cap S_{\mathcal L}$ were
	arbitrary, $\bar x$ is BF-stationary.

	Since $\mathcal L$ is a super-support of $\bar x$,
	\eqref{eq:sparse-normal-compact} gives
	$v_{\mathcal L}\in S_{\mathcal L}^{\perp}\subseteq
	N_{C_s}(\bar x)$. Hence
	$0\in g(\bar x)+N_C(\bar x)+N_{C_s}(\bar x)$, so $\bar x$ is
	M-stationary. This proves part~(1).

	If $\|\bar x\|_0=s$, there is only one super-support,
	namely $\mathcal L=I_1(\bar x)$, and
	\eqref{eq:sparse-normal-compact} gives
	$N_{C_s}(\bar x)=S_{\mathcal L}^{\perp}$.
	Thus~\eqref{eq:geometric-M} and~\eqref{eq:all-support-S} are
	identical, proving part~(2).

	For part~(3), suppose first that $\bar x$ is BF-stationary.
	For every super-support $\mathcal L$ of $\bar x$, BF gives
	\[
	g(\bar x)^T(z-\bar x)\ge0
	\qquad
	(z\in C\cap S_{\mathcal L}).
	\]
	Since $C\cap S_{\mathcal L}$ is closed and convex, this is
	equivalent to
	$-g(\bar x)\in N_{C\cap S_{\mathcal L}}(\bar x)$.
	Using~\eqref{eq:support-normal-sum}, there exist
	$u_{\mathcal L}\in N_C(\bar x)$ and
	$v_{\mathcal L}\in S_{\mathcal L}^{\perp}$ such that
	$-g(\bar x)=u_{\mathcal L}+v_{\mathcal L}$.
	Hence~\eqref{eq:all-support-S} holds for every super-support
	$\mathcal L$, so $\bar x$ is all-super-support S-stationary.

	The converse follows from part~(1). The equivalence with
	$-g(\bar x)\in N_\Omega^F(\bar x)$ follows from
	Proposition~\ref{prop:stationarity-relations}.
	Finally, since $\operatorname{ri}(S_{\mathcal L})=S_{\mathcal L}$,
	the condition
	$\operatorname{ri}(C)\cap S_{\mathcal L}\ne\varnothing$
	is the standard convex qualification giving
	\eqref{eq:support-normal-sum}.
\end{proof}

\begin{corollary}[All-super-support S-stationarity of accumulation points]
	\label{cor:certified-all-support-S}
	Under Assumption~\ref{ass:certified-main}, let $\bar w$ be any
	accumulation point of the infinite sequence generated by
	Algorithm~\ref{a.CPDQN}. If
	\begin{equation}\label{eq:limit-support-normal-sum}
		N_{C\cap S_{\mathcal L}}(\bar w)
		=
		N_C(\bar w)+S_{\mathcal L}^{\perp}
		\qquad
		\text{for every }
		\mathcal L\in\mathfrak L_s(\bar w),
	\end{equation}
	then $\bar w$ is all-super-support S-stationary. In particular, this conclusion holds when $C$ is polyhedral
	and when $C$ is closed and convex with
	$0\in\operatorname{int}C$.
\end{corollary}

\begin{proof}
	By Theorem~\ref{thm:certified-stationarity}, every accumulation
	point $\bar w$ is BF-stationary on every super-support.
	Theorem~\ref{thm:all-support-S-BF}(3), together with
	\eqref{eq:limit-support-normal-sum}, therefore gives
	all-super-support S-stationarity. If $C$ is polyhedral, then $C\cap S_{\mathcal L}$ is polyhedral,
	and the polyhedral normal-cone intersection rule gives
	\[
	N_{C\cap S_{\mathcal L}}(\bar w)
	=
	N_C(\bar w)+S_{\mathcal L}^{\perp}.
	\]
	If $0\in\operatorname{int}C$, then, for every
	$\mathcal L\in\mathfrak L_s(\bar w)$,
	$0\in\operatorname{ri}(C)\cap S_{\mathcal L}$.
	Hence the standard convex normal-cone intersection rule gives
	\eqref{eq:limit-support-normal-sum}.
\end{proof}

\section{Numerical experiments}\label{numerica}
We compare the proposed quasi-Newton variants with standalone sparse optimizers.
Restricted \texttt{FISTA} is used only inside supportwise routines and as a
warm-start component; it is not listed as a standalone comparison solver.
The 192-instance benchmark includes nine quasi-Newton configurations and
three standalone sparse optimizers: \texttt{BFS}, \texttt{PSS}, and
\texttt{ZCWS}. All twelve methods are evaluated under the same objective,
stationarity, and weighted oracle-work reporting protocol.

\subsection{Algorithms compared}\label{subsec:num-compared}
We consider three standalone optimizers, \texttt{BFS}, \texttt{PSS}, and
\texttt{ZCWS}, and nine implementations of the proposed safeguarded
quasi-Newton penalty-decomposition method. These have different roles:
restricted \texttt{FISTA} may be used for supportwise minimization
inside \texttt{BFS} and as a once-only warm start for \texttt{QNPD};
it is \emph{not} an additional standalone competitor. The \texttt{PSS} and \texttt{ZCWS} implementations follow their respective
support-growth, coordinate-search and support-exchange procedures.
For the comparison, they use feasible support-restricted optimization,
common oracle accounting, and the same terminal reference audit as the
other standalone comparator. They are not quasi-Newton variants.

For every \texttt{QNPD} variant, Step~\PDstep{1}{a.CPDQN}
computes the exact feasible reference through Algorithm~\ref{a.reference},
and Step~\PDstep{2}{a.CPDQN} checks its witness. On an attempted trial,
Steps~\PDstep{3}{a.CPDQN}--\PDstep{4}{a.CPDQN} prepare the clipped
diagonal model and invoke Algorithm~\ref{a.EPD-inner}, whose exact smooth
block, support proposal, model-descent gate and final residual are
Steps~\PDstep{2}{a.EPD-inner}--\PDstep{5}{a.EPD-inner}.
Step~\PDstep{5}{a.CPDQN} retains both true-objective acceptance
inequalities; it otherwise recovers to the feasible reference.
Step~\PDstep{6}{a.CPDQN} updates the penalty and diagonal.

\begin{table}[htbp]
	\centering\caption{Roles and modifications of the compared methods. The
		warm-start column indicates once-only support-restricted \texttt{FISTA} initialization for the \texttt{QNPD} variants;
		\texttt{FISTA} itself is not a standalone row. The three standalone
		methods have no penalty-decomposition trial or diagonal update.}
	\label{tab:num-variants-steps}
	\scriptsize
	\begin{tabular}{llll}\toprule
		Method & Warm start & PD trial schedule & Penalty rule\\\midrule
		\texttt{BFS} & internal support solver & --- & ---\\
		\texttt{PSS} & --- & --- & ---\\
		\texttt{ZCWS} & internal solver & --- & ---\\
		\texttt{QNPD} & no & always & default\\
		\texttt{QNPD+FISTA} & yes & always & default\\
		\texttt{QNPD+FISTA+OS} & yes & OS & default\\
		\texttt{QNPD+FISTA+periodic5} & yes & every fifth & default\\
		\texttt{QNPD+FISTA+slow-rho} & yes & always & slow\\
		\texttt{QNPD+FISTA+fixed-rho} & yes & always & fixed\\
		\texttt{QNPD+FISTA+OS+slow-rho} & yes & OS & slow\\
		\texttt{QNPD+FISTA+OS+fixed-rho} & yes & OS & fixed\\
		\texttt{QNPD+FISTA+OS+adaptive-rho} & yes & OS & adaptive\\
		\bottomrule\end{tabular}
\end{table}

The once-only restricted \texttt{FISTA} warm start is applied before
Step~\PDstep{0}{a.CPDQN} and is charged to the common evaluation budget.
The suffix \texttt{OS} skips the next four optional trial opportunities after
three consecutive rejected trials; \texttt{periodic5} attempts trials only
at $k=0,5,10,\ldots$. Skipped trials receive a zero trial budget, but
neither the reference step nor either outer acceptance condition is removed. The standard (\texttt{default}), conservative (\texttt{slow-rho}), and constant (\texttt{fixed-rho}) bounded penalty
rules use multiplicative factors $1.08$, $1.02$, and $1$,
respectively. The adaptive rule uses $1.02$ after acceptance
and $0.95$ otherwise. In all four cases, the penalty
remains fixed during each exact inner call. All nine quasi-Newton configurations use the diagonal rule selected in
Section~\ref{subsec:num-hessian-ablation}.

\subsubsection{Scope of the convergence guarantees}\label{subsubsec:num-theory-status}
The nine safeguarded quasi-Newton configurations are covered conditionally
by Theorems~\ref{thm:certified-descent} and~\ref{thm:certified-stationarity}
and Corollary~\ref{cor:certified-complexity} when
Assumption~\ref{ass:certified-main} and the exact reference and outer
acceptance tests hold. Finite optional warm starts, zero budgets at skipped
trials, and bounded penalty updates do not change those arguments.
The standalone \texttt{BFS}, \texttt{PSS}, and \texttt{ZCWS} methods do not
inherit these theorems. In particular, computing a reference witness at an
external solver's terminal point does not prove convergence of its iterates.
The fixed global $\mathrm{nf2g}:=n_f+2n_g$ budget may interrupt any run before a target is
reached, and numerical floating-point checks are not exact-oracle proofs.

\subsection{Scope, problem collection, and reproducibility}\label{subsec:num-protocol}\label{tab:cc_summary}\label{tab:dim_summary}
The reported experiments were performed in Python. The manuscript source supplied here does not include the implementations, instance generators, raw run logs, or analysis scripts; consequently the reported numerical results cannot be independently reproduced or audited from this source package. The standalone \texttt{BFS}, \texttt{PSS} and \texttt{ZCWS} routines use support-restricted projected first-order minimization; the nine \texttt{QNPD} variants share the safeguarded penalty-decomposition implementation. The support-growth and swap routines have been adapted to the same convex slices and weighted evaluation ledger to provide matched experimental conditions. The mathematical analysis in Sections~\ref{main} and \ref{convergence} is not modified by the experimental choices.

We use eight families suggested by the sparse symmetric-set setting~\cite{Beck2013,Beck2016,Tillmann2024}: convex positive-semidefinite quadratics; mean-variance models over the simplex; sparse linear least-squares regression; logistic losses; sparse PCA; indefinite quadratics; smoothed phase retrieval; and regularized control regression. The corresponding objective definitions and admissible sets are
\begin{align*}
	f_{\rm PSD}(x)&=\tfrac12x^TQx+c^Tx, & Q&\succeq0,\\
	f_{\rm MV}(x)&=\tfrac12x^TQx-c^Tx, & x&\in\Delta_n,\\
	f_{\rm LS}(x)&=\|Ax-b\|^2/(2m),\\
	f_{\rm log}(x)&=m^{-1}\sum_{i=1}^m\log(1+e^{-b_iA_ix}), & b_i&\in\{-1,1\},\\
	f_{\rm PCA}(x)&=-x^T\Sigma x, & \Sigma&\succeq0,\\
	f_{\rm ind}(x)&=\tfrac12x^TQx+c^Tx, & Q&=Q^T,\\
	f_{\rm phase}(x)&=\frac1{2m}\sum_{i=1}^m\big(\sqrt{(A_ix)^2+\delta^2}-b_i\big)^2,&\delta&>0,\\
	f_{\rm ctrl}(x)&=\|Ax-b\|^2/(2m)+\tfrac\lambda2\|x\|^2,&\lambda&\ge0.
\end{align*}
Every family additionally imposes $\|x\|_0\le s$. According to the instance, the convex component is $\mathbb R^n$, a box, a nonnegative box, the nonnegative orthant, $\Delta_n$, or an $\ell_1$ or $\ell_2$ ball. The eight families contribute $24$ instances each, for $192$ total; dimensions are $n\in\{10,20,40,80,120,200,500\}$. The nominal sparsity fractions are $0.15$, $0.25$, $0.50$ and $0.75$, rounded and clipped by $2\le s\le n-1$. We use two independently generated collections of $96$ instances, with integer random seeds \texttt{20271203} and \texttt{20271204}; the shared instance-generation templates are the same as in the preceding development study. The reported seeds alone do not specify the instance generators, exact input data, software versions, or method implementations. These materials and machine-readable per-run results are needed to reproduce the figures and tables. For noncompact sets, compactness of the initial feasible level set from the convergence assumptions is not separately certified by these numerical records.

\subsection{Solvers, budgets, and safeguards}\label{subsec:num-methods}
Each instance uses the same feasible initial point across methods. The diagonal curvature rule is \texttt{LM3}, selected from the objective, geometric M-stationarity, and reference-witness performance profiles in Section~\ref{subsec:num-hessian-ablation}. The same rule is used for all nine safeguarded configurations in the subsequent comparison. All \texttt{QNPD} variants retain the frozen-model smooth-block update and the prescribed true-objective acceptance tests. The initial \texttt{FISTA} warm stage chooses a single gradient-ranked support and performs restricted minimization there, with at most $120$ weighted oracle units.

Three optional mechanisms are ablated: (A) \texttt{FISTA} initialization; (B) a trial budget that can be set to zero, including \texttt{OS}, which skips the next four trial attempts after three consecutive trial rejections, and \texttt{periodic5}, which attempts a \texttt{QNPD} trial at outer indices divisible by five; and (C) bounded penalty schedules. The default multiplicative factor is $1.08$, versus $1.02$ for \texttt{slow-rho}; \texttt{fixed-rho} retains the current penalty, and \texttt{adaptive-rho} multiplies it by $1.02$ after an accepted trial and $0.95$ otherwise. In the implemented schedules $\rho\in[0.05,12]$; the initial penalty is cycled through $0.35,1,4,10$ across instances. The model spectral bounds are $\lambda_{\min}=0.15$ and $\lambda_{\max}=4$. The external reference parameter is $\sigma=0.002$. Frozen inner calls use up to seven updates and an inner weighted-oracle cap of $36$. \emph{No finite outer-iteration cap is imposed}: the outer counter ranges over $k=0,1,\ldots$ without an upper index, and the only global work limit is $\mathrm{nf2g}_{\max}=20000$. A successful reference-stationarity check may still terminate an iteration sequence before the work budget is exhausted. The standalone routines also retain their intrinsic numerical stagnation tests, not external iteration ceilings. Since the proposed variants explicitly stop on the reference-witness target whereas the standalone methods retain their own stopping rules, differences in target attainment are confounded with stopping policy; the common terminal audit does not eliminate this distinction. These are \emph{experimental} parameters, not the illustrative prospective parameters of the earlier numerical outline. The penalty and diagonal matrix are fixed during each inner call. A skipped trial is implemented by a zero trial budget as permitted in Algorithm~\ref{a.CPDQN}.

Every complete run has a common \emph{global evaluation-work} cap
\begin{equation}\label{eq:num-nf2g}
	\mathrm{nf2g}\le20000,
\end{equation}
where both objective and gradient outputs are counted, including warm starts, rejected model trials, bookkeeping and one terminal reference audit. Joint evaluations increment both counters. The source manuscript does not include the recorded projection counts, elapsed times, or per-run ledger, and this comparison does not provide a separately verified count of every internal convex-slice solve; therefore \eqref{eq:num-nf2g} is an oracle-work proxy rather than total computational complexity.  Many runs finish within a few milliseconds, making
wall-clock time comparisons sensitive to the implementation and computing environment. We therefore use the weighted count as our primary measure of computational cost.

\subsection{Terminal stationarity and performance measures}\label{subsec:num-targets}\label{subsec:stopping-and-stationarity}
To compare all implementations at a common point, we apply the same terminal feasible-reference audit to their final incumbents $w$. The audit computes a numerical Armijo reference
\[
p\in P_\Omega(w-\alpha g(w)),\qquad
E=(w-p)/\alpha+g(p)-g(w),
\]
and \emph{reports $p$ as the numerical output for every method}, including for the measured \texttt{BFS} baseline. Thus, the data rows named \texttt{BFS} represent that method followed by this common audit: their recorded objective values need not equal their pre-audit values. Because this operation may change both the point and its objective value, the results compare each method together with the same postprocessing step; they do not isolate the performance of the unaudited outputs. The reference target is $\|E\|\le2\times10^{-3}$; we also report whether the numerically computed geometric Mordukhovich residual $r_M:=\mathrm{rg}_M$ from Section~\ref{sec:geometric-residual} satisfies $r_M(p)\le10^{-3}$. These are distinct diagnostics: an $E$-target test is not interchangeable with an $r_M$ test, and the exact mathematical implications of Proposition~\ref{prop:reference-certificate} require its pointwise hypotheses and exact projection/normal operations. One \texttt{BFS} audit on a sparse-PCA problem encountered numerical Armijo stagnation; it is marked \emph{uncertified}, not silently counted as a success. Every other reported reference audit returned a finite witness. All returned numerical points passed the implementation feasibility check, with maximum reported error below $8\times10^{-13}$. These values are reported by the authors; the missing implementation and run records prevent independent verification of the support thresholds, active-set calculations, and accuracy of the numerical $r_M$ values.

For objective quality, write $f_{i,\rm best}=\min_a f(x_{i,a})$ for the best valid audited endpoint on instance $i$, and define the normalized retrospective gap
\begin{equation}\label{e.qs}
	d_{i,a}=\frac{f(x_{i,a})-f_{i,\rm best}}{\max\{1,|f_{i,\rm best}|,|f_{i,0}-f_{i,\rm best}|\}}.
\end{equation}
The quantity $d_{i,a}$ is used in the figures. At a common budget $B$ we likewise define
\begin{equation}\label{eq:num-budget-gap}
	d_{i,a}(B)=\frac{f_{i,a}(B)-f_{i,\rm best}(B)}{\max\{1,|f_{i,\rm best}(B)|,|f_{i,0}-f_{i,\rm best}(B)|\}},
\end{equation}
where $f_{i,a}(B)$ is the last recorded feasible objective at or before $\mathrm{nf2g}=B$ and $f_{i,\rm best}(B)$ is the best such value across compared methods. This is a \emph{retrospective} measure, not an optimality gap against a known global minimum. Its value and the corresponding objective target can change when methods are added to or removed from the comparison. The objective target in the endpoint work profiles is the best final objective among the twelve compared methods plus $10^{-3}\max\{1,f_{i,0}-f_{i,\rm best}\}$.

The plotted relative-cost curves are Dolan--Mor\'e-\emph{style endpoint profiles}~\cite{DolanMore}: the cost associated with a successful method is its \emph{total completed-run} $\mathrm{nf2g}$, rather than the first cost at which a target was reached; an unsuccessful run has infinite ratio. The denominator is the number of all $192$ instances, including cases on which no method attains the target. Such figures must \emph{not} be interpreted as first-hit Dolan--Mor\'e profiles: in particular, a lower completed-run cost does not establish that a method first attained the target sooner. Equation~\eqref{eq:num-budget-gap} and the logged objective trajectories support the separate common-budget comparison.

\subsection{Comparison of the four diagonal Hessian approximations}\label{subsec:num-hessian-ablation}
We first compare the four safeguarded diagonal rules
\texttt{WCI}, \texttt{LM1}, \texttt{LM2}, and \texttt{LM3} defined in
\eqref{eq:four-diagonal-proposals} to select one rule for the subsequent
comparison of the measured configurations. In the selection experiment,
only the curvature update in Step~\PDstep{6}{a.CPDQN} differs. The
spectral clipping, scaled closed-form trial preparation, exact frozen-model
updates, support proposals, reference step, trial acceptance inequalities,
penalty policy, and initial feasible point are the same. The four rules
are compared on the same $192$ instances (two random-seed collections),
without a \texttt{FISTA} warm start and without a finite outer-iteration
ceiling. Every run includes the terminal reference audit and satisfies
$\mathrm{nf2g}\le20000$.

The selection uses three separately defined endpoint targets. For
\emph{objective quality}, a solver succeeds when its audited objective is
within $10^{-3}\max\{1,f_{i,0}-f_{i,\rm best}\}+10^{-10}$ of the smallest
objective attained by \emph{the four diagonal variants} on that instance.
The additive $10^{-10}$ is the numerical comparison allowance.
For \emph{geometric stationarity}, success requires $r_M\le10^{-3}$,
and for the \emph{reference witness}, it requires
$\|E\|\le2\times10^{-3}$. For each target, Figure~\ref{fig:num-hessian-profile}
shows the exact right-continuous endpoint-work profile: Successful
costs are the full completed-run $\mathrm{nf2g}$, failed costs are
$+\infty$, and the denominator contains all $192$ cases. These are
not first-hit cost profiles. The objective target here is determined
\emph{within the four-Hessian selection experiment}; the corresponding
objective target for the subsequent twelve-method comparison uses all twelve
methods, as defined above.

The objective-success totals are $188$, $187$, $192$, and $192$ for
\texttt{WCI}, \texttt{LM1}, \texttt{LM2}, and \texttt{LM3}, respectively.
All four reach the reference-witness target on $170/192$ cases, whereas
the geometric M-target counts are $94$, $94$, $94$, and $95$, respectively. Their
median evaluation costs are $643$, $640.5$, $636$, and $646.5$, and
their mean costs are $3625.74$, $3622.52$, $3624.09$, and $3625.42$,
respectively. Thus the smallest median belongs to \texttt{LM2},
not \texttt{LM3}; no single rule minimizes every aggregate cost.

For selection, we compare the \emph{efficiency point} ${\tt pp}(1)$ of all
three profiles, rather than selecting a rule from the mean or median
cost alone. Figure~\ref{fig:num-hessian-profile} gives the objective
$r=1$ counts $104$, $115$, $125$, $138$, the geometric M counts
$54$, $58$, $57$, $72$, and the reference-witness counts
$94$, $109$, $113$, $124$, respectively, in the order
\texttt{WCI}, \texttt{LM1}, \texttt{LM2}, \texttt{LM3}. Therefore
\texttt{LM3} has the highest ${\tt pp}(1)$ in each of the \emph{three}
profiles, ties for the largest objective and reference success totals,
and has the highest geometric M success total. \emph{We select
	\texttt{LM3} for every \texttt{QNPD} configuration in the remaining
	experiments.} Because the curves cross away from $r=1$, \texttt{LM3} is selected empirically using ${\tt pp}(1)$; this does not imply uniform dominance. The subsequent comparison reuses the same $192$
instances, so it is exploratory and does not
independently validate the Hessian choice. In particular, the reported selection performance is affected by choosing the diagonal rule on the evaluation instances; a held-out collection is needed to assess whether that choice generalizes.

\begin{figure}[htbp]
	\centering
	\definecolor{HessianWC}{RGB}{31,119,180}
	\definecolor{HessianLMone}{RGB}{255,127,14}
	\definecolor{HessianLMtwo}{RGB}{44,160,44}
	\definecolor{HessianLMthree}{RGB}{214,39,40}
	\resizebox{.97\linewidth}{!}{%
%
	}
	\caption{Three right-continuous endpoint performance profiles for the
		four diagonal rules of~\eqref{eq:four-diagonal-proposals}: (\textcolor{blue}{left}) final
		objective quality relative to the best of four, (middle) geometric
		M-stationarity, and (\textcolor{blue}{right}) reference-witness certification. Total
		$\mathrm{nf2g}$ through the common terminal audit is used for successful
		runs; failures receive infinite ratios. All $192$ instances remain in
		the denominator. Distinct line types and markers identify the rules;
		the shared legend is below the three panels.}
	\label{fig:num-hessian-profile}
\end{figure}

\subsection{Overall numerical comparison}\label{subsec:num-results}
The complete matched results are shown in Figure~\ref{fig:num-all-methods}.
The three standalone methods perform support-based first-order searches,
whereas the remaining nine methods retain optional diagonal quasi-Newton
penalty-decomposition trials. All terminal residuals are measured on each
method's audited feasible output; this audit is not part of the standalone
methods' optimization iterations.


\begin{figure}[H]
	\centering
	\definecolor{PDQNoverallGray}{RGB}{70,70,70}
	\definecolor{PDQNoverallWorkGray}{RGB}{110,110,110}
	\definecolor{PDQNoverallTrialGray}{RGB}{150,150,150}
	\def\HeatMax{40}
	\resizebox{.98\linewidth}{!}{%
		\begin{tikzpicture}[x=1cm,y=.38cm,font=\small]
			\foreach \hmr/\hmmethod/\hme/\hmm/\hmwork/\hmworktwo/\hmacc/\hmrej/\hmskip/\hmapp in {
				1/{BFS}/100/100/199/398/0/0/0/0,
				2/{PSS}/144/149/3423/6846/0/0/0/0,
				3/{ZCWS}/119/109/19855/39710/0/0/0/0,
				4/{QNPD}/170/95/646.5/1293/185/35853/0/1,
				5/{QNPD+FISTA}/174/108/225.5/451/192/26491/0/1,
				6/{QNPD+FISTA+OS}/181/109/212/424/183/14628/19070/1,
				7/{QNPD+FISTA+periodic5}/183/110/174/348/77/7737/30995/1,
				8/{QNPD+FISTA+slow-rho}/174/108/221/442/541/26026/0/1,
				9/{QNPD+FISTA+fixed-rho}/174/109/221/442/1287/23490/0/1,
				10/{QNPD+FISTA+OS+slow-rho}/181/109/211/422/533/14579/18547/1,
				11/{QNPD+FISTA+OS+fixed-rho}/181/110/211/422/1280/14066/16924/1,
				12/{QNPD+FISTA+OS+adaptive-rho}/180/112/207/414/1429/13688/16253/1
			}{
				\node[anchor=east,font=\footnotesize\ttfamily]
				at (-.15,.5-\hmr) {\hmmethod};
				
				\foreach \hmx/\hmv in {0/\hme,.82/\hmm}{
					\edef\hmshade{\the\numexpr\HeatMax*\hmv/192\relax}
					\path[fill=PDQNoverallGray!\hmshade!white,
					draw=white,line width=.6pt]
					(\hmx,1-\hmr) rectangle (\hmx+.82,-\hmr);
					\node[font=\small\bfseries,text=black]
					at (\hmx+.41,.5-\hmr) {\hmv};
				}
				
				\edef\hmshade{\the\numexpr\HeatMax*\hmworktwo/40000\relax}
				\path[fill=PDQNoverallWorkGray!\hmshade!white,
				draw=white,line width=.6pt]
				(2.24,1-\hmr) rectangle (3.64,-\hmr);
				\node[font=\small\bfseries,text=black]
				at (2.94,.5-\hmr) {$\hmwork$};
				
				\foreach \hmx/\hmv in {4.24/\hmacc,5.59/\hmrej,6.94/\hmskip}{
					\ifnum\hmapp=1
					\edef\hmshade{\the\numexpr\HeatMax*\hmv/40000\relax}
					\path[fill=PDQNoverallTrialGray!\hmshade!white,
					draw=white,line width=.6pt]
					(\hmx,1-\hmr) rectangle (\hmx+1.35,-\hmr);
					\node[font=\small\bfseries,text=black]
					at (\hmx+.675,.5-\hmr) {\hmv};
					\else
					\path[fill=black!5,draw=white,line width=.6pt]
					(\hmx,1-\hmr) rectangle (\hmx+1.35,-\hmr);
					\node[font=\small\bfseries,text=black]
					at (\hmx+.675,.5-\hmr) {\textemdash};
					\fi
				}
			}
			
			\draw[black!45,line width=.4pt] (0,0) rectangle (1.64,-12);
			\draw[black!45,line width=.4pt] (2.24,0) rectangle (3.64,-12);
			\draw[black!45,line width=.4pt] (4.24,0) rectangle (8.29,-12);
			\draw[black!35,line width=.55pt] (-.15,-3)--(8.29,-3);
			\node[anchor=east,font=\footnotesize\bfseries] at (-.15,.55) {Method};
			\node[font=\small\bfseries] at (.82,1.55) {Targets};
			\node[font=\small\bfseries] at (2.94,1.55) {Median work};
			\node[font=\small\bfseries] at (6.265,1.55) {Trial totals};
			\node[font=\footnotesize] at (.41,.55) {$E$};
			\node[font=\footnotesize] at (1.23,.55) {$M$};
			\node[font=\footnotesize] at (2.94,.55) {$W$};
			\foreach \hmx/\hmlabel in {4.915/Accepted,6.265/Rejected,7.615/Skipped}
			\node[font=\scriptsize] at (\hmx,.55) {\hmlabel};
			
			\shade[left color=white,right color=PDQNoverallGray!\HeatMax!white]
			(0,-13.35) rectangle (1.64,-12.95);
			\shade[left color=white,right color=PDQNoverallWorkGray!\HeatMax!white]
			(2.24,-13.35) rectangle (3.64,-12.95);
			\shade[left color=white,right color=PDQNoverallTrialGray!\HeatMax!white]
			(4.24,-13.35) rectangle (8.29,-12.95);
			\foreach \hmleft/\hmright in {0/1.64,2.24/3.64,4.24/8.29}
			\draw[black!45,line width=.3pt]
			(\hmleft,-13.35) rectangle (\hmright,-12.95);
			\foreach \hmx/\hmlabel in {
				0/0,.82/96,1.64/192,
				4.24/0,5.2525/10000,6.265/20000,
				7.2775/30000,8.29/40000}{
				\draw[black!55] (\hmx,-13.35)--(\hmx,-13.48);
				\node[anchor=north,font=\scriptsize]
				at (\hmx,-13.55) {\hmlabel};
			}
			\draw[black!55] (2.24,-13.35)--(2.24,-13.48);
			\draw[black!55] (3.64,-13.35)--(3.64,-13.48);
			\node[anchor=north west,font=\scriptsize]
			at (2.24,-13.55) {0};
			\node[anchor=north east,font=\scriptsize]
			at (3.64,-13.55) {20000};
		\end{tikzpicture}%
	}
	\caption{Twelve methods on 192 matched instances. $E$ and $M$ count
		terminal reference-witness and geometric M-target successes; $W$ is median
		$\mathrm{nf2g}$. The three groups use separate grayscale scales.
		Dashes denote inapplicable quasi-Newton trial statistics.}
	\label{fig:num-all-methods}
\end{figure}

The \texttt{QNPD+FISTA+periodic5} configuration reaches
183/192 reference-witness targets at median $\mathrm{nf2g}=174$.
The corresponding results are 100/192 and 199 for
\texttt{BFS}, 144/192 and 3423 for \texttt{PSS},
and 119/192 and 19855 for \texttt{ZCWS}.
\texttt{PSS} reaches more geometric M-targets
($149/192$ versus $110/192$). Against \texttt{BFS}, the periodic
variant records $82$ audited objective wins, $61$ losses and $49$ ties.
The relative performance therefore depends on the reported target.

\begin{figure}[H]
	\centering
	\definecolor{EPBlue}{RGB}{31,119,180}
	\definecolor{EPOrange}{RGB}{255,127,14}
	\definecolor{EPGreen}{RGB}{44,160,44}
	\definecolor{EPRed}{RGB}{214,39,40}
	\definecolor{EPPurple}{RGB}{148,103,189}
	\definecolor{EPBrown}{RGB}{140,86,75}
	\definecolor{EPPink}{RGB}{227,119,194}
	\definecolor{EPGray}{RGB}{127,127,127}
	\definecolor{EPOlive}{RGB}{188,189,34}
	\definecolor{EPCyan}{RGB}{23,190,207}
	\resizebox{.96\linewidth}{!}{%
%
	}
	\caption{Three right-continuous endpoint-work profiles for all twelve
		methods: audited objective target, geometric M-residual and reference
		witness. A successful method is charged its completed-run $\mathrm{nf2g}$;
		unsuccessful runs contribute infinite performance ratios. The denominator
		is 192 throughout. The legend is below the panels.}
	\label{fig:num-endpoint-profiles}
\end{figure}

Figure~\ref{fig:num-endpoint-profiles} compares efficiency and robustness
separately for the three numerical targets. A larger profile value at
$r=1$ means more instances solved with the smallest reported terminal
cost for that target; as $r$ grows, the curves show success coverage
within the corresponding relative-cost factor. These are endpoint profiles,
not first-hit or iterate-wise convergence histories.

Figure~\ref{fig:num-gap-work} complements the aggregated endpoint profiles
with an instance-level view of terminal objective quality versus evaluation
work. Very small or zero retrospective objective gaps occur over a broad
range of completed-run $\mathrm{nf2g}$, while some substantially more expensive
runs retain larger gaps. Hence, increased evaluation work does not by itself
imply a better terminal objective value. The scatter also emphasizes that
these quantities describe audited terminal outcomes rather than within-run
convergence trajectories.

\begin{figure}[H]
	\centering

	\definecolor{GapBlue}{RGB}{31,119,180}
	\definecolor{GapOrange}{RGB}{255,127,14}
	\definecolor{GapGreen}{RGB}{44,160,44}
	\definecolor{GapRed}{RGB}{214,39,40}
	\definecolor{GapPurple}{RGB}{148,103,189}
	\definecolor{GapBrown}{RGB}{140,86,75}
	\definecolor{GapPink}{RGB}{227,119,194}
	\definecolor{GapGray}{RGB}{127,127,127}
	\definecolor{GapOlive}{RGB}{188,189,34}
	\definecolor{GapCyan}{RGB}{23,190,207}

	\resizebox{.94\linewidth}{!}{%
%
	}
	\caption{Retrospective audited terminal objective gap versus completed-run
		$\mathrm{nf2g}$ for the twelve methods. Each marker is a terminal feasible
		outcome, not a within-run objective trajectory or a certified optimality gap.}
	\label{fig:num-gap-work}
\end{figure}

Figure~\ref{fig:num-stationarity-work} shows the corresponding
work-residual relation for the five representative methods. For the
geometric M-target, \texttt{PSS} reaches $149/192$ cases, compared with
$110/192$ for \texttt{QNPD+FISTA+periodic5}, but does so with much larger
median completed-run work. For the reference-witness target, the ordering
changes: \texttt{QNPD+FISTA+periodic5} reaches $183/192$ cases, compared
with $170/192$ for \texttt{QNPD}, $100/192$ for \texttt{BFS}, $144/192$
for \texttt{PSS}, and $119/192$ for \texttt{ZCWS}. Together with
Figure~\ref{fig:num-all-methods}, this confirms that geometric
M-stationarity and reference-witness certification measure distinct
aspects of the terminal result and need not favor the same method.

\begin{figure}[H]
	\centering

	\definecolor{GapBlue}{RGB}{31,119,180}
	\definecolor{GapOrange}{RGB}{255,127,14}
	\definecolor{GapGreen}{RGB}{44,160,44}
	\definecolor{GapRed}{RGB}{214,39,40}
	\definecolor{GapPurple}{RGB}{148,103,189}
	\definecolor{GapBrown}{RGB}{140,86,75}
	\definecolor{GapPink}{RGB}{227,119,194}
	\definecolor{GapGray}{RGB}{127,127,127}
	\definecolor{GapOlive}{RGB}{188,189,34}
	\definecolor{GapCyan}{RGB}{23,190,207}

	\resizebox{.90\linewidth}{!}{%
%
	}
	\caption{Terminal geometric M residual (top) and reference-witness norm
		(bottom) versus completed-run $\mathrm{nf2g}$ for five representative methods,
		including all three standalone comparators. Each marker represents one
		terminal run; the dashed horizontal line denotes the reporting target.}
	\label{fig:num-stationarity-work}
\end{figure}

\subsection{Results by family and paired comparisons}\label{subsec:num-family}
Figure~\ref{fig:num-family-heatmap} reports reference-witness ($E$) and
geometric M-target ($M$) successes for four representative methods,
with 24 paired instances per family. The two panels use the same
0--24 color scale.

Figure~\ref{fig:num-paired} compares audited final objectives with
\texttt{BFS}. W/L/T denotes wins/losses/ties, with objective ties using
the relative tolerance
$10^{-8}\max\{1,|f_{i,a}|,|f_{i,\mathrm{BFS}}|\}$.
The work column gives the median of the paired differences in
$\mathrm{nf2g}$ (method minus \texttt{BFS}); negative values indicate less
work. The counts $E^+/E^-$ give reference-target successes gained/lost
relative to \texttt{BFS}.

\begin{figure}[H]
	\centering
	\definecolor{PDQNfamilyGray}{RGB}{70,70,70}
	\def\HeatMax{40}
	\resizebox{.98\linewidth}{!}{%
		\begin{tikzpicture}[x=1.16cm,y=.50cm,font=\small]
			\foreach \hmr/\hmfamily/\hmea/\hmma/\hmeb/\hmmb/\hmec/\hmmc/\hmed/\hmmd in {
				1/{PSD quadratic}/5/5/22/22/12/11/24/18,
				2/{Mean-variance}/23/23/14/14/24/24/24/24,
				3/{Least-squares}/9/9/20/20/11/8/24/10,
				4/{Logistic}/3/4/15/10/4/4/15/4,
				5/{Sparse PCA}/23/24/20/20/22/23/24/23,
				6/{Indefinite quadratic}/23/24/13/24/23/24/24/24,
				7/{Smoothed phase}/9/6/24/24/13/7/24/4,
				8/{Control regression}/5/5/16/15/10/8/24/3
			}{
				\node[anchor=east] at (-.10,.5-\hmr) {\hmfamily};
				\foreach \hmx/\hmv in {
					0/\hmea,1/\hmeb,2/\hmec,3/\hmed,
					4.5/\hmma,5.5/\hmmb,6.5/\hmmc,7.5/\hmmd
				}{
					\pgfmathsetmacro{\hmshade}{\HeatMax*\hmv/24}
					\path[
						fill=PDQNfamilyGray!\hmshade!white,
						draw=white,
						line width=.6pt
					]
					(\hmx,1-\hmr) rectangle (\hmx+1,-\hmr);
					\node[font=\small\bfseries,text=black]
					at (\hmx+.5,.5-\hmr) {\hmv};
				}
			}
			\foreach \hmoffset in {0,4.5}{
				\draw[black!45,line width=.4pt]
				(\hmoffset,0) rectangle (\hmoffset+4,-8);
				\foreach \hmx/\hmlabel in {.5/BFS,1.5/PSS,2.5/ZCWS}
				\node[anchor=north,font=\scriptsize\ttfamily]
				at (\hmoffset+\hmx,-8.2) {\hmlabel};
				\node[anchor=north,align=center,font=\scriptsize\ttfamily]
				at (\hmoffset+3.5,-8.2) {QNPD+FISTA\\+periodic5};
			}
			\node[font=\small\bfseries] at (2,.55)
			{(a) Reference witness $E$};
			\node[font=\small\bfseries] at (6.5,.55)
			{(b) Geometric $M$-target};
			
			\shade[
				top color=PDQNfamilyGray!\HeatMax!white,
				bottom color=white
			]
			(8.88,-8) rectangle (9.06,0);
			\draw[black!45,line width=.4pt] (8.88,-8) rectangle (9.06,0);
			\foreach \hmv in {0,6,12,18,24}{
				\pgfmathsetmacro{\hmy}{-8+8*\hmv/24}
				\draw[black!55] (9.06,\hmy)--(9.13,\hmy);
				\node[anchor=west,font=\footnotesize] at (9.16,\hmy) {\hmv};
			}
			\node[rotate=90,font=\footnotesize] at (9.72,-4) {Successes /24};
		\end{tikzpicture}%
	}
	\caption{Family-level target successes for four representative methods
		(24 paired instances per family): reference witness $E$ (left) and
		geometric M-target $M$ (right). Both panels use the same 0--24 scale.}
	\label{fig:num-family-heatmap}
\end{figure}

\begin{figure}[H]
	\centering
	\definecolor{PDQNpairedGray}{RGB}{70,70,70}
	\definecolor{PDQNworkLessGray}{RGB}{150,150,150} 
	\definecolor{PDQNworkMoreGray}{RGB}{220,220,220} 
	\def\HeatMax{40}
	\resizebox{.8\linewidth}{!}{%
		\begin{tikzpicture}[x=1cm,y=.43cm,font=\small]
			\foreach \hmr/\hmmethod/\hmw/\hml/\hmt/\hmwork/\hmep/\hmem/\hmsign in {
				1/{PSS}/96/62/34/3180/76/32/1,
				2/{ZCWS}/103/29/60/18920.5/25/6/1,
				3/{QNPD}/77/98/17/241/71/1/1,
				4/{QNPD+FISTA}/82/61/49/-14/74/0/-1,
				5/{QNPD+FISTA+OS}/82/61/49/-20.5/81/0/-1,
				6/{QNPD+FISTA+periodic5}/82/61/49/-34.5/83/0/-1,
				7/{QNPD+FISTA+slow-rho}/82/61/49/-16/74/0/-1,
				8/{QNPD+FISTA+fixed-rho}/82/61/49/-16/74/0/-1,
				9/{QNPD+FISTA+OS+slow-rho}/82/61/49/-20.5/81/0/-1,
				10/{QNPD+FISTA+OS+fixed-rho}/82/61/49/-18/81/0/-1,
				11/{QNPD+FISTA+OS+adaptive-rho}/82/61/49/-18/81/1/-1
			}{
				\node[anchor=east,font=\footnotesize\ttfamily]
				at (-.15,.5-\hmr) {\hmmethod};

				\foreach \hmx/\hmv in {0/\hmw,.84/\hml,1.68/\hmt,5.32/\hmep,6.20/\hmem}{
					\pgfmathsetmacro{\hmshade}{\HeatMax*\hmv/192}
					\path[
						fill=PDQNpairedGray!\hmshade!white,
						draw=white,
						line width=.6pt
					]
					(\hmx,1-\hmr) rectangle (\hmx+.84,-\hmr);
					\node[font=\small\bfseries,text=black]
					at (\hmx+.42,.5-\hmr) {\hmv};
				}

				\ifnum\hmsign<0
					\def\hmworkfill{PDQNworkLessGray}
				\else
					\ifnum\hmsign>0
						\def\hmworkfill{PDQNworkMoreGray}
					\else
						\def\hmworkfill{white}
					\fi
				\fi

				\path[fill=\hmworkfill,draw=white,line width=.6pt]
				(2.92,1-\hmr) rectangle (4.92,-\hmr);
				\node[font=\small\bfseries,text=black]
				at (3.92,.5-\hmr) {$\hmwork$};
			}
			
			\draw[black!45,line width=.4pt] (0,0) rectangle (2.52,-11);
			\draw[black!45,line width=.4pt] (2.92,0) rectangle (4.92,-11);
			\draw[black!45,line width=.4pt] (5.32,0) rectangle (7.04,-11);

			\node[anchor=east,font=\footnotesize\bfseries] at (-.15,.55) {Method};
			\node[font=\small\bfseries] at (1.26,1.45) {Objective};
			\node[font=\small\bfseries] at (3.92,1.45) {Paired median};
			\node[font=\small\bfseries] at (6.18,1.45) {Targets};

			\foreach \hmx/\hmlabel in {.42/W,1.26/L,2.10/T}
				\node[font=\footnotesize\bfseries] at (\hmx,.55) {\hmlabel};

			\node[font=\footnotesize] at (3.92,.55) {$\Delta(\mathrm{nf2g})$};
			\node[font=\footnotesize] at (5.74,.55) {$E^+$};
			\node[font=\footnotesize] at (6.62,.55) {$E^-$};
			
			\node[font=\scriptsize] at (1.26,-11.5) {Counts /192};

			\shade[left color=white,right color=PDQNpairedGray!\HeatMax!white]
			(0,-12.35) rectangle (2.52,-11.95);
			\draw[black!45,line width=.3pt] (0,-12.35) rectangle (2.52,-11.95);

			\foreach \hmv in {0,48,96,144,192}{
				\pgfmathsetmacro{\hmx}{2.52*\hmv/192}
				\draw[black!55] (\hmx,-12.35)--(\hmx,-12.48);
				\node[anchor=north,font=\scriptsize] at (\hmx,-12.55) {\hmv};
			}

			\path[fill=PDQNworkLessGray,draw=black!30]
			(2.92,-12.25) rectangle (3.12,-11.85);
			\node[anchor=west,font=\scriptsize] at (3.20,-12.05) {Less work};

			\path[fill=PDQNworkMoreGray,draw=black!30]
			(5.32,-12.25) rectangle (5.52,-11.85);
			\node[anchor=west,font=\scriptsize] at (5.60,-12.05) {More work};
		\end{tikzpicture}%
	}
	\caption{Paired audited comparisons with \texttt{BFS} on 192 instances.
		W/L/T count objective wins, losses and ties, and $E^+/E^-$ count reference-target
		successes gained/lost. Counts use a common 0--192 scale; the median paired
		work difference (method minus \texttt{BFS}) is shown in two gray levels:
		darker gray for negative values and lighter gray for positive values.}
	\label{fig:num-paired}
\end{figure}

\subsection{Interpretation and limitations}\label{subsec:num-limitations}
The results compare support-based first-order search with nine safeguarded
quasi-Newton configurations on matched synthetic instances. The
\texttt{PSS} and \texttt{ZCWS} procedures use projected support-restricted
first-order subproblems and source-consistent support-growth and exchange
rules. Their supportwise refinements, indexing, and stopping rules use the
actual derivatives, valid support indices, and the stated finite
budgets. The reported results concern these Python implementations and
do not imply equivalence with other implementations or general solver
dominance. Stationarity and objective-quality performance are distinct. The common terminal reference audit does not establish convergence of the standalone methods, and the absence of a pure reference-only baseline prevents isolation of the incremental contribution of the optional quasi-Newton trials. Moreover, comparing audited outputs conflates each solver's contribution with the common final reference step. The four-rule selection and the subsequent comparison reuse the same instances, so the results are exploratory and should not be treated as a held-out evaluation. The reported counts and paired medians are descriptive; no uncertainty intervals or multiple-comparison analysis are available. The synthetic families, dimensions, and chosen parameter schedules limit the scope of the findings. The $\mathrm{nf2g}$ ledger omits important solver-dependent costs, including potentially expensive sparse projections and internal convex-slice operations, and so cannot by itself establish a wall-clock or total-work advantage.

The present source package does not contain the code, exact generated instances, per-run data, or plotting scripts. The numerical values and figures are therefore transcribed as reported rather than independently verified here. Before a reproducibility claim can be made, these artifacts should be supplied with software and hardware specifications, dependencies, seeds, complete solver and stopping settings, and commands for regenerating every table and figure. A stronger study would include a reference-only baseline, first-hit work and elapsed time, projection and slice-solve counts, unaudited as well as audited endpoints, matched stopping rules, independent test instances, and a clear identification of any deviations from the published comparator algorithms. The global budget is $\mathrm{nf2g}\le20000$, with no finite outer-iteration ceiling. Numerical tolerances and projection calculations do not replace the exact-oracle assumptions of Section~\ref{convergence}.

\section{Conclusion}\label{conclusions}
An exact feasible reference makes finite-budget penalty-decomposition trials compatible with well-defined outer progress even when their consensus or filter tests fail. The comparison envelope controls both reference and accepted steps. Under local smoothness and a compact initial feasible level set, every accumulation point satisfies a common-stepsize projection condition; the returned-point witness provides universal approximate BF and a qualified geometric M bound with $O(\epsilon^{-2})$ reference complexity. Under the supportwise normal-cone sum rule, BF also implies all-super-support S-stationarity. The conditional core results and examples show where consensus, outer regularity, full support, and radius inactivity enter the analysis. Exact projection remains computationally significant, and we do not claim acceleration or global optimality. The guarantees apply even if all optional trials fail; they do not establish that those trials improve performance over the reference method. In the exploratory 192-instance Python study, the periodic-trial \texttt{QNPD} variant attains 183 reference-witness targets with median $\mathrm{nf2g}=174$ and occasionally improves feasible objectives; the three standalone comparators are included in the matched benchmark. These exploratory numerical findings do not establish general acceleration; the reported results are not independently reproducible from the present source package. Future work includes controlled inexact reference projections, release of the full experimental artifacts, independently selected test instances, implementation-independent benchmarking, and experiments on real data.

\begin{appendices}
	\section{Controlled inexact block solves}\label{sec:inexact-extension}
We replace the exact solve in Step~\PDstep{2}{a.EPD-inner} by controlled
errors, while retaining the exact projections and model-descent gate in
Steps~\PDstep{3}{a.EPD-inner}--\PDstep{4}{a.EPD-inner}. The same analysis
applies to a fixed trial call in Step~\PDstep{4}{a.CPDQN}.

Inexact block minimization is central to the penalty-decomposition
analysis in \cite[Section~4]{KanzowLapucci}. Here we consider a different
form of inexactness: square-summable errors in the frozen strongly convex
quadratic solves, with the selected-slice projection and model-descent
test unchanged.

	\begin{proposition}[Inner descent with square-summable solve errors]\label{prop:inexact-inner}
		Freeze $z,\mu,{\tt pp},H$ as in Theorem~\ref{thm:BCD_convergence}, put $A=H+\rho I\succeq m I$ with $m>0$, and suppose the inner iterates satisfy
		\[
		e_\ell:=\nabla_x\Phi(x_\ell,y_{\ell-1}),\qquad
		\sum_{\ell\ge1}\|e_\ell\|^2<\infty.
		\]
		Let each $y$ step be an exact selected-slice projection in the fixed compact domain $C_\zeta$, satisfying~\eqref{eq:inner-support-descent}, with the preceding support retained on rejection. Then the sequence is bounded, its consecutive block differences vanish and its final residual $\nabla_x\Phi(x_\ell,y_\ell)$ tends to zero. Every cluster pair $(x^*,y^*)$ has an extracted size-$s$ support $\mathcal L$ such that
		\[
		x^*\in\operatorname*{argmin}_{x\in\mathbb R^n}\Phi(x,y^*),\qquad
		y^*\in\operatorname*{argmin}_{y\in C_\zeta\cap S_{\mathcal L}}\Phi(x^*,y).
		\]
		For a fixed safeguarded trial, the same assertion holds with its fixed call radius $R_k$ in place of $\zeta$; no radius is changed during that inner call.
	\end{proposition}
	\begin{proof}
		For $\Delta x_\ell=x_\ell-x_{\ell-1}$, the quadratic identity and Young's inequality give
		\[
		\Phi(x_{\ell-1},y_{\ell-1})-\Phi(x_\ell,y_{\ell-1})
		=\tfrac12\Delta x_\ell^TA\Delta x_\ell-e_\ell^T\Delta x_\ell
		\ge\tfrac m4\|\Delta x_\ell\|^2-\|e_\ell\|^2/m.
		\]
		The projected step cannot increase $\Phi$. Hence
		$V_\ell=\Phi(x_\ell,y_\ell)+m^{-1}\sum_{i>\ell}\|e_i\|^2$ satisfies
		$V_{\ell-1}-V_\ell\ge m\|\Delta x_\ell\|^2/4$ and has the frozen quadratic's common finite lower bound on the compact projected domain. Summation gives square-summable $x$ differences.
		The inexact equation $Ax_\ell=H z-g(z)+\mu+\rho y_{\ell-1}+e_\ell$ directly bounds the sequence. Subtraction at successive steps gives
		$A(x_{\ell+1}-x_\ell)=\rho(y_\ell-y_{\ell-1})+e_{\ell+1}-e_\ell$.
		Since $e_\ell\to0$, $y$ differences vanish, and the final residual is $e_\ell+\rho(y_{\ell-1}-y_\ell)\to0$. Extract a constant-support cluster subsequence. Vanishing consecutive differences and continuity let the exact projected comparison inequalities pass to the limit; the vanishing $x$ residual gives its block minimum. This argument applies to any fixed call radius.
	\end{proof}
	The core identity~\eqref{eq:residual-multiplier-identity} uses the measured final residual algebraically, even for an inexact solve. Thus, the accepted-consensus bounds and the AM proof of Theorem~\ref{thm:outer-original-CCM} extend when the residual and consensus tests in Step~\PDstep{5}{a.EPD-inner}, exact final projection, full support and outer regularity are verified. Recovery in Step~\PDstep{5}{a.CPDQN} retains its separate reset. Square-summable errors justify the stated infinite inner theorem; any finite implementation that invokes it must expose and enforce those errors.

\end{appendices}

\subsection*{Acknowledgements}
The authors thank Nadav Hallak for the MATLAB implementations described in~\cite{Beck2016}.

\subsection*{Funding}
The second author acknowledges financial support from the Austrian Science Fund (FWF) under \url{https://doi.org/10.55776/PAT2747625}. 

\bfi{Declaration of Generative AI and AI-assisted Technologies in the Writing Process.} During the preparation of this work, the authors used ChatGPT (OpenAI) to improve the language and readability of the manuscript. After using this tool, the authors carefully reviewed and edited the content as necessary and take full responsibility for the final content of the publication.


\end{document}